\documentclass[11pt]{article}
\title{Reducing Marketplace Interference Bias Via Shadow Prices}
\author{Ido Bright$^1$, Arthur Delarue$^2$, Ilan Lobel$^3$}
\date{\footnotesize$^1$Lyft, $^2$Stewart School of Industrial and Systems Engineering, Georgia Institute of Technology, $^3$NYU Stern School of Business}
\usepackage{natbib,graphicx,fullpage,amsmath,amsfonts,amsthm,amssymb,setspace, natbib, booktabs}
\usepackage{subcaption, bbm}
\usepackage{wrapfig}
\usepackage{subfiles}
\usepackage[dvipsnames]{xcolor}
\usepackage{tikz}
\usetikzlibrary{arrows,intersections,calc,patterns,decorations.pathreplacing}
\usepackage[shortlabels]{enumitem}
\usepackage{wrapfig}
\usepackage{placeins}
\usepackage{floatrow}
\usepackage{sidecap}

\newtheorem{theorem}{Theorem}
\newtheorem{corollary}{Corollary}
\newtheorem{proposition}{Proposition}
\newtheorem{definition}{Definition}
\newtheorem{assumption}{Assumption}
\newtheorem{lemma}{Lemma}
\newtheorem*{thm:main}{Theorem \ref{thm:main} (restated)}

\newcommand{\control}{\text{control}}
\newcommand{\treatment}{\text{treatment}}
\newcommand{\experiment}{\text{experiment}}
\newcommand{\treatmentemph}{\text{\emph{treatment}}}
\newcommand{\controlemph}{\text{\emph{control}}}
\newcommand{\experimentemph}{\text{\emph{experiment}}}
\newcommand{\abs}[1]{\left|#1\right|}
\newcommand{\norm}[1]{\left\lVert#1\right\rVert}
\newcommand{\Var}[1]{\mathrm{Var}\left[#1\right]}

\newcommand{\revision}[1]{{\color{black}#1}}

\usepackage{bm}

\begin{document}
\maketitle

\begin{abstract}
Marketplace companies rely heavily on experimentation when making changes to the design or operation of their platforms. The workhorse of experimentation is the randomized controlled trial (RCT), or A/B test, in which users are randomly assigned to treatment or control groups.  However, marketplace interference causes the Stable Unit Treatment Value Assumption (SUTVA) to be violated, leading to bias in the standard RCT metric. In this work, we propose {techniques} for platforms to run standard RCTs and still obtain meaningful estimates despite the presence of marketplace interference. We specifically consider a {generalized} matching setting, in which the platform explicitly matches supply with demand via a linear programming algorithm. Our first proposal is for the platform to estimate the value of global treatment and global control via optimization. We prove that this approach is unbiased in the fluid limit. Our second proposal is to compare the average shadow price of the treatment and control groups rather than the total value accrued by each group. We prove that {this} technique corresponds to the correct first-order approximation (in a Taylor series sense) of the value function of interest {even in a finite-size system}. We then use this result to prove that, under reasonable assumptions, our estimator is less biased than the RCT estimator. At the heart of our result is the idea that it is relatively easy to model interference in matching-driven marketplaces since, in such markets, the platform mediates the spillover.
\end{abstract}

\section{Introduction}\label{sec:intro}

Modern technology companies use experiments to make \revision{many product} decisions \revision{both} minor (tweaking parameters) \revision{and} major (shipping new features). \revision{A} fundamental challenge in \revision{platform} experimentation is marketplace interference. For instance, consider a ride-hailing platform experimenting with a demand-side price discount. The platform performs a randomized control trial (RCT), or A/B test, where some demand units are offered the \revision{discount} while others are \revision{not}. Because the treated units are more likely to book rides as a result of the discount, they reduce the total supply available to all demand-side units, including the control units. This interference between treatment and control units causes the Stable Unit Treatment Value Assumption (SUTVA) to fail, and induces bias in the standard estimator used to \revision{measure} the value generated by the treatment. 

A standard answer to this problem is to replace the ``user-split'' design, where individual units are randomly assigned to treatment and control, with a ``time-split'' (or ``switchback'') design, where the entire market switches repeatedly between treatment and control. However, this design is \revision{in many cases} infeasible. Consider, for example, a user interface change that a user finds jarring or confusing at first, but eventually prefers to the original design after repeated exposure. A user-split experiment could easily capture such an effect, and analysts could even discard data from users' adjustment period at the start of the experiment. Meanwhile, a time-split design would repeatedly  alternate between user interfaces over the duration of the experiment, a clearly undesirable outcome.

Instead of modifying the experimental design, \revision{in this paper} we propose a method for directly removing most interference bias from user-split experiments. Most existing works assume that removing bias from a marketplace user-split experiment is a very difficult task. Often, these papers consider experiments in platforms like Airbnb or Amazon where interference is caused by customer choice, which is difficult to model correctly \revision{and tractably}. In contrast, we consider generalized matching marketplaces, where the platform assigns supply units to demand units via a matching or network flow algorithm. Examples of such settings include ride-hailing, delivery, and supply chains. We argue that because such generalized matching platforms control the flow of units through the system, they have a great deal of knowledge about the interference structure. They can leverage this \revision{information} to remove most of the marketplace interference bias from user-split experiments. 

We consider a model with a finite number of supply and demand types. Each type has a given Poisson arrival rate under the control policy. \revision{We consider a demand-side treatment which affects these demand arrival rates.} We do not assume the platform knows any of these arrival rates.
The platform has a matching value $v_{i,j}$ for assigning a demand unit of type $i$ to a supply unit of type $j$.
After collecting  demand and supply requests for a certain amount of time, the platform runs a \revision{matching} cycle where it assigns supply units to demand units via a network flow linear program (LP). The platform's goal is to quantify the global treatment effect, i.e., the difference in value between the setting where all units are treated and the setting where no units are treated. Because these settings are not simultaneously observable, the platform \revision{must use} data from an experiment in which a fraction of users (treatment group) are exposed to treatment, while the remaining users (control group) are not.

The standard approach to this estimation problem is what we call the \emph{RCT} estimator, 
which estimates the global treatment effect by comparing the total value obtained by treatment units with the total value obtained by control units. Our first main result is that, under a fairly broad condition, the RCT estimator systematically overestimates the global treatment effect. We show this result by taking the system to its fluid limit and showing, first, that the RCT estimator corresponds to a linear approximation of the value function, and second, that this linear approximation is of poor quality in the presence of interference. The condition is that the treatment effect is sign-consistent across types, which means that if the effect is positive for one type, then it is positive for all types. This condition is used to rule out examples where the RCT estimator overestimates the amplitude of both a positive treatment effect on one type and a negative effect on a different type in such a way that the total bias is zero. This would be an example of the standard estimator performing well due to luck rather than good design.

The fluid limit of the value function suggests a new potential approach to the problem: estimate the system arrival rate under both global treatment and global control, solve \revision{two} LPs to compute the value function under both global treatment and global control, and estimate the global treatment effect \revision{as} the difference between the value of these two LPs. We call this \revision{second} approach the \emph{Two-LP} estimator. Our second main result is to prove that this estimator is unbiased in the fluid limit. However, this solution is not perfect: the estimate can be biased in a finite-sized system.

This motivates a third estimator, which we call the \emph{shadow price} (SP) estimator. The SP estimator uses the shadow prices of the matching LP's constraints to calculate the value of the treatment. Intuitively, the shadow price of a demand constraint measures the incremental value of an additional unit of demand, and therefore captures the true value of increasing (or reducing) demand. Our next main result is that the SP estimator is less biased than the RCT estimator under a couple of conditions. The first condition is sign-consistency of the treatment effect. The second condition is that the experiment is symmetric: the control and treatment groups each represent 50\% of the total population. We can \revision{either} relax this symmetry assumption with a modified estimator \revision{based on a} convex combination of the SP and RCT estimators, \revision{or replace it with the practically reasonable assumption that the treatment is not disruptive enough to meaningfully alter the interference pattern in the marketplace.}

Like the RCT estimator, the SP estimator is based on a linear approximation of the matching value function, and therefore it \revision{can be} biased. However, unlike the RCT estimator, the SP estimator uses the ``correct'' linear approximation, in the sense that it is the first-order Taylor series approximation of the value function. Crucially, we prove that the SP estimator is the correct linear approximation even in a finite-sized system, not just in the fluid limit. Therefore, it is not difficult to construct finite-sized instances where the Two-LP estimator is a lot more biased than the SP estimator.

\revision{Though our technique is designed with the treatment effect on the platform's objective function in mind, \revision{we show} it can be extended to measure effects on secondary objectives. In practice, firms running experiments often} track a number of metrics \revision{beyond overall value}. \revision{We can use LP complemtary slackness to extend our shadow price technique to these secondary metrics.}

We validate our proposed method via two simulations. For the first simulation, we use publicly available data about New York City's taxi and ride-hailing ride requests over a given month. We construct a matching value metric based on efficiency defined as ride distance minus pickup distance. We simulate ride requests over a day and construct a matching cycle, then consider an experiment that increases ride request rates. We show that our proposed estimator greatly outperforms the RCT estimator and nearly matches the true global treatment effect. The second simulation captures a supply chain fulfillment experiment. Like the first simulation, it also demonstrates that the SP estimator is a great tool for reducing marketplace interference bias in experiments.

\section{Related Literature}

Interference in experiments is by now a well-known problem and an area of particular focus for recent literature, with empirical studies by \cite{blake2014marketplace} and
\cite{fradkin2019simulation} showing that interference effects are quite significant in real-world experiments respectively involving auctions and recommendation systems. The bulk of the literature on this topic focuses on modifying classical experimental designs to reduce interference. In contrast, our paper considers a standard RCT design and proposes a method to reduce bias without modifying the experiment design.

Early works on this topic focused on experiments over social networks rather than marketplaces. In social network settings, the key concern is that treating a user might affect their peers. The most common method to address this problem involves a bucketing design (also known as a subplot or subnetwork design), where nodes that form a closely-knit group in some network sense are given the same treatment assignment (see, for instance, \cite{ugander2013graph}). More recently, similar methods have been applied to estimate and reduce bias in marketplace experiments \citep{holtz2020reducing}. In the social network setting, \cite{chin2019regression} uses natural fluctuations in the randomized treatment assignment and constructs estimators that differentiate between users that have more neighbors in the control group (and can serve as a proxy for global control) and users that have more neighbors in the treatment group (and can serve as a proxy for global treatment).

Beyond bucketing, another \revision{widely used technique} is the switchback design or time-split, where an entire market is switched back and forth between treatment and control over time. This technique dates back all the way to \cite{cochran1939long} but it has recently gained significant traction. For instance, \cite{bojinov2022design} studies how to optimally design a switchback experiment. \cite{chamandy2016experimentation} and  \cite{sneider2019experiment} discuss the use of switchback experiments at Lyft and Doordash, respectively. \cite{chamandy2016experimentation}, in particular, discusses the tradeoffs between a bucketing-type approach and a switchback design. As discussed in Section \ref{sec:intro}, switchback designs are not feasible for experiments that involve significant changes in the user interface.

More recently, \cite{johari2022experimental} and \cite{bajari2021multiple} have proposed experiment designs where the random assignment to treatment and control is applied to both supply and demand units simultaneously. While \revision{our analysis shares similar aspects} (such as large-scale limits), neither framework applies easily to our problem of interest. Both \cite{johari2022experimental} and \cite{bajari2021multiple} implicitly or explicitly describe marketplaces where supply and demand units interact directly (rather than through a matching algorithm), so supply-side \revision{and} demand-side experiment\revision{s are interchangeable for most treatments}. \cite{bajari2021multiple} impose a local interference assumption which does not hold in our setting, while \cite{johari2022experimental} use a choice model to capture spillovers. Under a similar marketplace model, \cite{li2022interference} analyze the effect of the treatment proportion of supply-side and demand-side experiments on estimator variance and interference bias.

Another recent stream of work focuses on experiment designs that seek to measure the marginal effect of a treatment. \cite{wager2021experimenting} consider experiments in which users are exposed to price shocks large enough to measure an effect, but small enough to preserve marketplace equilibrium. The paper then proposes an SGD-like method to construct a sequence of experiments that converge to the correct estimates despite the presence of interference. In contrast, our paper proposes a method to linearly approximate the interference structure of a single experiment. Our linear approximation is constructed using gradients in a matching marketplace; in the different context of a price-based multi-commodity market, \cite{munro2021treatment} find that indirect treatment effects due to interference are related to direct treatment effects via utility gradients.
\section{The Model}

\subsection{The Platform}

We consider a two-sided matching platform that collects supply and demand requests over time. We assume the platform collects requests over a period of time that we call a matching cycle and then matches them at the end of the cycle with the goal of maximizing some reward function.

\paragraph{Supply and demand types.} Each unit of supply or demand belongs to one of a finite number of types. We index demand types by $i=1,\ldots,n_d$ and supply types by $j=1,\ldots,n_s$. Supply units of type $j$ arrive according to a Poisson process with parameter $\pi_j$, and we denote the vector of supply arrival rates by $\bm{\pi}=\left(\pi_1, \ldots, \pi_{n_s}\right)$. Demand intent units of type $i$ arrive according to a Poisson process with parameter $\tilde{\lambda}_i$. Upon arrival, a demand intent unit observes some information about the state of the marketplace (such as a price or expected pick-up time), and chooses to make a request with probability $p_i$, which might depend on the marketplace state. Request realizations are assumed to be independent of each other and of the Poisson arrival processes. If a demand intent unit does not make a request, it leaves the marketplace forever. We will use the term demand units to refer to units that have submitted requests, as opposed to demand intent units, which might or might not have submitted requests. The arrival process of demand units is therefore a thinning of the demand intent arrival process and is a Poisson process itself with arrival rate $\lambda_i=\tilde{\lambda}_ip_i$ for demand units of type $i$. The demand arrival rates are denoted by $\bm{\lambda}=\left(\lambda_1, \ldots, \lambda_{n_d}\right)$.

We will use $\tau$ to denote a scaling parameter and use the vectors $\mathbf{D}^{\tau} = (D_1^\tau,...,D_{n_d}^\tau)$ and $\bm{S}^{\tau}=(S_1^\tau,...,S_{n_s}^\tau)$ to represent respectively the total number of demand and supply units of each type during a matching cycle. The scaling parameter allows us to consider different regimes where supply and demand arrivals are rare (low $\tau$) or frequent (high $\tau$) and can be thought of as the marketplace density. For each $i$ and $j$, $D_i^\tau$ and $S_j^\tau$ are Poisson-distributed random variables with respective means $\lambda_i \tau$ and $\pi_j \tau$. 

\paragraph{\revision{Special case: bipartite matching.}} Given a set of supply and demand units, the platform's goal is to assign each demand unit to a supply unit. \revision{In order to build intuition, we first present the special case where the platform simply needs to solve a bipartite matching (or transportation) problem. In this case, the platform's optimization problem given demand $\mathbf{D}^{\tau}=\bm{d}$ and supply $\bm{S}^{\tau}=\bm{s}$ can be written as:}

\begin{subequations}
\label{eq:matching-formulation}
\begin{align}
\Phi^{b}(\bm{d},\bm{s})=&\;\max\sum_{i=1}^{n_d}\sum_{j=1}^{n_s}v_{i,j}x_{i,j},\\
&\;\text{s.t. }\sum_{j=1}^{n_s}x_{i,j}\le d_i & \forall i\in [n_d],\label{eq:matching-demand}\\
&\;\phantom{\text{s.t. }}\sum_{i=1}^{n_d}x_{i,j}\le s_j & \forall j\in [n_s],\label{eq:matching-supply}\\
&\;\phantom{\text{s.t. }}x_{i,j}\geq 0&\forall i\in[n_d],j\in[n_s],\label{eq:matching-relaxed}
\end{align}
\end{subequations}

The decision variables $x_{i,j}$ indicate the number of demand units of type $i$ {matched with} supply units of type $j$. Constraints~\eqref{eq:matching-demand} and \eqref{eq:matching-supply} ensure that the matching is feasible from both the demand and supply sides. \revision{Because} the constraint matrix is totally unimodular, we can ignore integrality constraints on the decision variables.

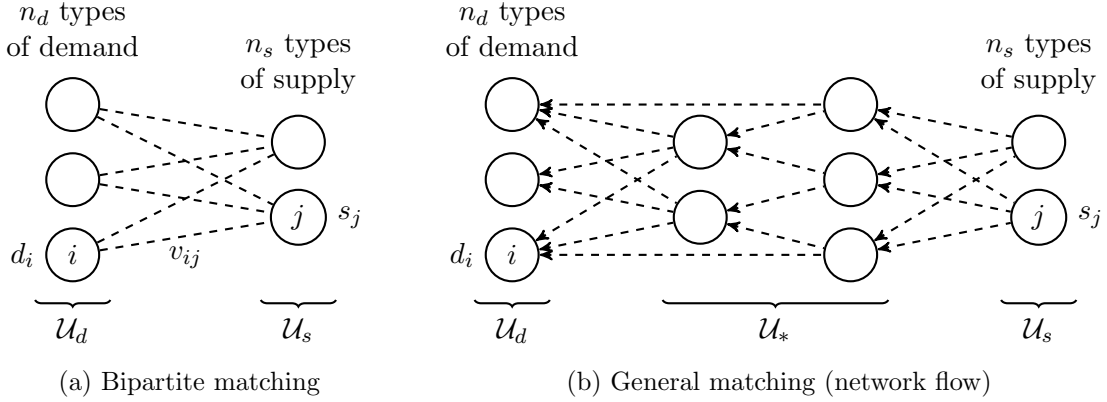
\begin{figure}
    \begin{subfigure}[b]{0.35\columnwidth}
    \centering
    \begin{tikzpicture}[
    thick,
    >=stealth'
  ]
    \coordinate[label={[align=center]above:{$n_d$ types\\of demand}}](DEMAND_LABEL) at (0, 4.5);
    \node [circle,draw,fill=white,minimum size=20,label={}] (D1) at (0,4) {};
    \node [circle,draw,fill=white,minimum size=20] (D2) at (0,3) {};
    \node [circle,draw,fill=white,minimum size=20,label={left:$d_i$}] (Di) at (0,2) {$i$};
    
    \coordinate[label={[align=center]above:{$n_s$ types\\of supply}}](SUPPLY_LABEL) at (3, 4);
    \node [circle,draw,fill=white,minimum size=20,label={}] (S1) at (3,3.5) {};
    \node [circle,draw,fill=white,minimum size=20,label={right:$s_j$}] (Sj) at (3,2.5) {$j$};
    \draw[dashed] (D1) -- (S1);
    \draw[dashed] (D1) -- (Sj);
    \draw[dashed] (D2) -- (S1);
    \draw[dashed] (D2) -- (Sj);
    \draw[dashed] (Di) -- (S1);
    \draw[dashed] (Di) -- (Sj);
    \coordinate[label={[label distance=0mm]below:$v_{ij}$}] (DS_MID) at ($(Di)!0.5!(Sj)$);
    \draw [decorate,decoration = {brace,mirror}] (-0.5,1.4) --  (0.5,1.4);
    \coordinate[label={below:$\mathcal{U}_d$}]() at (0, 1.3);
    \draw [decorate,decoration = {brace,mirror}] (2.5,1.4) --  (3.5,1.4);
    \coordinate[label={below:$\mathcal{U}_s$}]() at (3, 1.3);
  
\end{tikzpicture}
    \caption{Bipartite matching}
    \end{subfigure}%
    \begin{subfigure}[b]{0.6\columnwidth}
    \centering
    \begin{tikzpicture}[
    thick,
    >=stealth'
  ]
    \coordinate[label={[align=center]above:{$n_d$ types\\of demand}}](DEMAND_LABEL) at (0, 4.5);
    \node [circle,draw,fill=white,minimum size=20,label={}] (D1) at (0,4) {};
    \node [circle,draw,fill=white,minimum size=20] (D2) at (0,3) {};
    \node [circle,draw,fill=white,minimum size=20,label={left:$d_i$}] (Di) at (0,2) {$i$};
    
    \coordinate[label={[align=center]above:{$n_s$ types\\of supply}}](SUPPLY_LABEL) at (7, 4);
    \node [circle,draw,fill=white,minimum size=20,label={}] (S1) at (7,3.5) {};
    \node [circle,draw,fill=white,minimum size=20,label={right:$s_j$}] (Sj) at (7,2.5) {$j$};

    \node [circle,draw,fill=white,minimum size=20,label={}] (I1) at (2.5,2.5) {};
    \node [circle,draw,fill=white,minimum size=20,label={}] (I2) at (2.5,3.5) {};
    \node [circle,draw,fill=white,minimum size=20,label={}] (I3) at (4.5,2) {};
    \node [circle,draw,fill=white,minimum size=20,label={}] (I4) at (4.5,3) {};
    \node [circle,draw,fill=white,minimum size=20,label={}] (I5) at (4.5,4) {};
    
    \draw[dashed, <-] (D1) -- (I1);
    \draw[dashed, <-] (D1) -- (I2);
    \draw[dashed, <-] (D2) -- (I1);
    \draw[dashed, <-] (D2) -- (I2);
    \draw[dashed, <-] (Di) -- (I1);
    \draw[dashed, <-] (Di) -- (I2);
    \draw[dashed, <-] (D1) -- (I5);
    \draw[dashed, <-] (Di) -- (I3);
    \draw[dashed, <-] (I1) -- (I3);
    \draw[dashed, <-] (I1) -- (I4);

    \draw[dashed, <-] (I2) -- (I4);
    \draw[dashed, <-] (I2) -- (I5);
    \draw[dashed, <-] (I3) -- (S1);
    \draw[dashed, <-] (I4) -- (S1);
    \draw[dashed, <-] (I5) -- (S1);
    \draw[dashed, <-] (I3) -- (Sj);
    \draw[dashed, <-] (I4) -- (Sj);
    \draw[dashed, <-] (I5) -- (Sj);

    \draw [decorate,decoration = {brace,mirror}] (-0.5,1.4) --  (0.5,1.4);
    \coordinate[label={below:$\mathcal{U}_d$}]() at (0, 1.3);
    \draw [decorate,decoration = {brace,mirror}] (2,1.4) --  (5,1.4);
    \coordinate[label={below:$\mathcal{U}_*$}]() at (3.5, 1.3);
    \draw [decorate,decoration = {brace,mirror}] (6.5,1.4) --  (7.5,1.4);
    \coordinate[label={below:$\mathcal{U}_s$}]() at (7, 1.3);
  
\end{tikzpicture}
    \caption{General matching (network flow)}
    \end{subfigure}
    \caption{Diagram of our matching setting. The general setting simply replaces the direct edges between demand and supply with an arbitrary directed flow network.}
    \label{fig:matching-graph}
\end{figure}

\paragraph{General \revision{case:} matching via network flows.}
    \revision{Though the bipartite matching case is more intuitive, the results in this paper apply in a more general setting, where the platform's optimization problem is any minimum-cost network flow problem. The goal of the platform is still to connect demand units and supply units; however, it must now do so by pushing units of flow through an arbitrary capacitated network. More precisely, we are now given} a directed graph $\mathcal{G}=(\mathcal{U}, \mathcal{A})$.
    \revision{The node set $\mathcal{U}=\mathcal{U}_d\cup\mathcal{U}_s\cup\mathcal{U}_*$ contains three types of nodes: demand, supply, and intermediate.} The set $\mathcal{U}_d$ has one node for each demand type ($|\mathcal{U}_d|=n_d$), and the node $u^d_i$ corresponding to the $i$-th demand type is a sink with demand $D_i^\tau$; similarly, the set $\mathcal{U}_s$ has one node for each supply type ($|\mathcal{U}_s|=n_s$), and the node $u^s_j$ corresponding to the $j$-th supply type is a source with supply $S_j^\tau$. Intermediate nodes are neither sources nor sinks, and their set $\mathcal{U}_*$ can be arbitrarily large. To simplify notation and without loss of generality, we assume demand nodes (sinks) only have incoming edges and supply nodes (sources) only have outgoing edges, while intermediate nodes can have both incoming and outgoing edges. Given a demand vector $\mathbf{D}^{\tau}=\bm{d}$ and a supply vector $\bm{S}^{\tau}=\bm{s}$, we can write the platform's optimization problem as \revision{the} linear program:
    \begin{subequations}
        \label{eq:general-matching-formulation}
        \begin{align}
            \Phi_\tau(\bm{d},\bm{s})=&\;\max\sum_{(u,w)\in\mathcal{A}}v_{u,w}x_{u,w},\\
&\;\text{s.t. }\sum_{w:(w,u)\in\mathcal{A}}x_{w,u} \le d_i  & \forall u=u_i^d\in\mathcal{U}_d,\label{eq:generalized-matching-demand}\\
&\;\phantom{\text{s.t. }}\sum_{w:(u,w)\in\mathcal{A}}x_{u,w} \le s_j  & \forall u=u_j^s\in\mathcal{U}_s,\label{eq:generalized-matching-supply}\\
&\;\phantom{\text{s.t. }}\sum_{w:(u,w)\in\mathcal{A}}x_{u,w} - \sum_{w:(w,u)\in\mathcal{A}}x_{w,u} = 0  & \forall u\in\mathcal{U}_*,\label{eq:generalized-matching-flowbalance}\\
&\;\phantom{\text{s.t. }}0\le x_{u,w} \le \tau k_{u,w} &\forall (u,w)\in\mathcal{A}.\label{eq:generalized-matching-capacity}
        \end{align}
    \end{subequations}
The decision variable $x_{u,w}$ indicates the flow along arc $(u,w)$ in the graph. Constraints~\eqref{eq:generalized-matching-demand} and \eqref{eq:generalized-matching-supply} ensure that we do not exceed demand and supply capacity; constraint~\eqref{eq:generalized-matching-flowbalance} ensures flow balance at each intermediate node. Finally, constraint~\eqref{eq:generalized-matching-capacity} imposes a capacity on each edge – note that this capacity also scales with $\tau$. Each unit of flow along arc $(u,w)$ generates a value $v_{u,w}$ to the platform. The problem is formulated as a value-maximization rather than a cost-minimization problem, but the two are mathematically equivalent. \revision{When the set of intermediate nodes $\mathcal{U}_*$ is empty, we recover the special case of bipartite matching. A diagram of our matching settings is shown in Figure~\ref{fig:matching-graph}.}

\revision{As before}, since the constraint matrix of this problem is totally unimodular, we can relax the integer constraints \revision{and retain an integral} optimal solution. To avoid introducing more notation, we will use $\Phi_\tau(\cdot,\cdot)$ to represent both the LP itself and its optimal solution value, depending on the context. For $\tau=1$ we will drop the subscript altogether and refer to $\Phi(\cdot,\cdot)$.

\subsection{The Experimental Setting}

We consider a treatment that affects the demand arrival rate to the platform. This treatment could involve a change in the user experience (e.g., a smartphone notification to increase engagement) or a financial incentive (e.g., a coupon). A key requirement is that the treatment should not affect the matching function, i.e., the matching algorithm should not distinguish between treated and untreated units. For clarity of exposition, we assume that the treatment does not affect the supply arrival rate. However, due to the symmetry of constraints~\eqref{eq:generalized-matching-demand} and \eqref{eq:generalized-matching-supply}, the framework extends to a treatment that only affects the supply arrival rate, or that jointly affects both the demand and supply arrival rates.

We assume that, if the change were to be rolled out to all demand intent units, the effect of the change would be to modify the request probability of a demand unit of type $i$ from $p_i$ to $p_i+q_i$, with $q_i\in[-p_i,1-p_i]$. From the perspective of the platform, this modification in request probability is equivalent to a change in demand arrival rates for the treatment group. We denote this change in arrival rate as $\bm{\beta}$, such that $\beta_i=\tilde{\lambda}_iq_i$. Evaluating the impact of our intervention means estimating the difference in value between the \emph{global treatment} state, in which demand units arrive at rate $\bm{\lambda}+\bm{\beta}$, and the \emph{global control} state, in which demand units arrive at rate $\bm{\lambda}$. For a particular scaling factor $\tau$, let $\bm{D}^{\tau, \bm{\lambda}}$ denote the realized demand vector with demand arrival rate $\bm{\lambda}$ (for simplicity, we maintain the notation of $\bm{S}^\tau$ instead of $\bm{S}^{\tau,\bm\pi}$ since our demand-side treatment does not affect supply). Our objective is to estimate the global treatment effect $\Delta^{\tau}$, which is defined as follows.

\begin{definition}\label{def:gte}
For a given scaling parameter $\tau$, the global treatment effect $\Delta^{\tau}$ is the difference between the expected platform value under global treatment and under global control, namely
\begin{equation*}
\Delta^{\tau} = \frac{1}{\tau}\mathbb{E}\left[\Phi_\tau(\bm{D}^{\tau, \bm{\lambda} + \bm{\beta}},\bm{S}^{\tau})-\Phi_\tau(\bm{D}^{\tau, \bm{\lambda}},\bm{S}^{\tau})\right].
\end{equation*}
\end{definition}

We include a $1/\tau$ term in front of the expectation in Definition \ref{def:gte} because we are interested in a ``scale-free'' treatment effect, rather than the treatment effect for a system of a particular scale. 

Unfortunately, we cannot directly observe the global treatment effect since we cannot simultaneously observe the marketplace under both global treatment and global control. Instead, we use a \emph{demand-split} experiment design: we assign each demand unit upon arrival to the \emph{treatment} group, with probability $\rho$, or the \emph{control} group with probability $1-\rho$. Therefore, the arrival rate of treated demand units will be $\bm{\lambda}^{\treatment} = \rho(\boldsymbol{\lambda}+\boldsymbol{\beta})$ and the arrival rate of control demand units will be $\bm{\lambda}^{\control} = \left(1-\rho\right)\boldsymbol{\lambda}$, for an experiment demand arrival rate of:
\begin{equation*}
\bm{\lambda}^{\experiment}=\bm{\lambda}^{\treatment}+\bm{\lambda}^{\control}=\rho\left(\boldsymbol{\lambda}+\boldsymbol{\beta}\right)+\left(1-\rho\right)\boldsymbol{\lambda}=\boldsymbol{\lambda}+\rho\boldsymbol{\beta}.
\end{equation*}

Given a scaling factor $\tau$, we observe demand $\bm{D}^{\tau,\experiment}=\bm{D}^{\tau,\control} + \bm{D}^{\tau,\treatment}$ in the experiment, where $D_i^{\tau,\control}\sim \text{Poisson}(\lambda^{\control}_i\tau)$ and $D_i^{\tau,\treatment}\sim \text{Poisson}(\lambda^{\treatment}_i\tau)$, implying that $D^{\tau,\experiment}_i\sim \text{Poisson}(\lambda_i\tau + \rho\beta_i\tau)$. During the experiment, we obtain data by solving
\begin{equation}\label{eq:experiment-lp}\Phi_\tau(\bm D^{\tau,\experiment},\bm{S}^{\tau}).\end{equation} In the following section we describe this experimental data in more detail.

\subsection{The Experimental Data}

When the platform solves the matching problem $\Phi(\bm{D}^{\tau,\experiment},\bm{S}^\tau)$ during an experiment, it obtains data specifying the optimal match. \revision{We first describe this data in the special case of bipartite matching, then consider the general network flow case.}

\paragraph{\revision{Bipartite matching.}} \revision{By solving the matching linear program in the experiment state, the platform obtains two types of data. \emph{Primal data} describes the optimal matching explicitly, and includes the number of control (respectively, treatment) demand units of type $i$ matched with supply units of type $j$, denoted by $X^{\tau,\control}_{i,j}$ (respectively, $X_{i,j}^{\tau,\treatment}$). Because the matching problem is a linear program, the platform also obtains \emph{dual data}, in the form of an optimal solution to the dual of the LP in Eq.~\eqref{eq:matching-formulation}. Given realized demand $\bm{d}$ and supply $\bm{s}$, the dual is another linear program, which can be written as:}

\begin{subequations}
\label{eq:matching-dual}
\begin{align}
    \Phi_\tau^{b}(\bm{d},\bm{s})=&\;\min\sum_{i=1}^{n_d}a_id_i+\sum_{j=1}^{n_s}b_js_j,\\
    &\;\text{s.t. }a_i+b_j\ge v_{i,j} & \forall i\in[n_d],j\in[n_s],\\
    &\;\phantom{\text{s.t. }}a_i\ge 0 & \forall i\in[n_d],\\
    &\;\phantom{\text{s.t. }}b_j\ge 0 & \forall j\in[n_s],
\end{align}
\end{subequations}
where $a_i$ corresponds to the shadow price of an additional unit of demand of type $i$, $b_j$ corresponds to the shadow price of an additional unit of supply of type $j$. Most linear programming algorithms provide an optimal dual solution ``for free'' along with the optimal primal solution (the optimal dual solution certifies the optimality of the primal solution). For an experiment with demand $\bm{D}^{\tau,\experiment}=\bm{D}^{\tau,\control}+\bm{D}^{\tau,\treatment}$ and supply $\bm{S}^{\tau}$, we denote this dual data by $(\bm{A}^{\tau,\experiment}, \bm{B}^{\tau,\experiment})$.

\paragraph{\revision{General matching.}} \revision{In the general case, we similarly obtain primal and dual data. This time, the primal data must describe the number of demand units from the control group (respectively, treatment group) using each edge $(u,w)$ in the flow network, denoted by $X_{u,w}^{\tau,\control}$ (respectively, $X_{u,w}^{\tau,\treatment}$). We denote the total flow along edge $(u,w)$ by $X_{u,w}^{\tau,\experiment}=X_{u,w}^{\tau,\control}+X_{u,w}^{\tau,\treatment}$.}

In addition to the flow along each edge, another, more interpretable primal quantity \revision{in the general case} is the number of demand units of type $i$ matched to supply units of type $j$, and more precisely the number matched via a particular path in the graph, since different paths may yield different values. We denote by $n_{i,j}$ the number of paths from the supply node $u_j^s$ to the demand node $u_i^s$. For the $p$-th such path, we denote by $Y_{i,j,p}^{\tau,\control}$ (respectively $Y_{i,j,p}^{\tau,\treatment}$) the number of demand units from the control group (respectively treatment group) matched to a supply unit of type $j$ via path $p$. The total number of demand units of type $i$ matched to a supply unit of type $j$ along path $p$ is written as $Y_{i,j,p}^{\tau,\experiment}=Y_{i,j,p}^{\tau,\control}+Y_{i,j,p}^{\tau,\treatment}$. Each demand unit of type $i$ matched to a supply unit of type $j$ creates value $\nu_{i,j,p}=\sum_{(u,w)\in\mathcal{A}_{i,j,p}}v_{u,w}$, where $\mathcal{A}_{i,j,p}\subseteq \mathcal{A}$ designates the edges in path $p$ connecting supply node $u_j^s$ to demand node $u_i^d$.

\revision{We also obtain dual data in the general case.} Because demand nodes have no outgoing edges and supply nodes have no incoming edges, we can partition the edge set $\mathcal{A}$ into four types of edges: $\mathcal{A}_{s\to d}$ denotes the set of edges linking supply nodes to demand nodes, $\mathcal{A}_{s\to *}$ edges from supply nodes to intermediate nodes, $\mathcal{A}_{*}$ edges from intermediate nodes to other intermediate nodes, and $\mathcal{A}_{*\to d}$ edges from intermediate nodes to demand nodes. The dual of the general matching problem can then be formulated as follows:
\begin{subequations}
\label{eq:general-matching-dual}
\begin{align}
    \Phi_\tau(\bm{d},\bm{s})=&\;\min\sum_{i=1}^{n_d}a_id_i+\sum_{j=1}^{n_s}b_js_j+\sum_{(u,w)\in\mathcal{A}}\xi_{u,w}k_{u,w}\tau,\\
    &\;\text{s.t. }b_j + a_i + \xi_{u_j^s,u_i^d} \ge v_{u_j^s,u_i^d} & \forall (u_j^s,u_i^d)\in\mathcal{A}_{s\to d},\\
    &\;\phantom{\text{s.t. }}b_j-m_w + \xi_{u_j^s,w}\ge v_{u_j^s,w} & \forall (u_j^s,w)\in\mathcal{A}_{s\to *},\\
    &\;\phantom{\text{s.t. }}m_w+a_i + \xi_{w,u_i^d}\ge v_{w,u_i^d} & \forall (w, u_i^d)\in\mathcal{A}_{*\to d},\\
    &\;\phantom{\text{s.t. }}m_u - m_w + \xi_{u,w} \ge v_{u,w} & \forall (u,w)\in\mathcal{A}_{*},\\
    &\;\phantom{\text{s.t. }}a_i\ge 0 & \forall i\in[n_d],\\
    &\;\phantom{\text{s.t. }}b_j\ge 0 & \forall j\in[n_s],\\
    &\;\phantom{\text{s.t. }}\xi_{u,w}\ge 0 & \forall (u,w)\in\mathcal{A}.
\end{align}
\end{subequations}

\revision{As before}, $a_i$ corresponds to the shadow price of an additional unit of demand of type $i$, $b_j$ corresponds to the shadow price of an additional unit of supply of type $j$. \revision{The additional variable $\xi_{u,w}$} corresponds to the shadow price of increasing the capacity of edge $(u,w)$, \revision{while the} dual variable $m_u$ for each intermediate node $u\in\mathcal{U}_*$ can be thought of as defining a potential function which relates the demand and supply shadow prices via the paths that connect their respective nodes in the graph. \revision{These additional variables ensure correctness of the dual, but are not useful in our methodology. Only the shadow prices of demand and supply matter, which is why the bipartite matching example can be easier to reason about.}
For an experiment with demand $\bm{D}^{\tau,\experiment}=\bm{D}^{\tau,\control}+\bm{D}^{\tau,\treatment}$ and supply $\bm{S}^{\tau}$, we denote \revision{the} dual data by $(\bm{A}^{\tau,\experiment}, \bm{B}^{\tau,\experiment}, \bm{M}^{\tau,\experiment}, \bm{\Xi}^{\tau,\experiment})$.

\paragraph{\revision{Interchangeability of demand and control units.}} A key assumption of our framework is that treatment affects the arrival rate of demand units, but not the value they bring to the platform. In other words, the platform treats demand units in the treatment and control groups \revision{the same way}. As a result, there may be multiple optimal solutions, since a control and treatment unit of the same type can always be exchanged without affecting the objective value. We assume the platform selects between these optimal solutions at random, respecting the following assumption.

\begin{assumption}\label{ass:blind}
The matching LP is blind to the treatment or control status of the demand units. Formally, the primal matching data verify:
\[
\mathbb{E}\left[Y_{i,j,p}^{\tau,\controlemph}|\bm{D}^{\tau,\controlemph}, \bm{D}^{\tau,\treatmentemph}, \bm S^\tau\right]=\frac{D_i^{\tau,\controlemph}}{D_i^{\tau,\experimentemph}}Y_{i,j,p}^{\tau,\experimentemph} \qquad \forall i\in[n_d], j\in[n_s], p\in[n_{i,j}].
\]
\end{assumption}

\revision{Assumption~\ref{ass:blind} is motivated by practice; we would like the experiment to inform us about the state of the system under global control (where no units are treated) and under global treatment (where all units are treated). In both of these states, all units of the same type are interchangeable.}
\section{{The Standard Estimator and Interference Bias}} \label{sec:rct-failure}

In this section, we consider the most commonly used estimator for the treatment effect, and identify the source of its bias under marketplace interference.

\subsection{The RCT Estimator}

Because randomized control trial designs are the workhorse of experimentation studies, there  exists a standard estimator for the global treatment effect. We refer to it as the randomized control trial (RCT) estimator, often called Horvitz-Thompson estimator in the statistics literature.

\begin{definition}
\label{def:rct-estimator}
The RCT estimator for the global treatment effect is given by:
\begin{equation*}
    \hat{\Delta}^{\tau}_{\text{\emph{RCT}}}=\frac{1}{\tau}\left(\frac{1}{\rho}\sum_{(u,w)\in\mathcal{A}}v_{u,w}X_{u,w}^{\tau,\treatmentemph}-\frac{1}{1-\rho}\sum_{(u,w)\in \mathcal{A}}v_{u,w}X_{u,w}^{\tau,\controlemph}\right).
\end{equation*}
\end{definition}

In the equation above, the first term inside the parenthesis estimates the value of global treatment by re-scaling the value obtained from the treatment group by the treatment fraction $\rho$. Similarly, the second term estimates the value of global control by re-scaling the control group value by the control fraction $1-\rho$. Equivalently, the RCT estimator can be written as:
\begin{equation}
    \label{eq:path-based-rct}
    \hat{\Delta}^{\tau}_{\text{RCT}}=\frac{1}{\tau}\left(\frac{1}{\rho}\sum_{i=1}^{n_d}\sum_{j=1}^{n_s}\sum_{p=1}^{n_{i,j}}\nu_{i,j,p}Y_{i,j,p}^{\tau,\treatment}-\frac{1}{1-\rho}\sum_{i=1}^{n_d}\sum_{j=1}^{n_s}\sum_{p=1}^{n_{i,j}}\nu_{i,j,p}Y_{i,j,p}^{\tau,\control}\right),
\end{equation}
where the total value in each group is aggregated by summing over paths rather than edges.

In the special case of bipartite matching, the RCT estimator is given by
\begin{equation*}
    \hat{\Delta}^{\tau}_{\text{RCTb}}=\frac{1}{\tau}\left(\frac{1}{\rho}\sum_{i=1}^{n_d}\sum_{j=1}^{n_s}v_{i,j}X_{i,j}^{\tau,\treatment}-\frac{1}{1-\rho}\sum_{i=1}^{n_d}\sum_{j=1}^{n_s}v_{i,j}X_{i,j}^{\tau,\control}\right),
\end{equation*}
where $X_{i,j}^{\tau,\control}$ (respectively $X_{i,j}^{\tau,\treatment}$) marks the number of demand units of type $i$ from the control group (respectively treatment group) matched with a supply unit of type $j$. The edge-based and path-based definitions coincide \revision{because} there is a single edge between each supply and demand node.
Comparing the {total value} obtained from the treatment group and the {total value obtained from the} control group is the  standard estimation technique in randomized control trials.

\subsection{The Fluid Limit}

Exact analysis of our discrete model can be challenging due to the stochasticity of the Poisson processes generating supply and demand. The challenge is particularly acute when the market is ``thin'' --- i.e., when the low density of supply and demand confers disproportionate influence to small variations in Poisson arrivals. Because these effects vanish when density increases, we \revision{study} a high-density limit of the system, where $\tau$ \revision{grows} arbitrarily large. For simplicity, we only consider integer densities $\tau$. \revision{A} key quantity of interest \revision{is} the fluid limit value for a given demand rate:
\begin{equation*}
\label{eq:definition-infinite-horizon-value}
\lim_{\tau\rightarrow\infty}\mathbb{E}\left[\frac{1}{\tau}\Phi_\tau\left(\mathbf{D}^{\tau,\bm \lambda},\mathbf{S}^{\tau}\right)\right].
\end{equation*}

When density tends to infinity, we can ignore the stochastic fluctuations of the Poisson process and focus only on the expected number of demand and supply arrivals, given by the rates $\bm{\lambda}$ and $\bm{\pi}$. We now prove that the fluid limit value is given by the LP from Eq. \eqref{eq:matching-formulation} where the inputs are the rates $\bm \lambda$ and $\bm \pi$ rather than realizations of demand and supply. 

\begin{theorem}
\label{thm:fluid-limit}For a given demand rate $\bm \lambda$, the fluid limit
of the average total value is given by
\begin{equation}
\lim_{\tau\rightarrow\infty}\mathbb{E}\left[\frac{1}{\tau}\Phi_\tau\left(\mathbf{D}^{\tau, \bm \lambda},\mathbf{S}^{\tau}\right)\right]=\Phi\left(\boldsymbol{\lambda},\boldsymbol{\pi}\right)\label{eq:thm5}.
\end{equation}
Further, let $X^{\tau}_{u,w}$ and $(A^{\tau}_i, B^{\tau}_j, M^{\tau}_w,\Xi^\tau_{u,w})$ \revision{correspond}
to the optimal primal and dual solutions of the matching LP $\Phi_\tau\left(\frac{1}{\tau}\mathbf{D}^{\tau,\bm \lambda},\frac{1}{\tau}\mathbf{S}^{\tau}\right)$. Then {$X^{\tau}_{u,w}$} converges almost surely to 
the optimal primal solution of $\Phi\left(\bm{\lambda},\bm{\pi}\right)$; and {$(A^{\tau}_i, B^{\tau}_j, M^{\tau}_w, \Xi^\tau_{u,w})$} converge almost surely to 
the optimal dual solution of $\Phi\left(\bm{\lambda},\bm{\pi}\right)$. 
\end{theorem}

Theorem~\ref{thm:fluid-limit} allows us to simplify our analysis \revision{by studying the deterministic linear program $\Phi(\bm \lambda, \bm \pi)$ in lieu of a complex stochastic system. The simplified system retains the key property under study, namely} the effect of  marketplace contention for scarce supply. The first implication of Theorem~\ref{thm:fluid-limit} is that it simplifies the expression for the global treatment effect from Definition~\ref{def:gte}.

\begin{corollary}
\label{cor:gte-fluid-limit}
The fluid limit of the global treatment effect satisfies:
\begin{equation}
\Delta {= \lim_{\tau\to \infty}\Delta^\tau}=\Phi\left(\boldsymbol{\lambda}+\boldsymbol{\beta},\boldsymbol{\pi}\right)-\Phi\left(\boldsymbol{\lambda},\boldsymbol{\pi}\right).
\end{equation}
\end{corollary}

In the fluid limit, the global treatment effect is simply the difference between the optimal values of two linear programs. In contrast, running a demand-split experiment \revision{means considering} the ``intermediate'' linear program $\Phi(\bm \lambda^{\experiment}, \bm \pi)=\Phi(\bm\lambda + \rho\bm\beta, \bm\pi)$. In effect, the experiment state interpolates between the global control and global treatment states. We formalize this notion of interpolation by introducing a useful tool, which we call the partial-treatment value function.

\begin{definition}
    Given a matching experiment with supply arrival rate $\bm{\pi}$, demand arrival rate of $\bm{\lambda}$ under global control, and demand arrival rate $\bm{\lambda}+\bm{\beta}$ under global treatment, for $\eta\in[0,1]$, the partial-treatment value function is given by $\Psi(\eta)=\Phi(\bm{\lambda}+\eta\bm{\beta},\bm \pi)$.
\end{definition}

Using the partial-treatment value function, we can equivalently write the global treatment effect as $\Delta=\Psi(1)-\Psi(0)$. Under our demand-split experiment with treatment fraction $\rho$, we observe the total value $\Psi(\rho)$. 
Using classical results from linear programming, we can establish that $\Psi(\cdot)$ has a particular functional form, which will be useful in our analysis.

\begin{proposition}   \label{prop:concave-integrable}
Over the interval $[0,1]$, the partial-treatment value function $\Psi(\cdot)$ is a concave piece-wise linear function with finitely many pieces.
\end{proposition}

To guarantee that all estimators are well-defined in the fluid limit, we will also make the following assumption throughout the paper.

\begin{assumption}
\label{ass:unique} 
The optimal primal and dual solutions of $\Phi(\bm \lambda + \rho\bm \beta,\bm \pi)$ as defined in Eqs.~\eqref{eq:general-matching-formulation} and \eqref{eq:general-matching-dual} are both unique.
\end{assumption}

\subsection{Performance of the RCT Estimator}\label{ssec:rct-performance}

By allowing us to restrict our analysis to a deterministic setting, the fluid limit allows us to gain insight into the performance of the {standard} RCT estimator.

\begin{proposition}
\label{prop:infinite-horizon-limit-rct}
In the fluid limit, the RCT estimator satisfies:
\begin{equation}
    \label{eq:infinite-horizon-limit-rct}
    {\hat{\Delta}_{\text{\emph{RCT}}}
=\lim_{\tau\to\infty}\hat{\Delta}^\tau_{\text{\emph{RCT}}}}
    =\bar{\bm{v}}^*\cdot \bm\beta,
\end{equation}
where
\begin{equation}
\bar{v}_i^*={\sum_{j=1}^{n_s}\sum_{p=1}^{n_{i,j}}\frac{y^*_{i,j,p}}{\lambda^{\experimentemph}_i}\nu_{i,j,p},}
\end{equation}
and $y^*_{i,j,p}$ is the number of demand units of type $i$ matched to supply units of type $j$ along path $p$ in the optimal solution of $\Phi(\bm\lambda^{\experimentemph}, \bm\pi)$.
\end{proposition}

Proposition~\ref{prop:infinite-horizon-limit-rct} provides a simple explanation of the RCT estimator's behavior in the fluid limit: it first uses the experiment data to compute the average value obtained by each demand type $i$, called $\bar v_i^*$. It then estimates the value of treatment on group $i$ by multiplying $\bar v_i^*$ with $\beta_i$, which is the demand effect of the treatment on group $i$.

An easy way to understand what the RCT estimator does is to observe that it corresponds to building a specific linear approximation of the {matching} value function. In particular, it {notices that at the experiment point,} the value function verifies $\Phi(\bm {\lambda^{\experiment}},\pi) = \bm{\bar v^*} \cdot \bm {\lambda^{\experiment}}${, and assumes this linear relationship continues to apply when the demand is $\bm\lambda$ (global control) and $\bm\lambda + \bm\beta$ (global treatment). This approximation would be correct in the absence of marketplace interference.} In {its} presence{, however}, the value function is concave rather than linear  (Proposition \ref{prop:concave-integrable}). {We can show that ignoring \revision{this concavity} introduces a systematic bias in the RCT estimator. Before we do so, we introduce our third and final assumption on the treatment effect.

\begin{assumption}
\label{ass:sign-consistent}
The treatment effect is sign-consistent across types, i.e., either $\beta_i\ge 0$ for all $i$ or $\beta_i\le 0$ for all $i$.
\end{assumption}
}

Without this assumption, it is possible to construct an example with $\beta_1 > 0$ and $\beta_2 < 0$ where, due to interference, the RCT estimator simultaneously overestimates (in amplitude) the positive effect from the increased demand of type 1 and the negative effect from the decreased demand of type 2, in such a way that the respective biases (which have opposite signs) cancel. The resulting estimator can have zero bias, but this is due to a lucky coincidence rather than any inherent quality of the estimator. {All the results that follow rely on Assumptions \ref{ass:blind}-\ref{ass:sign-consistent} unless otherwise stated.}

\begin{theorem}
\label{thm:rct-overestimate}
The RCT estimator always overestimates the global treatment effect (in amplitude), i.e., $\abs{\hat{\Delta}_{\text{\emph{RCT}}}}\ge \abs{\Delta}$.  In particular, if $\bm\beta\ge0$, the RCT estimator simultaneously overestimates the value of global treatment and underestimates the value of global control (and vice versa if $\bm\beta\le0$).
\end{theorem}

We illustrate Theorem~\ref{thm:rct-overestimate} through a running example, visualized in Figure~\ref{fig:example-setup}. Consider a {bipartite matching} setting \revision{with} one demand type ($n_d=1$), with arrival rate $\lambda_1=1.5$ under control, and $\lambda_1+\beta=5.5$ under treatment. Additionally, we have six supply types ($n_s=6$), with $\pi_j=1$ and \revision{$v_{1,j}=2^{(2-j)}$} for each supply type $j$. Our example exhibits contention since the arrival rate of the highest-value supply $\pi_1=1$ is less than the arrival rate of demand $\lambda_1=1.5$ under control (and much less than under treatment). As a result, the function $\lambda \to \Phi(\lambda, \bm{\pi})$ is strictly concave. We plot this function in Figure~\ref{fig:example-function}, noting that the true global treatment effect $\Delta$ \revision{equals} the difference in the function's value at $\lambda_1+\beta$ (global treatment) and its value at $\lambda_1$ (global control).

\begin{figure}[ht!]
\centering
\begin{subfigure}[b]{0.36\columnwidth}
\centering
\begin{tikzpicture}[
    thick,
    >=stealth'
  ]
    \coordinate[label={[align=right]left:{\footnotesize Control:\\\footnotesize$\lambda_1=1.5$}}]() at (0.6, 3.5);
    \coordinate[label={[align=right]left:{\footnotesize Treatment:\\\footnotesize$\lambda_1+\beta=5.5$}}]() at (0.6, 1.5);
    \node [circle,draw,fill=white,minimum size=20,label={}] (D1) at (0,2.5) {};
    \draw [decorate,decoration = {brace,mirror}] (3.5,-0.3) --  (3.5,5.3);
    \coordinate[label={[align=left]right:{\footnotesize 6 supply\\\footnotesize types,\\\footnotesize$\bm\pi=\bm1$}}](SUPPLY_LABEL) at (3.6, 2.5);
    \node [circle,draw,fill=white,minimum size=20] (S1) at (3,5) {};
    \node [circle,draw,fill=white,minimum size=20] (S2) at (3,4) {};
    \node [circle,draw,fill=white,minimum size=20] (S3) at (3,3) {};
    \node [circle,draw,fill=white,minimum size=20] (S4) at (3,2) {};
    \node [circle,draw,fill=white,minimum size=20] (S5) at (3,1) {};
    \node [circle,draw,fill=white,minimum size=20] (S6) at (3,0) {};
    \draw[dashed] (D1) -- (S1);
    \draw[dashed] (D1) -- (S2);
    \draw[dashed] (D1) -- (S3);
    \draw[dashed] (D1) -- (S4);
    \draw[dashed] (D1) -- (S5);
    \draw[dashed] (D1) -- (S6);
    
    \coordinate[label={[label distance=-2mm]above left:{\scriptsize $v_{1,1}=2$}}]() at ($(D1)!0.73!(S1)$);
    \coordinate[label={[label distance=0mm]above:{\scriptsize 1}}]() at ($(D1)!0.7!(S2)$);
    \coordinate[label={[label distance=-1mm]above:{\scriptsize $\frac{1}{2}$}}]() at ($(D1)!0.7!(S3)$);
    \coordinate[label={[label distance=-1mm]above:{\scriptsize $\frac{1}{4}$}}]() at ($(D1)!0.7!(S4)$);
    \coordinate[label={[label distance=-0.75mm]above:{\scriptsize $\frac{1}{8}$}}]() at ($(D1)!0.7!(S5)$);
    \coordinate[label={[label distance=-0.5mm]above:{\scriptsize $\frac{1}{16}$}}]() at ($(D1)!0.7!(S6)$);
  
\end{tikzpicture}
\caption{Matching graph}
\end{subfigure}
\begin{subfigure}[b]{0.63\columnwidth}
\centering
\begin{tikzpicture}[
    thick,
    >=stealth'
  ]
  \coordinate (O) at (0,0);
  \draw[->] (-0.3,0) -- (7.1,0) coordinate[label = {below:$\lambda$}] (xmax);
  \draw[->] (0,-0.3) -- (0,5) coordinate[label = {left:}] (ymax);
  \draw[-] (0,0) -- (1,2);
  \draw[-] (1,2) coordinate (A) -- (2,3) coordinate (B);
  \coordinate (PSI0) at ($(A)!0.5!(B)$);
  \draw[-] (2,3) -- (3,3.5);
  \draw[-] (3,3.5) coordinate (I) -- (4,3.75) coordinate (J);
  \coordinate (EXP) at ($(I)!0.5!(J)$);
  \draw[-] (4,3.75) -- (5,3.875);
  \draw[-] (5,3.875) coordinate (E) -- (6,3.9375) coordinate (F);
  \draw[-] (6,3.9375) -- (7,3.9375) coordinate[label={225:$\Phi(\lambda, \bm{\pi})$}];
  \draw[darkgray,densely dotted] (1.5,0) coordinate[label = {[label distance=0mm,align=center]below:{$\lambda_1=1.5$\\\footnotesize{(global control)}}}](CONTROL) -- (PSI0);
  \coordinate (PSI1) at (5.5, 3.90625);
  \draw[darkgray,densely dotted] (5.5,0) coordinate[label = {[label distance=0mm,align=center]below:{$\lambda_1+\beta=5.5$\\\footnotesize{(global treatment)}}}] (G) -- (PSI1);
  
  \draw[Blue,densely dotted] (PSI0) -- (PSI0 -| 0,0) coordinate[label={[label distance=0mm, align=right]left:{\footnotesize$\Phi(\lambda_1,\bm\pi)=\Psi(0)$\\\footnotesize(control value)}}] (PSI0_AXIS);
  \draw[Blue,densely dotted] (PSI1) -- (PSI1 -| 0,0) coordinate[label={[label distance=0mm,align=right]left:{\footnotesize$\Phi(\lambda_1+\beta,\bm\pi)=\Psi(1)$\\\footnotesize(treatment value)}}] (PSI1_AXIS);
  \draw[Blue, <->] (PSI0_AXIS -| 0.5,0) -- (PSI1_AXIS -| 0.5, 0);
  \coordinate [label={[Blue,label distance=-1mm]left:$\Delta$}](DELTA) at ($(PSI0_AXIS -| 0.5,0)!0.5!(PSI1_AXIS -| 0.5, 0)$);
  
\end{tikzpicture}
\caption{{Matching value as a function of demand}}
\label{fig:example-function}
\end{subfigure}
\caption{{Illustrative bipartite} matching example with a single demand type ($n_d=1$) and six supply types ($n_s=6$). The demand arrival rate is $\lambda_1=1.5$ under global control, and $\lambda_1+\beta=5.5$ under global treatment. The supply arrival rate is 1 for each type. The matching values of the unique demand type {decrease geometrically for each} supply type, creating contention for the highest-value supply units of type 1. {As a result, the matching value function (solid line) is concave in the demand arrival rate; it also has seven linear pieces, with slopes $v_{1,1}, \ldots, v_{1,6}, 0$.}}
\label{fig:example-setup}
\end{figure}
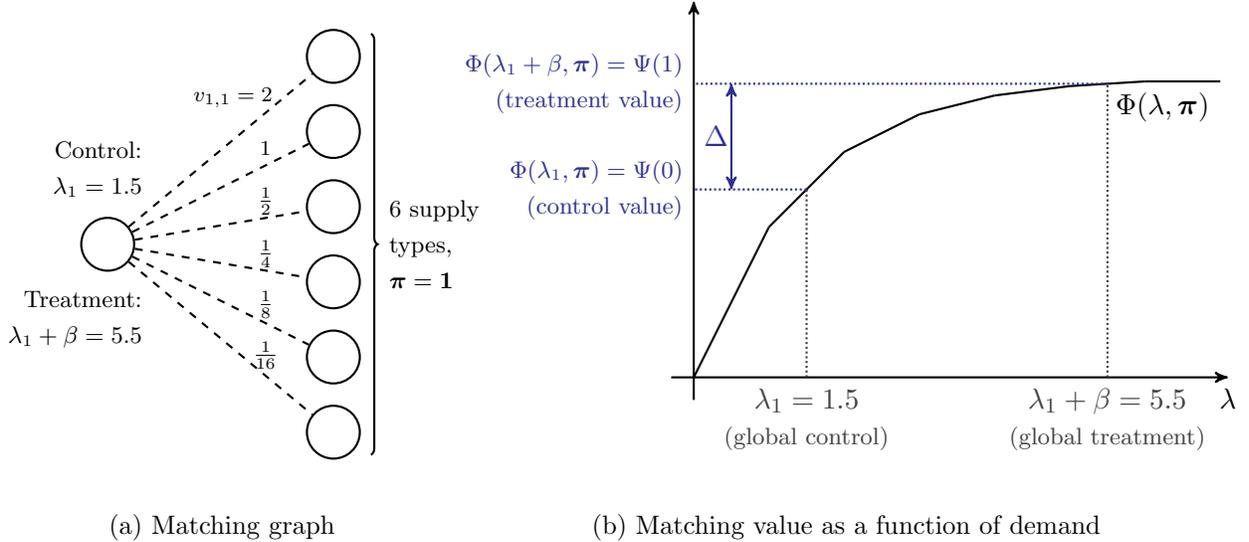

\begin{figure}[ht]
\floatbox[{\capbeside\thisfloatsetup{capbesideposition={right,top},capbesidewidth=4cm}}]{figure}[\FBwidth]
{\caption{Bias of the RCT estimator {in the example from Figure \ref{fig:example-setup}. The estimator implicitly constructs a linear approximation (dashed line) of the value function based on the observed value at the experiment state, where the demand arrival rate is $\lambda_1 + \rho\beta$. Bias occurs because this linear approximation fails to account for concavity.}}\label{fig:rct-bias}}
{\begin{tikzpicture}[
    thick,
    >=stealth'
  ]
  \coordinate (O) at (0,0);
  \draw[->] (-0.3,0) -- (7.1,0) coordinate[label = {below:$\lambda$}] (xmax);
  \draw[->] (0,-0.3) -- (0,6.5) coordinate[label = {left:}] (ymax);
  \draw[-] (0,0) -- (1,2);
  \draw[-] (1,2) coordinate (A) -- (2,3) coordinate (B);
  \coordinate (PSI0) at ($(A)!0.5!(B)$);
  \draw[-] (2,3) -- (3,3.5);
  \draw[-] (3,3.5) coordinate (I) -- (4,3.75) coordinate (J);
  \coordinate (EXP) at ($(I)!0.5!(J)$);
  \draw[-] (4,3.75) -- (5,3.875);
  \draw[-] (5,3.875) coordinate (E) -- (6,3.9375) coordinate (F);
  \draw[-] (6,3.9375) -- (7,3.9375) coordinate[label={225:$\Phi(\lambda, \bm{\pi})$}];
  \draw[darkgray,densely dotted] (1.5,0) coordinate[label = {below:$\lambda_1$}](CONTROL) -- (PSI0);
  \draw[darkgray,densely dotted] (EXP) -- (3.5,0) coordinate[label = {below:$\lambda_1+\rho\beta$}];
  \draw[darkgray,densely dotted] (5.5,0) coordinate[label = {below:$\lambda_1+\beta$}] (G) -- (5.5,6) coordinate (H);
  \coordinate (PSI1) at (intersection of E--F and G--H);
  \draw[Blue,densely dotted] (PSI0) -- (PSI0 -| 0,0) coordinate[label={[label distance=0mm, align=right]left:{$\Phi(\lambda_1,\bm\pi)=\Psi(0)$\\\footnotesize(control actual value)}}] (PSI0_AXIS);
  \draw[Blue,densely dotted] (PSI1) -- (PSI1 -| 0,0) coordinate[label={[label distance=0mm,align=right]left:{$\Phi(\lambda_1+\beta,\bm\pi)=\Psi(1)$\\\footnotesize(treatment actual value)}}] (PSI1_AXIS);
  \draw[Blue, <->] (PSI0_AXIS -| 0.5,0) -- (PSI1_AXIS -| 0.5, 0);
  \coordinate [label={[Blue,label distance=-1mm]left:$\Delta$}](DELTA) at ($(PSI0_AXIS -| 0.5,0)!0.5!(PSI1_AXIS -| 0.5, 0)$);
  \draw[OliveGreen,densely dashed] (0,0) -- ($(EXP)!-3.2cm!(0,0)$) coordinate[label={[OliveGreen]right:$\bar{v}^*\lambda$}] (PROLONG);
  \coordinate (EXP_LOWER) at (intersection of CONTROL--PSI0 and O--PROLONG);
  \draw[OliveGreen, densely dotted] (EXP_LOWER) -- (EXP_LOWER -| 0,0) coordinate[label={[align=right]left:{$\bar{v}^*\lambda_1$\\\footnotesize(control estimated value)}}] (EXP_LOWER_AXIS);
  \coordinate (EXP_UPPER) at (intersection of G--H and O--PROLONG);
  \draw[OliveGreen, densely dotted] (EXP_UPPER) -- (EXP_UPPER -| 0,0) coordinate[label={[align=right]left:{$\bar{v}^*(\lambda_1+\beta)$\\\footnotesize(treatment estimated value)}}] (EXP_UPPER_AXIS);
  \draw[OliveGreen, <->] (EXP_UPPER_AXIS -| 1.0,0) -- (EXP_LOWER_AXIS -| 1.0, 0);
  \coordinate [label={[OliveGreen, label distance = -1mm]right:$\hat{\Delta}_{\text{RCT}}$}](DELTA) at ($(EXP_UPPER_AXIS -| +1.0,0)!0.5!(EXP_LOWER_AXIS -| +1.0, 0)$);
  
\end{tikzpicture}}
\end{figure}
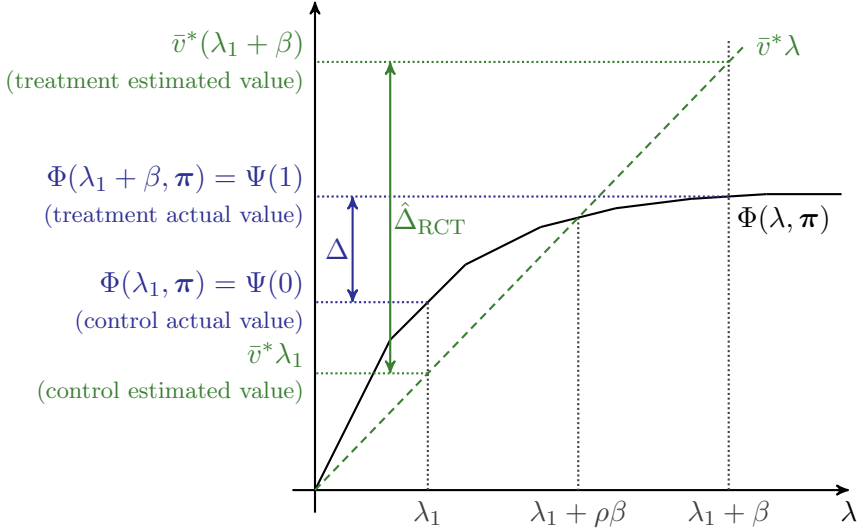

In an experimental setting, we cannot simultaneously evaluate the system value under both global control and global treatment; we \revision{can only} evaluate the value function at an intermediate state where $\lambda=\lambda_1+\rho\beta$ (in this case we assume $\rho=0.5$). Then, according to Proposition~\ref{prop:infinite-horizon-limit-rct}, the RCT estimator constructs a linear approximation of the value function, defined using two points: the observed system value in the experiment state, and the known system value of 0 when demand is 0 (the origin). Figure~\ref{fig:rct-bias} shows this linear approximation, which corresponds to the RCT estimator, as a dashed line. \revision{Clearly,} the dashed line is not a particularly good approximation of the solid line{:} it simultaneously \emph{underestimates} the system value under global control ($\bar{v}^*\lambda_1\le \Psi(0)$) while \emph{overestimating} the system value under global treatment ($\bar{v}^*(\lambda_1+\beta)\ge \Psi(1)$). As a result, the RCT estimator is biased, and significantly so when the concavity of the value function is more pronounced. In the different context of choice-based marketplaces, \cite{li2022interference} show an equivalent result to Theorem~\ref{thm:rct-overestimate} establishing the persistent overestimation behavior of standard estimators.

\section{{A First Alternative to the Standard Estimator}}\label{sec:twolp}

{We now propose a first alternative to the RCT estimator to reduce marketplace interference, which separately estimates the value function in the global control and global treatment states.}

\subsection{{The Two-LP Estimator}}

Bias in the RCT estimator occurs because the estimator's implicit approximation of the value function does not take into account its concavity. We propose an alternative approach that does not require implicitly approximating the value function. \revision{Instead}, we use experiment data to estimate the demand arrival rates $\bm{\lambda}$ under global control and $\bm{\lambda}+\bm{\beta}$ under global treatment, as well as the supply arrival rates $\bm{\pi}$. We then use these estimated arrival rates to directly \revision{estimate} $\Phi(\bm{\lambda}+\bm{\beta}, \bm{\pi})$ and $\Phi(\bm{\lambda}, \bm{\pi})$. Intuitively, we construct \revision{estimated} counterfactual settings for global treatment and global control, then compute the optimal matching in both cases. Because this estimator involves setting up and solving two more linear programs, we refer to it as the \emph{Two-LP} estimator.

\begin{definition}[Two-LP estimator]
For a given scaling parameter $\tau$, we define the Two-LP estimator for the global treatment effect as
\begin{multline*}
    \hat{\Delta}^\tau_{\text{2LP}} =\Phi\left(\hat{\bm{\lambda}}(\bm{D}^{\tau,\controlemph},\bm{D}^{\tau,\treatmentemph} )+\hat{\bm{\beta}}(\bm{D}^{\tau,\controlemph},\bm{D}^{\tau,\treatmentemph} ),\hat{\bm{\pi}}(\bm{S}^{\tau,\bm{\pi}})\right)-\\\Phi\left(\hat{\bm{\lambda}}(\bm{D}^{\tau,\controlemph},\bm{D}^{\tau,\treatmentemph} ),\hat{\bm{\pi}}(\bm{S}^{\tau,\bm{\pi}})\right),
\end{multline*}
where $\hat{\bm{\lambda}}(\cdot)$, $\hat{\bm{\beta}}(\cdot)$ and $\hat{\bm{\pi}}(\cdot)$ designate maximum-likelihood estimators for the arrival rates $\bm{\lambda}$, $\bm{\beta}$ and $\bm{\pi}$ obtained from the observed demand and supply counts.
\end{definition}

The Two-LP estimator results from a two-step approach. First, the experimental data are used to estimate the Poisson arrival rates for demand under global control and under global treatment, as well as the supply arrival rates. The estimated rates, denoted by $\hat{\bm{\lambda}}(\cdot)$, $\hat{\bm{\beta}}(\cdot)$ and $\hat{\bm{\pi}}(\cdot)$ are then used to compute the counterfactual value under global control and under global treatment.

We can simplify \revision{the} Two-LP estimator using the standard \revision{Poisson} maximum-likelihood estimator. Given that the number of \revision{control} demand units of type $i$ is sampled from a Poisson process with parameter $(1-\rho)\lambda_i$, the maximum-likelihood estimators for $\lambda_i$\revision{, $\beta_i$, and $\pi_j$ are} given by 
\[
\hat{\lambda}_i=\frac{D_i^{\tau,\control}}{(1-\rho)\tau},~
\hat{\beta}_i=\frac{D_i^{\tau,\treatment}}{\rho\tau}-\frac{D_i^{\tau,\control}}{(1-\rho)\tau} \text{, and } \hat{\pi}_j=\frac{S_j^{\tau,\pi_j}}{\tau},
\]
\revision{yielding} the following equivalent formulation for the Two-LP estimator.

\begin{proposition}
\label{prop:twolp-simplified}
The Two-LP estimator can equivalently be written as
\[
    \hat{\Delta}^\tau_{\text{2LP}} =\frac{1}{\tau}\left(\Phi_\tau\left(\frac{1}{\rho}\bm{D}^{\tau,\treatmentemph}, \bm{S}^{\tau,\bm{\pi}}\right)-\Phi_\tau\left(\frac{1}{1-\rho}\bm{D}^{\tau,\controlemph},\bm{S}^{\tau,\bm{\pi}}\right)\right).
\]
\end{proposition}

Proposition~\ref{prop:twolp-simplified} provides a simple nonparametric intuition for the Two-LP estimator. Assume the experiment is symmetric ($\rho=0.5$). Then we estimate the total demand in the global control state by doubling the control demand in the experiment state. Similarly, we estimate the total demand in the global treatment state by doubling the treatment demand in the experiment state. 

The Two-LP estimator seems like an ideal approach for the fluid limit problem, where we have a sufficient number of demand and supply units to estimate the Poisson arrival rates for both demand and supply \emph{exactly}. Indeed, we can show this estimator is unbiased in the fluid limit.

\begin{theorem}
    \label{thm:twolp-unbiased}
    In the fluid limit, the Two-LP estimator verifies: $\lim\limits_{\tau\to\infty}\hat{\Delta}^\tau_{\text{2LP}} = \Delta.$
\end{theorem}

\subsection{{Limitations of the Two-LP estimator}}

\paragraph{Theoretical limitations.}
We have established that the Two-LP estimator is asymptotically unbiased in the fluid limit.
However, the fluid limit is an idealized setting \revision{where $\tau$ is large}; it turns out that in a finite-sample setting, the Two-LP estimator does not necessarily remain unbiased. Recall that the global treatment effect we seek to estimate is given by:
\[
\Delta^\tau=\frac{1}{\tau}\mathbb{E}\left[\Phi\left(\bm{D}^{\tau,\bm{\lambda}+\bm{\beta}}, \bm{S}^{\tau,\bm{\pi}}\right)-\Phi\left(\bm{D}^{\tau,\bm{\lambda}},\bm{S}^{\tau,\bm{\pi}}\right)\right].
\]

In contrast, the expected value of the Two-LP estimator can be written as:
\[
\mathbb{E}[\hat{\Delta}^\tau_{\text{2LP}}]=\frac{1}{\tau}\mathbb{E}\left[\Phi\left(\frac{1}{\rho}\bm{D}^{\tau,\rho(\bm{\lambda+\beta})}, \bm{S}^{\tau,\bm{\pi}}\right)-\Phi\left(\frac{1}{1-\rho}\bm{D}^{\tau,(1-\rho)\bm{\lambda}},\bm{S}^{\tau,\bm{\pi}}\right)\right].
\]

The main difference between these two expectations is that the former relies on the Poisson-distributed vector $\bm{D}^{\tau,\bm{\lambda}+\bm{\beta}}$ (respectively $\bm{D}^{\tau,\bm{\lambda}}$) while the latter uses the \emph{rescaled} ---and therefore not Poisson-distributed--- vector $\frac{1}{\rho}\bm{D}^{\tau,\rho(\bm{\lambda+\beta})}$ (respectively $\frac{1}{1-\rho}\bm{D}^{\tau,(1-\rho)\bm{\lambda}}$). \revision{The rescaled demand vector has higher variance due to the $1/\rho$ factor; since the value function $\Phi(\cdot)$ is also nonlinear, the two expectations do not equal each other for finite $\tau$. We next show that for small $\tau$, there exist settings} where the Two-LP estimator exhibits considerable bias.

\begin{proposition}[\revision{Poor} behavior of the Two-LP estimator]
    \label{prop:twolp-worst-case}
    \revision{There exists a matching marketplace experiment where the bias of the Two-LP estimator, defined as $\text{\emph{Bias}}^\tau_{\text{\emph{2LP}}}=\abs{\hat{\Delta}^\tau_{\text{\emph{2LP}}} - \Delta^\tau}$, verifies:}
    \begin{align*}
        \revision{\rho \to 0} ~& \revision{\Rightarrow \text{\emph{Bias}}_{\text{\emph{2LP}}} \to 2\Delta^\tau},\\
        \revision{\rho = \frac{1}{2}} ~& \revision{\Rightarrow \text{\emph{Bias}}_{\text{\emph{2LP}}} = \frac{1}{2}\Delta^\tau},\\
        \revision{\rho \to 1} ~& \revision{\Rightarrow \text{\emph{Bias}}_{\text{\emph{2LP}}} \to \Delta^\tau}.
    \end{align*}
    \revision{In other words, the bias of the two-LP estimator is 50\% if the experiment is symmetric. If the experiment is maximally asymmetric, the bias can reach 100\% on one side and 200\% on the other.}
\end{proposition}

\revision{We can construct such a problem instance by considering} a bipartite matching instance with an equal number of demand and supply types. Each demand unit of type $i$ generates a value of 1 when matched with a supply unit of type $j=i$, and 0 otherwise; the capacity of the edge from the demand node for type $i$ to the supply node for type $j=i$ is also 1. We set low arrival rates such that the Poisson arrival processes in both global control and global treatment are well approximated by Bernoulli processes, i.e., each demand type has either zero or one arrivals. With $\rho=0.5$, the Two-LP estimator estimates the value of global treatment by doubling the number of treated units, and estimates the value of global control by doubling the number of control units. However, doubling the number of units will yield no additional value due to the capacity constraint on each type, whereas doubling the arrival rates (as we would under global control and global treatment) does increase total value to the platform. \revision{Similar flaws arise when $\rho\to 0$ or $\rho\to 1$; we note that the bias is positive in one case and negative in the other, which means the two-LP estimator can get not only the magnitude but also the sign of the intervention wrong.} The capacity constraint is not essential to our construction\revision{; supply could play the constraining role instead, but using capacity simplifies the analysis.} We \revision{illustrate} our construction in Figure~\ref{fig:twolp-failure}. 

\begin{figure}[t]
    \centering
    \includegraphics[width=\columnwidth]{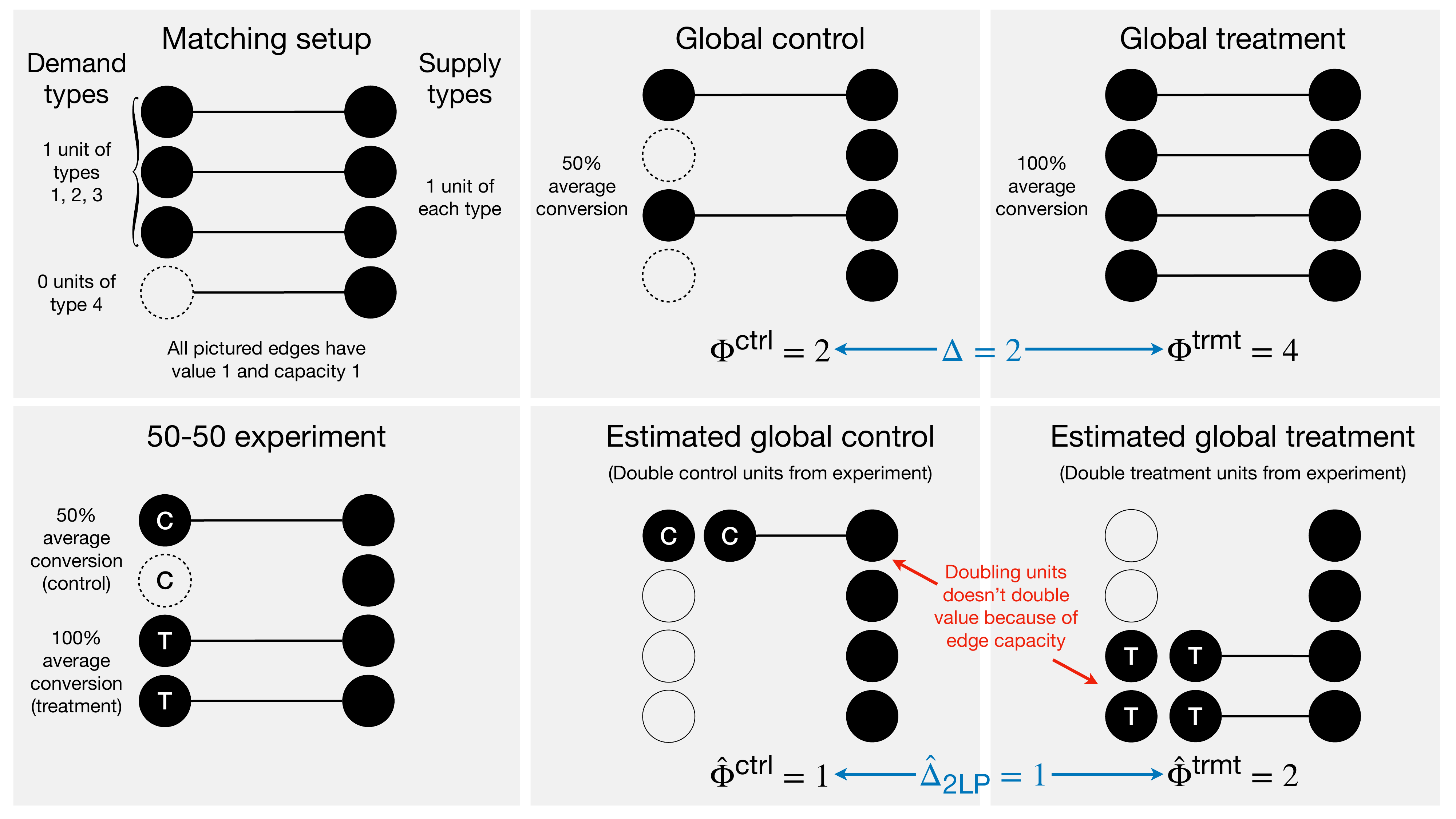}
    \caption{{Example of a situation where the Two-LP estimator fails. We set up a capacitated matching instance where edges between the $i$-th demand and $i$-th supply type have value and capacity one, while all other edges have value zero. Treatment doubles the demand arrival rates, leading to a global treatment effect of 2. However, the Two-LP estimator constructed from a 50-50 experiment doubles the \emph{units} instead of the \emph{rates}, underestimating both control and treatment values (and thus, the global treatment effect) by a factor of two.}}
    \label{fig:twolp-failure}
\end{figure}

\revision{Proposition~\ref{prop:twolp-worst-case} establishes that increased variance due to rescaling Poisson processes together with nonlinearity of the value function can lead to bias in the finite-sample setting. This bias disappears in the large-$\tau$ limit as variance goes to zero. It also increases as the experiment becomes more asymmetric and $1/\rho$ or $1/(1-\rho)$ grow large. We verify this behavior numerically in Section~\ref{sec:simulation}.}

\paragraph{Computational limitations.} In addition to this theoretical limitation in the finite-sample setting, the Two-LP estimator also suffers from practical limitations. First, because it requires solving two additional linear programs, it is more computationally intensive than the standard RCT estimator, which simply adds quantities obtained from the experiment matching linear program. This additional overhead can become cumbersome if we consider multiple matching cycles, or if we wish to perform additional analyses such as permutation tests, which require recomputing test statistics many times. Second, the Two-LP estimator requires access to the matching value function to simulate the counterfactual outcomes under global treatment and global control. Maintaining a simulation environment to perform this analysis might be a significant engineering task, for example \revision{if the} matching \revision{environment relies on} real-time signals that are difficult to replicate offline.

\section{{The Shadow Price Estimator}}
\label{sec:performance}

We now propose another estimator to reduce interference bias. In a way, this estimator interpolates between the RCT and Two-LP estimators. Like the former, it can be computed directly from experimental data, without solving additional optimization problems. Like the latter, it reduces interference bias as compared to the standard estimator.

\subsection{{Estimator Definition}}

    The RCT estimator can be biased in a marketplace setting because it does not account for the diminishing returns of additional demand units in a marketplace with constrained supply. This concept requires a notion of \emph{marginal value} of demand, which we model via LP duality. Recall that in our experiment, we obtain the dual data $(\bm{A}^{\tau,\experiment}, \bm{B}^{\tau,\experiment}, \bm{M}^{\tau,\experiment}, \bm{\Xi}^{\tau,\experiment})$.

In particular, ${A}_i^{\tau,\experiment}$ \revision{is} the marginal value (or shadow price) of a demand unit of type $i$. We can define an analog to the RCT estimator, where instead of summing the total value obtained by the treatment and control groups, we sum the total \emph{marginal} value obtained by each group.

\begin{definition}
\label{def:mdv-definition}
For a given scaling parameter $\tau$, the shadow price (SP) estimator for the global treatment effect is defined as:
\[
    \hat{\Delta}^{\tau}_{\text{\emph{SP}}}=\frac{1}{\tau}\bm{A}^{\tau,\experimentemph}\cdot\left(\frac{1}{\rho}\bm{D}^{\tau,\treatmentemph}-\frac{1}{1-\rho}\bm{D}^{\tau,\controlemph}\right).
\]
\end{definition}

Like the RCT estimator, the SP estimator consists of two terms: the first estimates the value of global treatment, while the second estimates the value of global control. However, rather than simply scaling the values of the control and treatment groups in the experiment, \revision{we} rely on linear programming sensitivity analysis to extrapolate the values of global treatment and global control.

\subsection{Performance {in the Fluid Limit}} \label{ssec:sp-performance-fluid}

In {Section~\ref{ssec:rct-performance}}, we showed that the RCT estimator {implicitly} uses a linear approximation of the value function to estimate the global treatment effect. Using a linear approximation of the value function to estimate the global treatment effect is not a bad idea. The problem with the RCT estimator is that it uses the wrong linear approximation. \revision{We now} argue that the SP estimator uses the correct linear approximation of the partial-treatment value function at the experiment state.

Let $\bm{a}^{\eta}$ denote the optimal {demand} dual variables associated with the partial-treatment matching problem {in the fluid limit} $\Phi(\bm{\lambda}+\eta\bm{\beta}, \bm{\pi})$. We \revision{first} derive \revision{the} fluid limit of the SP estimator.

\begin{proposition}
    \label{prop:mdv-estimator-limit}
    {In the fluid limit,} the SP estimator satisfies:
    \begin{equation}
    \label{eq:mdv-estimator-limit}
    \hat{\Delta}_{\text{SP}}{=\lim_{\tau\to\infty}\hat{\Delta}^\tau_{\text{\emph{SP}}}}=\bm{a}^{\rho}\cdot\bm{\beta}.
    \end{equation}
\end{proposition}

Proposition~\ref{prop:mdv-estimator-limit} establishes that the SP estimator is also based on a linear approximation of the value function. However, this approximation is constructed using first-order information obtained from the shadow prices of the matching problem. The dual variables $\bm a^\rho$ correspond to the marginal values \revision{of additional} demand units \revision{of each type}. \revision{The} inner product with $\bm \beta$ \revision{is} the gain in the objective value predicted by the constraint shadow prices (locally optimal in a Taylor series sense).

\begin{figure}[ht]
    \centering
    \begin{tikzpicture}[
    thick,
    >=stealth'
  ]
  \coordinate (O) at (0,0);
  \draw[->] (-0.3,0) -- (8,0) coordinate[label = {below:$\lambda$}] (xmax);
  \draw[->] (0,-0.3) -- (0,5) coordinate[label = {left:}] (ymax);
  \draw[-] (0,0) -- (1,2);
  \draw[-] (1,2) coordinate (A) -- (2,3) coordinate (B);
  \coordinate (PSI0) at ($(A)!0.5!(B)$);
  \draw[-] (2,3) -- (3,3.5);
  \draw[-] (3,3.5) coordinate (I) -- (4,3.75) coordinate (J);
  \coordinate (EXP) at ($(I)!0.5!(J)$);
  \draw[-] (4,3.75) -- (5,3.875);
  \draw[-] (5,3.875) coordinate (E) -- (6,3.9375) coordinate (F);
  \draw[-] (6,3.9375) -- (7,3.9375) coordinate[label={right:$\Phi(\lambda, \bm{\pi})$}];
  \draw[darkgray,densely dotted] (1.5,0) coordinate[label = {below:$\lambda_1$}](CONTROL) -- ($(PSI0)!-0.8cm!(CONTROL)$) coordinate (Z);
  \draw[darkgray,densely dotted] (EXP) -- (3.5,0) coordinate[label = {below:$\lambda_1+\rho\beta$}];
  \draw[darkgray,densely dotted] (5.5,0) coordinate[label = {below:$\lambda_1+\beta$}] (G) -- (5.5,4.3) coordinate (H);
  \coordinate (PSI1) at (intersection of E--F and G--H);
  \draw[Blue,densely dotted] (PSI0) -- (PSI0 -| 0,0) coordinate[label={[label distance=0mm]left:$\Phi(\lambda_1,\bm\pi)=\Psi(0)$}] (PSI0_AXIS);
  \draw[Blue,densely dotted] (PSI1) -- (PSI1 -| 0,0) coordinate (PSI1_AXIS);
  \coordinate[label={[Blue]left:$\Phi(\lambda_1+\beta,\bm\pi)=\Psi(1)$}](PSI1_AXIS_LABEL) at ($(PSI1_AXIS) - (0,0.1)$);
  \draw[Blue, <->] (PSI0_AXIS -| 0.5,0) -- (PSI1_AXIS -| 0.5, 0);
  \coordinate [label={[Blue,label distance=-1mm]left:$\Delta$}](DELTA) at ($(PSI0_AXIS -| 0.5,0)!0.6!(PSI1_AXIS -| 0.5, 0)$);
  \draw[RedViolet, dashdotted, very thick] ($(I)!-1.8cm!(J)$) coordinate (PROLONG_LEFT) -- ($(I)!2.85cm!(J)$) coordinate[label={[RedViolet]45:$a^\rho\lambda+C$}] (PROLONG_RIGHT);
  \coordinate (EXP_LOWER) at (intersection of CONTROL--Z and PROLONG_LEFT--PROLONG_RIGHT);
  \draw[RedViolet, densely dotted] (EXP_LOWER) -- (EXP_LOWER -| 0,0) coordinate[label={left:$a^\rho\lambda_1 + C$}] (EXP_LOWER_AXIS);
  \coordinate (EXP_UPPER) at (intersection of G--H and PROLONG_LEFT--PROLONG_RIGHT);
  \draw[RedViolet, densely dotted] (EXP_UPPER) -- (EXP_UPPER -| 0,0) coordinate (EXP_UPPER_AXIS);
  \coordinate[label={[RedViolet]left:$a^\rho(\lambda_1+\beta)+C$}](EXP_UPPER_AXIS_LABEL) at ($(EXP_UPPER_AXIS) + (0,0.1)$);
  \draw[RedViolet, <->] (EXP_UPPER_AXIS -| 1.0,0) -- (EXP_LOWER_AXIS -| 1.0, 0);
  \coordinate [label={[RedViolet, label distance = -1mm]right:$\hat{\Delta}_{\text{SP}}$}](DELTA) at ($(EXP_UPPER_AXIS -| +1.0,0)!0.5!(EXP_LOWER_AXIS -| +1.0, 0)$);
  
\end{tikzpicture}
    \caption{Linear approximation corresponding to the SP estimator.}
    \label{fig:sp-example}
\end{figure}
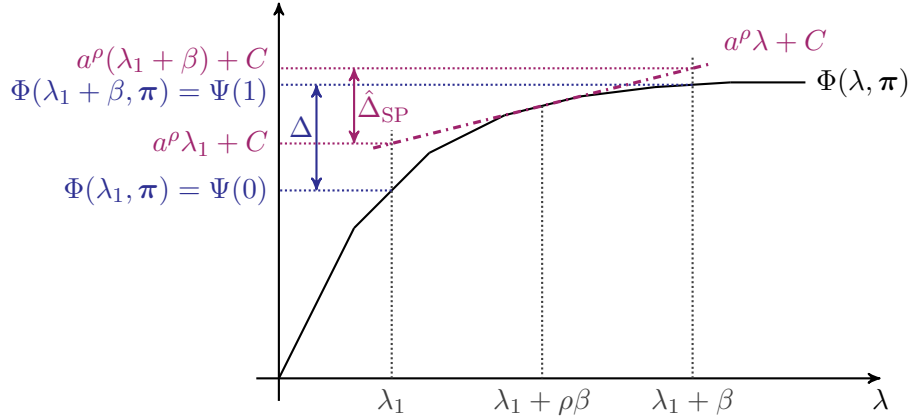

We provide evidence of this result in Figure~\ref{fig:sp-example}. In the same example {from Section~\ref{ssec:rct-performance} (recall Figure~\ref{fig:example-setup} for setup details),} we draw the shadow price-based linear approximation (dashdotted line) and the resulting SP estimator. Like the RCT estimator, the SP estimator tends to overestimate the value of global treatment. However, due to the concavity of the value function, it \emph{also} overestimates the value of global control. Depending on the relative magnitude of these two biases, the SP estimator can over- or underestimate the treatment effect. {We next show how the SP estimator reduces bias (in absolute terms) as compared to the RCT estimator.}

\paragraph{Bias reduction.}
Both the RCT and SP estimators rely on the implicit construction of a linear approximation of the partial-treatment value function $\Psi(\cdot)$. The main difference is that the RCT estimator determines the slope of its implicit approximation using the average value $\bar{v}^*_i$ obtained from each type $i$, whereas the SP estimator uses the marginal value obtained from type $i$, as captured by the shadow price $a^\rho_i$. It turns out there is a fundamental reason why $a^\rho_i$ is a useful quantity in estimating the global treatment effect $\Delta$. Recalling from Proposition~\ref{prop:concave-integrable} that $\Psi(\cdot)$ is a piecewise linear function with finitely many pieces, which implies it is Riemann-integrable, we can \revision{rewrite the global treatment effect $\Delta$ using the fundamental theorem of calculus.}

\begin{proposition}[Fundamental theorem of calculus.]
\label{prop:fundamental-theorem-calculus}
Let $\Psi'(\eta)=\bm{a}^\eta\cdot\bm{\beta}$ designate the derivative of the partial-treatment value function $\Psi(\eta)$, defined on all but finitely many points of $[0,1]$. Then we can write the global treatment effect $\Delta$ as 
\begin{equation}
    \label{eq:fundamental-theorem-calculus}
    \Delta=\Psi(1)-\Psi(0)=\int_0^1\bm{a}^\eta\cdot\bm{\beta}d\eta.
\end{equation}
\end{proposition}

Using this framework, we can directly compare the bias of the RCT and SP estimators, by comparing the true partial-treatment value function $\Psi(\cdot)$ and the respective linear approximations $\hat{\Psi}_{\text{RCT}}$ and $\hat{\Psi}_{\text{SP}}$ implicitly constructed by the RCT and SP estimators. In particular, we can now prove {a core} result of this paper: \revision{under our assumptions,} it is always possible to design an experiment such that the SP estimator is less biased than the RCT estimator.

\begin{theorem}
\label{thm:main}
 Let $\text{\emph{Bias}}_{\text{\emph{RCT}}}=|\hat{\Delta}_{\text{\emph{RCT}}}-\Delta|$ and $\text{\emph{Bias}}_{\text{\emph{SP}}}=|\hat{\Delta}_{\text{\emph{SP}}}-\Delta|$. If the experiment is symmetric, then $\text{\emph{Bias}}_{\text{\emph{SP}}} \leq \text{\emph{Bias}}_{\text{\emph{RCT}}}$.
\end{theorem}

\begin{figure}[t]
    \centering
    \input{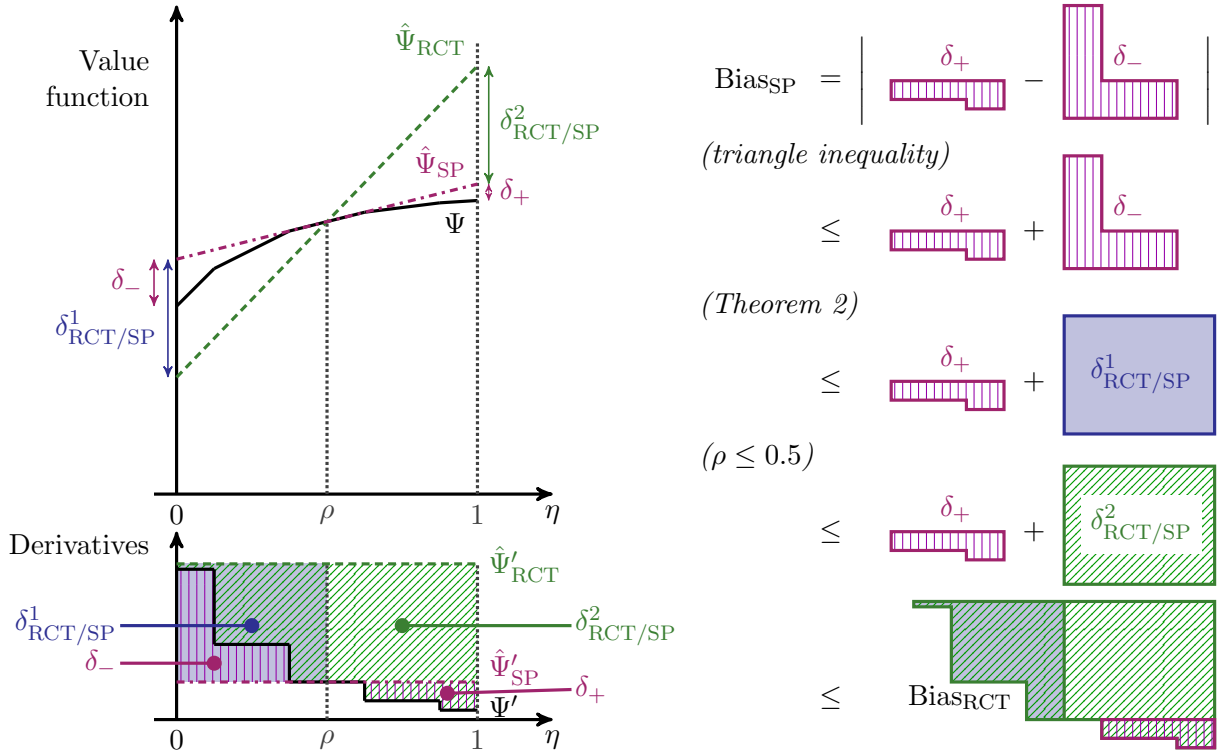}
    \caption{Proof by picture of Theorem~\ref{thm:main}, in the case where $\bm\beta\ge 0$. The absolute RCT estimator bias is given by the green, diagonally hashed area. The absolute SP estimator bias is bounded above by the purple, vertically hashed areas. We first note that the L-shaped purple area on the left, denoted $\delta_-$, is bounded by the area of the solid blue rectangle bounded by vertical lines at $\eta=0$ and $\eta=\rho$, denoted by $\delta^1_{\text{RCT/SP}}$, which is itself bounded by the area $\delta^2_{\text{RCT/SP}}$ of the green diagonally hashed rectangle on the upper {right} as long as $\rho \le 0.5$.}
    \label{fig:main-proof}
\end{figure}

The proof of Theorem~\ref{thm:main} separately considers the case of a consistently positive treatment effect ($\bm\beta\ge 0$) and a consistently negative treatment effect ($\bm\beta\le0$). We summarize the proof of the former case graphically in Figure~\ref{fig:main-proof}, noting that our reasoning is analogous in the latter case. The proof relies on exposing no more than half the population to treatment, i.e., $\rho\le0.5$. Conversely, when the treatment effect is consistently negative, we should expose no less than half the population to treatment. Symmetric experiments ($\rho=0.5$) satisfy both constraints, which is useful since we typically do not know a priori whether the treatment effect is consistently positive or negative. {Symmetric experiments are often used in practice because they tend to minimize variance; however, in some situations, rolling a treatment out to half of the marketplace may be too costly. We discuss this issue further in Section~\ref{ssec:beyond-symmetric} and propose \revision{two} simple methods to address it.}

\paragraph{Variance reduction.}
The SP estimator also improves upon the RCT estimator in terms of variance. The following result compares the asymptotic variances of the two estimators.

\begin{theorem}
    \label{thm:mdv-variance}
    \revision{Under Assumptions~\ref{ass:blind} and \ref{ass:unique}}, the asymptotic variance of the SP estimator is no larger than the variance of the RCT estimator, i.e.,
    \begin{equation}
    \label{eq:mdv-estimator-variance}
     \lim_{\tau\to\infty}\Var{\sqrt{\tau}\hat{\Delta}_{\mathrm{SP}}^\tau} \le  \lim_{\tau\to\infty}\Var{\sqrt{\tau}\hat{\Delta}_{\mathrm{RCT}}^\tau}.
    \end{equation}
\end{theorem}

The variance of the SP estimator is tricky to analyze {because the breakpoints in the piecewise linear matching function can make it difficult to estimate the second derivative of the partial-treatment value function.} However, in the fluid limit, the probability that the system operates at a breakpoint in the partial-treatment value function goes to 0. In this setting, the RCT {(respectively SP) estimator is the product of the average (respectively marginal) value of a type with the increase in demand for that type. Because the average value is always greater than or equal to the marginal value, the RCT estimator is more sensitive to variance in demand.}

\revision{U}nlike our characterization of bias, Theorem~\ref{thm:mdv-variance} does not require sign-consistency of the treatment effect on demand $\bm\beta$. When analyzing bias, sign-consistency \revision{precludes} the case where the RCT estimator overestimates both the positive effect from increased demand of one type and the negative effect from decreased demand of another type, in such a way that the biases cancel. When it comes to variance, however, overestimating the magnitude of either a positive or negative effect will result in the same contribution to overall variance. Theorem~\ref{thm:mdv-variance} suggests that the bias correction obtained via the shadow price estimator does not come at the expense of low variance.

\subsection{{Finite-sample Performance}} \label{ssec:finite-sample}

{
\revision{The} SP estimator reduces bias as compared to the standard RCT estimator in the fluid limit. However, unlike the Two-LP estimator, it \revision{may} not eliminate interference bias altogether in this setting. We now explore why the SP estimator may \revision{still} be preferable to the Two-LP estimator.

In particular, performance in the fluid limit does not tell the whole story. Indeed, Proposition~\ref{prop:twolp-worst-case} shows that the Two-LP estimator can be strongly biased in the finite-sample regime. It is therefore of interest to study the behavior of the SP estimator in the finite-sample case. We first define the generalization of the partial-treatment function to the finite-sample setting:

\begin{definition}
    Given a matching experiment with supply arrival rate $\bm{\pi}$, demand arrival rate of $\bm{\lambda}$ under global control, and demand arrival rate $\bm{\lambda}+\bm{\beta}$ under global treatment, for $\eta\in[0,1]$, the finite-sample partial-treatment value function is given by $$\Psi_\tau(\eta)=\mathbb{E}\left[\revision{\frac{1}{\tau}}\Phi_\tau(\bm{D}^{\tau,\bm{\lambda}+\eta\bm{\beta}}, \bm{S}^{\tau,\bm{\pi}})\right].$$
\end{definition}

By construction, the scaled finite-sample partial-treatment value function $\Psi_{\tau}(\eta)$ tends to the partial-treatment value function $\Psi(\eta)$, i.e. $\Psi(\eta)=\lim_{\tau\to\infty}\Psi_\tau(\eta)$. Like $\Psi(\cdot)$, $\Psi_\tau(\cdot)$ is concave and continuous. However, while $\Psi(\cdot)$ is piecewise linear, $\Psi_\tau(\cdot)$ is smooth -- its derivative is uniquely defined everywhere. In the finite-sample setting, the \revision{random} arrivals on the supply and demand side effectively ``smooth out'' the breakpoints in the underlying value function.
Once again, the global treatment effect verifies $\Delta^\tau=\Psi_{\tau}(1)-\Psi_{\tau}(0)$. \revision{We now establish} that, as in the fluid limit, the SP estimator results from a first-order approximation of the partial-treatment value function.

\begin{theorem}
    \label{thm:shadow-price-derivative}
    For any $\tau>0$, we have that
    \begin{equation}
        \mathbb{E}\left[\hat{\Delta}^\tau_{\text{SP}}\right]=\frac{\partial}{\partial\eta}\mathbb{E}\left[\revision{\frac{1}{\tau}}\Phi_\tau(\bm{D}^{\tau,\bm{\lambda}+\eta\bm{\beta}}, \bm{S}^{\tau,\bm{\pi}})\right]\Bigg\rvert_{\eta=\rho}=\Psi'_{\tau}(\rho)
    \end{equation}
\end{theorem}
}

{
Theorem~\ref{thm:shadow-price-derivative} states that the shadow price estimator correctly estimates the derivative of the partial-treatment value function at the experiment state ($\eta=\rho$). 
This principled behavior immediately allows us to bound the expected bias of the shadow price estimator as follows.

\begin{corollary}
Assume that the partial-treatment value function is twice differentiable with its second derivative bounded by $M$ on $[0,1]$, i.e., $\abs{\Psi''_{\tau}(\eta)}\le M,~\forall\eta\in[0,1].$
Then the expected bias of the shadow price estimator is bounded above by:
\[
\abs{\mathbb{E}\left[\hat{\Delta}^\tau_{\text{SP}}\right]-\Delta^\tau}\le \frac{M}{2}.
\]
\end{corollary}

Superficially, Theorem~\ref{thm:shadow-price-derivative} seems like an intuitive result: taking the difference in total marginal value between the treatment and control group yields the ``correct'' derivative of the partial-treatment value function. However, the proof is nontrivial, and relies on the particular form of the Poisson probability mass function, as well as the dual structure of the generalized matching problem. The proof also relies on the following result, which might be of independent interest.}

\begin{theorem}[Local linearity of matching function]\label{thm:directional-derivative}
Consider the {generalized} marketplace matching problem for some $\tau$ with demand $\bm{d}\in\mathbb{Z}_+^{n_d}$ and supply $\bm{s}\in\mathbb{Z}_+^{n_s}$. {Further assume that all capacities are integer: $k_{u,w}\tau\in\mathbb{Z}$ for each edge $(u,w)$.} Consider a demand perturbation $\bm{\varepsilon}\in\mathbb{R}_+^{n_d}$ such that $\sum_{i=1}^{n_d}\varepsilon_i\le 1$, and let $\bm{e}_i\in\mathbb{R}^{n_d}$ designate the $i$-th unit vector. Then the following two linearity results hold for any $\tau$:
\begin{equation}
    \Phi_\tau\left(\bm{d}+\bm{\varepsilon},\bm{s}\right)=\Phi_\tau\left(\bm{d},\bm{s}\right)+\sum_{i=1}^{n_d}\varepsilon_i\left[\Phi_\tau\left(\bm{d}+\bm{e}_i,\bm{s}\right) - \Phi_\tau\left(\bm{d},\bm{s}\right)\right].\label{eq:linear-positive}
\end{equation}
\begin{equation}
    \Phi_\tau\left(\bm{d}-\bm{\varepsilon},\bm{s}\right)=\Phi_\tau\left(\bm{d},\bm{s}\right)-\sum_{i=1}^{n_d}\varepsilon_i\left[\Phi_\tau\left(\bm{d},\bm{s}\right) - \Phi_\tau\left(\bm{d}-\bm{e}_i,\bm{s}\right)\right].\label{eq:linear-negative}
\end{equation}
\end{theorem}

{
Proposition~\ref{prop:concave-integrable} (LP sensitivity analysis) tells us that the function $\Phi_\tau(\cdot,\cdot)$ is concave and piecewise linear. Theorem~\ref{thm:directional-derivative} further tells us that for any particular supply/demand configuration, the matching value function is linear in any sufficiently small perturbation in demand. The result allows us to define subgradients of the value function $\bm{a}^+$ and $\bm{a}^-$, such that:
\begin{align}
    a_i^+ &= \Phi_\tau\left(\bm{d}+\bm{e}_i,\bm{s}\right) - \Phi_\tau\left(\bm{d},\bm{s}\right)\\
    a_i^- &= \Phi_\tau\left(\bm{d},\bm{s}\right) - \Phi_\tau\left(\bm{d}-\bm{e}_i,\bm{s}\right) \label{eq:downward-dual}.
\end{align}

These two subgradients are identical \revision{when} the matching value function is locally linear, and the optimal shadow prices are unique ($\bm{a}^+=\bm{a}^-=\bm{a}$). \revision{This corresponds to Assumption~\ref{ass:unique}, which is reasonable in the fluid limit} since the rates $\bm{\lambda}$, $\bm{\beta}$ and $\bm{\pi}$ are real numbers. However, in the finite-sample setting, the optimal solution may be degenerate, i.e, the optimal dual solution may not be unique. In this case, we can choose any optimal dual solution. Two choices are $a_i^+$ and $a_i^-$: the marginal value of \emph{adding} a unit of demand of type $i$, and the marginal value of \emph{removing} a unit of demand of type $i$. \revision{Intuitively, we} are more interested in the latter, therefore, when the optimal shadow prices are not unique, we select $a_i=a_i^-$. \revision{This also turns out to be the right choice to ensure Theorem~\ref{thm:shadow-price-derivative} holds.} We discuss how to compute $\bm{a}^+$ and $\bm{a}^-$ in Section~\ref{sec:simulation}.
}

\subsection{Beyond Symmetric Experiments} \label{ssec:beyond-symmetric}

{So far, we have established that the SP estimator reduces bias from the RCT estimator in the fluid limit, and remains well-behaved in the finite-sample regime. However, \revision{bias reduction} depends on a symmetric experimental design, i.e., where $\rho=0.5$. This assumption is limiting: in practice, platforms may want to roll out a potentially expensive or risky treatment to a much smaller fraction of the market. We now explore ways to \revision{relax} this assumption. We begin with a counter-example that shows what happens in the fluid limit when the symmetric design is violated.}

\begin{proposition}
    \label{prop:nonsymmetric-breaking}
    If the experiment is not symmetric ($\rho\neq0.5$), it is possible to construct an example where $\text{\emph{Bias}}_{\text{\emph{SP}}} > \text{\emph{Bias}}_{\text{\emph{RCT}}}$.
\end{proposition}

We can prove Proposition~\ref{prop:nonsymmetric-breaking} by constructing such an example. We consider a single demand type ($n_d=1$) and supply type ($n_s=1$), verifying $\lambda_1=0$ (demand in global control), $\beta_1=1$ (demand in global treatment), and $\pi_1=0.625$ (supply level). We obtain value $v_{1,1}=1$ from each supply-demand match. We \revision{let} $\rho=0.75$. The example is \revision{shown} in Figure~\ref{fig:counter-example}. 

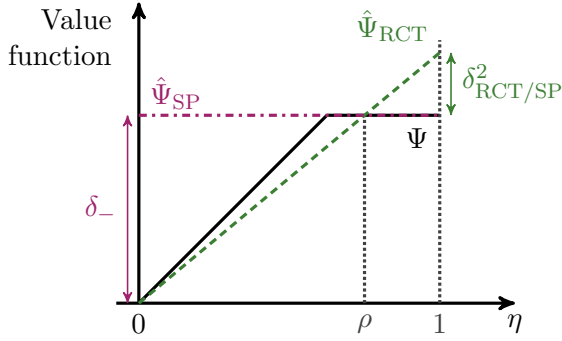
\begin{figure}
\floatbox[{\capbeside\thisfloatsetup{capbesideposition={right,top},capbesidewidth=7.2cm}}]{figure}[\FBwidth]
{\caption{{Toy example of an asymmetric experiment where the SP estimator does not reduce bias. T}he SP estimator is unbiased for global treatment but overestimates global control by $\delta_-$. {T}he RCT estimator is unbiased for global control but overestimates global treatment by $\delta^2_{\text{RCT/SP}}$. Due to the asymmetric design, $\delta^2_{\text{RCT/SP}}<\delta_-$.
}\label{fig:counter-example}}
{\begin{tikzpicture}[
    very thick,
    >=stealth'
  ]
  \node[anchor=east] at (0,6.5) {\begin{tabular}{r} Value\\function \end{tabular}};
  \coordinate [label={below:0}] (O) at (0,3);
  \draw[->] ($(O) + (-0.3,0)$) -- ($(O) + (5,0)$) coordinate[label = {below:$\eta$}] (xmax);
  \draw[->] (O) -- ($(O) + (0,4)$) coordinate[label = {left:}];
  \draw[-] ($(O) + (0,0)$) -- ($(O) + (2.5,2.5)$);
  \draw[-] ($(O) + (2.5,2.5)$) -- ($(O) + (4,2.5)$);
  \coordinate[label={left:$\Psi$}] (PSI_DEF) at ($(O) + (4,2.2)$);
  \draw[darkgray,densely dotted] ($(O) + (4,0)$) coordinate[label = {below:1}] (TREATMENT_BASE) -- ($(O) + (4,3.5)$) coordinate (TREATMENT_VALUE);
  \coordinate (EXP_VALUE) at ($(O) + (3,2.5)$);
  \draw[darkgray,densely dotted] ($(O) + (3,0)$) coordinate[label = {below:$\rho$}] (EXP_BASE) -- (EXP_VALUE);
  \coordinate (RCT_CONTROL) at ($(O) + (0, 0)$);
  \coordinate[label={[OliveGreen]above left:$\hat{\Psi}_{\text{RCT}}$}] (RCT_TREATMENT) at ($(O) + (4, 3.333)$);
  \draw[OliveGreen,densely dashed] (RCT_CONTROL) -- (RCT_TREATMENT);
  \coordinate (SP_CONTROL) at ($(O) + (0, 2.5)$);
  \coordinate (SP_TREATMENT) at ($(O) + (4, 2.5)$);
  \draw[RedViolet, dashdotted, very thick] (SP_CONTROL) -- (SP_TREATMENT);
  \coordinate[label={[RedViolet] right:$\hat{\Psi}_{\text{SP}}$}] (PSI_SP_DEF) at ($(SP_CONTROL) + (0,0.3)$);
  \coordinate (DELTA_MINUS_ARROW_BOTTOM) at ($(O) + (-0.15,0)$);
  \coordinate (DELTA_MINUS_ARROW_TOP) at ($(SP_CONTROL) + (-0.15, 0)$);
  \draw[semithick,RedViolet,<->] (DELTA_MINUS_ARROW_BOTTOM) -- (DELTA_MINUS_ARROW_TOP);
  \coordinate[label={[RedViolet]left:$\delta_-$}] (DELTA_MINUS_LABEL) at ($(DELTA_MINUS_ARROW_TOP)!0.5!(DELTA_MINUS_ARROW_BOTTOM)$);
  \coordinate (DELTA_RCTSP2_ARROW_BOTTOM) at ($(SP_TREATMENT) + (0.15,0)$);
  \coordinate (DELTA_RCTSP2_ARROW_TOP) at ($(RCT_TREATMENT) + (0.15, 0)$);
  \draw[semithick,OliveGreen,<->] (DELTA_RCTSP2_ARROW_BOTTOM) -- (DELTA_RCTSP2_ARROW_TOP);
  \coordinate[label={[OliveGreen]right:$\delta^2_{\text{RCT/SP}}$}] (DELTA_RCTSP2_LABEL) at ($(DELTA_RCTSP2_ARROW_TOP)!0.5!(DELTA_RCTSP2_ARROW_BOTTOM)$);

\end{tikzpicture}}
\end{figure}

In this example the SP estimator is unbiased for global treatment but overestimates global control. Meanwhile the RCT estimator is unbiased for global control but overestimates global treatment. Due to the asymmetric design, overestimating global control turns out to be ``worse'' than overestimating global treatment. \revision{We note that this counter-example} {requires} a significant discontinuity in the partial-treatment value function.

For a fixed treatment fraction $\rho\neq0.5$, we can construct examples where the SP estimator performs poorly. \revision{An initial mechanism to guard against such realizations is to randomize $\rho$.}

\begin{theorem}
    \label{thm:zero-bias-random}
    Consider an experiment where the treatment fraction \revision{$\rho$} is selected uniformly at random in $[0,1]$. Then the SP estimator is unbiased in expectation.
\end{theorem}

Theorem~\ref{thm:zero-bias-random} relies on the simple observation that taking the expectation of the SP estimator in the fluid limit means integrating $\hat{\Delta}_{\text{RCT}}=\bm{a}^{\rho}\cdot\bm{\beta}$ over the interval $[0,1]$, which is precisely the definition of the true global treatment effect given by Proposition~\ref{prop:fundamental-theorem-calculus}. {The same result applies in the finite-sample case using Theorem~\ref{thm:shadow-price-derivative}.} 
\revision{U}nlike Theorem~\ref{thm:main}, \revision{this result} does not rely on sign-consistency of the treatment effect $\bm\beta$. However, its \revision{usefulness} may be limited, since selecting the treatment fraction at random has other drawbacks, both theoretical ({increased} variance) and practical ({incompatible with a fixed experiment budget}).

{In practice, experimenters may wish to expose only a very small fraction $\rho \ll 0.5$ of the marketplace to the treatment, particularly if the treatment is expensive or the experiment has a finite budget. Fortunately, if the treatment effect on demand is positive, i.e., $\bm{\beta}\ge \bm{0}$, the result from Theorem~\ref{thm:main} extends to any $\rho \le 0.5$. However, if $\bm{\beta}\le\bm{0}$, and we correspondingly observe a negative value of $\hat{\Delta}_{\text{SP}}$, there is a possibility (outlined in Proposition~\ref{prop:nonsymmetric-breaking}) that the shadow price estimator is more biased than the standard estimator.}

{One way to achieve our bias reduction objective is to construct a new estimator which is a convex combination of the shadow price and standard estimators. By combining the two estimators, we can ensure that we always reduce bias as compared to the standard estimator.}

{
\begin{definition}
    The shadow-price-plus (SP+) estimator is defined as
    \[
    \hat{\Delta}_{\text{\emph{SP+}}}=(1-2\rho)\,\hat{\Delta}_{\text{\emph{RCT}}}+2\rho\,\hat{\Delta}_{\text{\emph{SP}}}.
    \]
\end{definition}
}

{
By inspecting the definition, we see that the SP+ estimator behaves more like the SP estimator when the treatment fraction $\rho$ tends to 0.5, and more like the RCT estimator when the treatment fraction $\rho$ tends to zero. This behavior allows the SP+ estimator to leverage the shadow price estimator's bias reduction without succumbing to pathological cases like the one described in Proposition~\ref{prop:nonsymmetric-breaking}. We formalize this intuition in Theorem~\ref{thm:shadow-price-plus}.

\begin{theorem}
    \label{thm:shadow-price-plus}
    When $\rho < 0.5$ and $\bm{\beta}\le \bm{0}$, the SP+ estimator reduces bias from the RCT estimator:
    \[
    \abs{\hat{\Delta}_{\text{\emph{SP+}}}-\Delta}\le \abs{\hat{\Delta}_{\text{\emph{RCT}}}-\Delta}.
    \]
\end{theorem}

Theorem~\ref{thm:shadow-price-plus} establishes that the SP+ estimator always guarantees a reduction in bias. When we observe a negative treatment effect with $\rho<0.5$, we can therefore use the SP+ estimator instead of the SP estimator. If we observe a positive treatment effect, then we can stick with the SP estimator. We can also construct a symmetric version of the SP+ estimator which guarantees bias reduction when $\rho > 0.5$ and $\bm{\beta}\ge \bm{0}$, i.e.,
\[
    \hat{\Delta}_{\text{\emph{SP+}}}=(2\rho-1)\,\hat{\Delta}_{\text{\emph{RCT}}}+2(1-\rho)\,\hat{\Delta}_{\text{\emph{SP}}}.
\]
We omit the analogous statement of Theorem~\ref{thm:shadow-price-plus} (which trivially holds by symmetry) to this case, for brevity and because exposing more than half of customers to treatment is practically unrealistic.
}

\revision{Theorems~\ref{thm:zero-bias-random} and \ref{thm:shadow-price-plus} provide provable ways to reduce bias for asymmetric experiments. However, we will see numerically that the bias of the shadow price estimator is usually much lower than the bias of the RCT estimator, even in asymmetric experiments. Indeed, the counter-example presented in Proposition~\ref{prop:nonsymmetric-breaking} is pathological --- in global control, there is no interference to speak of, while in global treatment, there is so much interference that the partial-treatment value function in Figure~\ref{fig:counter-example} is flat. In practice, the interventions that platforms experiment with are almost always marginal. It is therefore unlikely that the intervention itself is fundamentally altering the structure of interference in the marketplace. We can therefore introduce the following result guaranteeing bias reduction for any treatment fraction.

\begin{theorem}
    \label{thm:conditions-bias-reduction}
    For any treatment fraction $\rho$, we can guarantee that $\text{\emph{Bias}}_{\text{\emph{SP}}} \leq \text{\emph{Bias}}_{\text{\emph{RCT}}}$ as long as the following condition holds:
    \begin{equation}
        \label{eq:bias-reduction}
        (\bar{\bm{v}}^* - \bm{a}^0)\cdot \bm\beta \ge (\bm{a}^0 - \bm{a}^1)\cdot \bm\beta.
    \end{equation}
    Furthermore, $\text{\emph{Bias}}_{\text{\emph{SP}}} = 0$ as long as:
    \begin{equation}
        \label{eq:bias-elimination}
        (\bm{a}^0 - \bm{a}^1)\cdot \bm\beta = 0.
    \end{equation}
\end{theorem}

The first part of Theorem~\ref{thm:conditions-bias-reduction} guarantees that the shadow price estimator reduces bias as long as the left-hand side of~\eqref{eq:bias-reduction} is larger than the right-hand side. The left-hand side is large if there is significant interference in the marketplace in global control, namely, average values are larger than marginal values. The right-hand side is small if the structure of interference is similar in global control and global treatment. These conditions are often satisfied in practice: indeed, if it is the case that interference is very weak in global control and very strong in global treatment, a single experiment is unlikely to be the right tool to assess the impact of the proposed intervention. The second part of Theorem~\ref{thm:conditions-bias-reduction} also guarantees that the shadow price estimator is unbiased as long as the interference pattern is exactly the same in global control and global treatment.

Unlike Theorems~\ref{thm:zero-bias-random} and \ref{thm:shadow-price-plus}, the conditions described in Theorem~\ref{thm:conditions-bias-reduction} are not verifiable from the experiment state, as they require knowledge of the shadow prices under global treatment and global control. However, as we see in Section~\ref{sec:simulation}, they often hold in practice.

}

\section{Numerical Results}

{We now validate our theoretical results on estimator performance using real and simulated data. We show that we can improve performance as compared to the standard estimator, even when we depart from some of our theoretical assumptions.}

\subsection{Ride-hailing}

{We first consider a ride-hailing setting simulated using data from the \cite{NYCdata}, which catalogs roughly 10 million origin-destination pairs from June 2023 taxi rides in New York City. From this large dataset, we can construct simulated problem instances as follows: we independently sample $n_d$ origins, representing \emph{ride requests} and $n_s$ destinations, representing \emph{drivers} (assuming drivers remain in the area where they dropped off their last passenger). Ride requests and drivers are randomly assigned a time within 24 hours. Drivers can be matched to a request if it appears within 15 minutes of the driver appearing on the platform. Riders must be matched to drivers already on the platform. The value of matching request $i$ to driver $j$ is given by a notion of efficiency: $v_{i,j}=d^{\text{ride}}_i - d^{\text{pickup}}_{i,j}$, where $d_i^\text{ride}$ is the length of ride $i$ in miles (from the data) and $d^\text{pickup}_{i,j}$ is the (Haversine) pickup distance from driver $j$ to request $i$. We assume that the platform does not need to fulfill rides with negative efficiency. For tractability, we simplify the graph to only allow matches between a request and its 40 closest drivers.}

\begin{figure}[ht]
    \begin{subfigure}[t]{0.57\columnwidth}
        \includegraphics[width=\columnwidth]{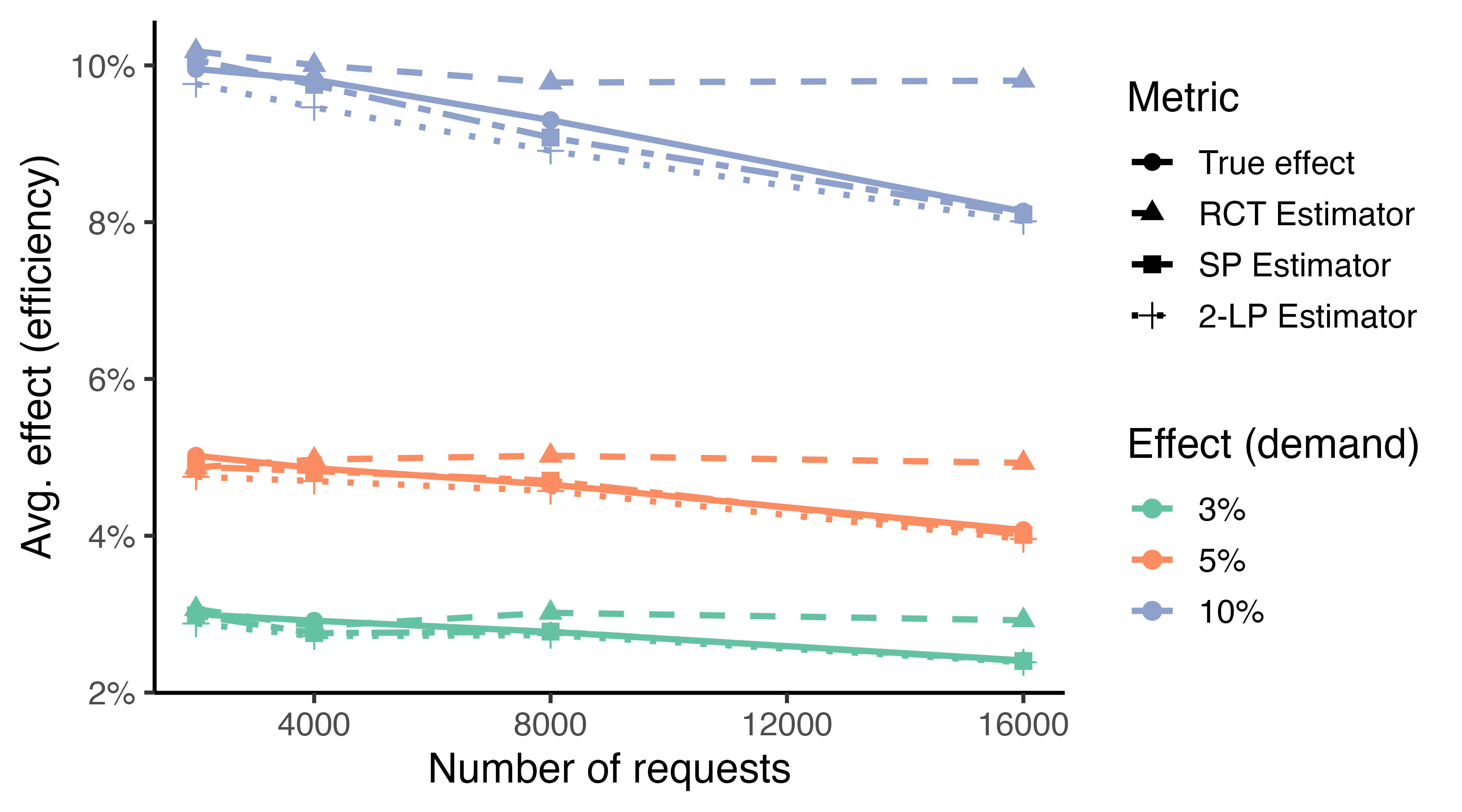}
        \caption{Mean}
        \label{fig:sim_results}
    \end{subfigure}%
    \begin{subfigure}[t]{0.42\columnwidth}
        \includegraphics[width=\columnwidth]{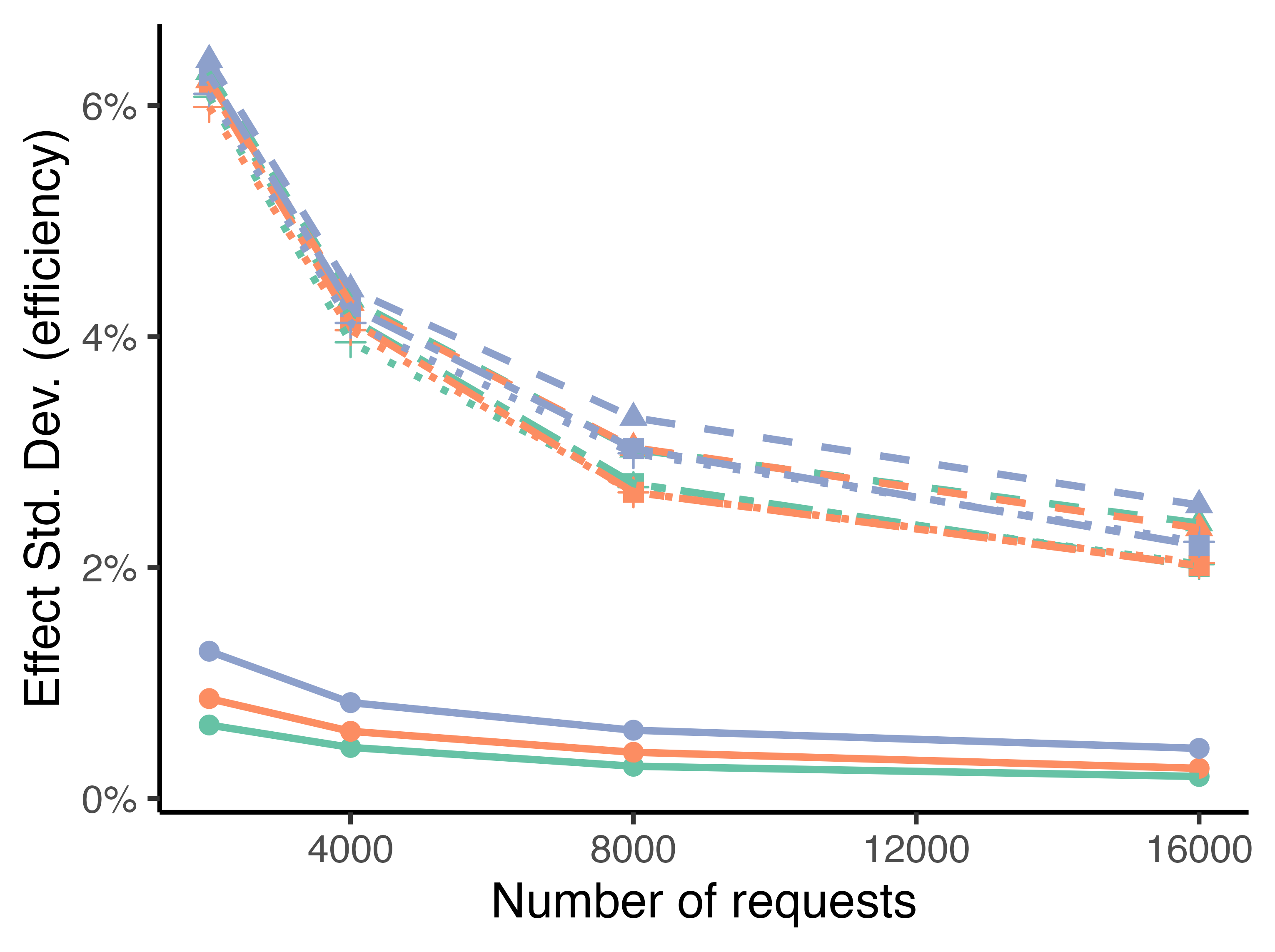}
        \caption{Standard deviation}
        \label{fig:sim_results_std}
    \end{subfigure}
    \centering
    \caption{{Estimated effect on total system value (efficiency) of treatments of different magnitudes in various supply-demand configurations.
    }
    }
\end{figure}

{This first simulation environment is designed to stretch the limits of our discrete-type assumption. Because we consider each ride request and driver to be its own distinct type (due to its unique location in space and time), there is always exactly one unit of each type. While this simulation allows us to test our method on a real dataset, one downside of this explosion of types is that it leads to primal degeneracy:} since the optimal primal solution is binary, every variable $x_{i,j}^*=1$ simultaneously satisfies both a demand and a supply capacity constraint to equality. {We discussed how to handle multiple duals in Section~\ref{ssec:finite-sample}, by choosing the ``lower sugradient'' $a_i^-$. Thanks to Theorem~\ref{thm:directional-derivative}, \revision{all these lower subgradients} can be computed efficiently by reducing the value of the right-hand side of each demand constraint by $0<\varepsilon<1/n_d$.}

\label{sec:simulation}

{We consider twelve experimental settings, with four levels of contention for supply \revision{(2000, 4000, 8000, and 16000 requests, against 16000 drivers)} and three magnitudes of the treatment effect on the arrival rate of demand (3\%, 5\%, 10\%). For each of these twelve settings, we sample the appropriate number of requests and drivers 80 times. For each sample, we randomly assign rides in equal proportions to treatment and control, then we capture the treatment effect by randomly removing an appropriate fraction of requests from the control group. We repeat this treatment-control assignment and down-sampling 20 times for each sampled supply-demand graph, and we report the average of the true global treatment effect, as well as the average values of the RCT, SP, and Two-LP estimators. We then plot the median of our 80 replications in Figure~\ref{fig:sim_results}.
}

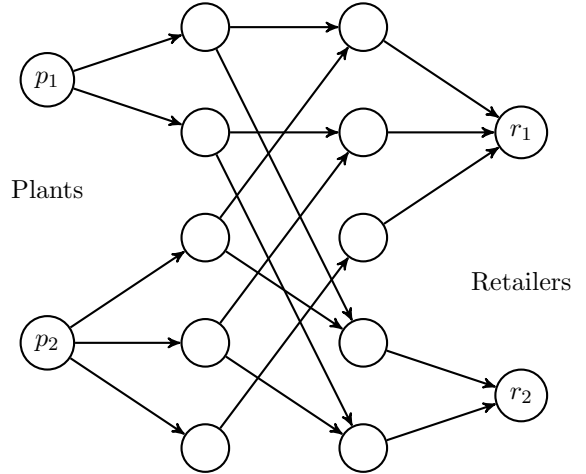
\begin{wrapfigure}{r}{0.48\textwidth}
  \begin{center}
    \begin{tikzpicture}[>=stealth',scale=0.7, every node/.style={circle,thick,draw,minimum size=20,scale=0.9}]
    \coordinate[label={above:Plants}](PLANTS) at (0, 4);
    \coordinate[label={above:Retailers}](RETAILERS) at (9, 2);
	\node (p1) at (0,7) {$p_1$};
	\node (p2) at (0,2) {$p_2$};
	\node (r1) at (9,6) {$r_1$};
	\node (r2) at (9,1) {$r_2$};
	\node (1) at (3,8) {};
	\node (2) at (3,6) {};
	\node (3) at (3,4) {};
	\node (4) at (3,2) {};
	\node (5) at (3,0) {};
	\node (6) at (6,8) {};
	\node (7) at (6,6) {};
	\node (8) at (6,4) {};
	\node (9) at (6,2) {};
	\node (10) at (6,0) {};
	\draw (p1) edge[->,thick] (1);
	\draw (p1) edge[->,thick] (2);
	\draw (p2) edge[->,thick] (3);
	\draw (p2) edge[->,thick] (4);
	\draw (p2) edge[->,thick] (5);
	\draw (6) edge[->,thick] (r1);
	\draw (7) edge[->,thick] (r1);
	\draw (8) edge[->,thick] (r1);
	\draw (9) edge[->,thick] (r2);
	\draw (10) edge[->,thick] (r2);
	\draw (1) edge[->,thick] (6);
	\draw (1) edge[->,thick] (9);
	\draw (2) edge[->,thick] (7);
	\draw (2) edge[->,thick] (10);
	\draw (3) edge[->,thick] (6);
	\draw (3) edge[->,thick] (9);
	\draw (4) edge[->,thick] (7);
	\draw (4) edge[->,thick] (10);
	\draw (5) edge[->,thick] (8);
\end{tikzpicture}%
  \end{center}
   \caption{Example of a simple supply chain network with 2 plants and 2 retailers, connected by 2 layers of 5 nodes each.}
    \label{fig:supply-chain}
\end{wrapfigure}

Our results show marketplace interference {arising} from the contention between {requests} when supply levels are low {relative to demand}. When contention is high we see that the {increase in total efficiency due to} new {requests} is smaller than the {increase in demand}, by up to 20\%. Due to interference, this effect is not captured in the standard RCT estimator which produces a biased estimate that remains stubbornly proportional to $e$; in contrast, the SP estimator {and Two-LP estimator} manage to eliminate nearly all the bias induced by network effects, {for all effect magnitudes}. 

In {Figure~\ref{fig:sim_results_std}} we show the median standard deviation of the actual treatment effect and {all estimators across all replications of a single configuration}. {As the number of drivers increases, estimator variance decreases. Furthermore, the RCT estimator typically has the most variance, and the gap with the other two estimators increases as the number of requests increases.}

\subsection{{Supply chains}}

We next consider a simple supply chain network from \cite[Chapter 9]{ahuja1988network}, displayed in Figure~\ref{fig:supply-chain}. The network has two plant nodes and two retail nodes, connected via two layers of warehouses consisting of 5 nodes each. We augment the network with capacities and costs, which are fixed for all simulations. Each edge has a maximum capacity $k_{u,w}$ sampled from a Poisson distribution with parameter 80, and a cost $c_{u,w}$ sampled uniformly from $[0,5]$. The fixed cost of production per unit is $\kappa_1=37$ in plant 1 and $\kappa_2=20$ in plant 2. The price per unit is $p_1=50$ at retailer 1 and $p_2=60$ at retailer 2. The value $v_{u,w}$ of an edge is given by $-c_{u,w}$ for all intermediate edges, $-c_{u,i}+p_{i}$ for all edges $(u,i)$ going into retailer node $i$, and $-c_{j,w}-\kappa_j$ for all edges $(j,w)$ coming out of plant node $j$.

W{e assume} the plants $p_1$ and $p_2$ have production capacities $S_1\sim\text{Poisson}(\pi_1=130)$ and $S_2\sim\text{Poisson}(\pi_2=190)$. The demand $D_1$ at retailer $r_1$ and $D_2$ at retailer $r_2$ are also Poisson random variables with parameters $\lambda_1$ and $\lambda_2$ which we will vary as part of the simulation. The treatment introduces a demand change $\beta_1$ at retailer $r_1$ and $\beta_2$ at retailer $r_2$. {In particular, we} consider two scenarios for the demand arrival rates $\bm{\lambda}$: an undersupply scenario where $\bm\lambda=(130,120)$, and an oversupply scenario where $\bm\lambda=(60,60)$. We {study five} possible treatment effects of varying magnitude and sign: $\boldsymbol{\beta}\in\{ {\footnotesize\left(10,10\right),\left(20,20\right),\left(-10,-10\right)}$, $\footnotesize{\left(-20,20\right),\left(20,-20\right)}\}.$ Readers will notice that the effects are not always sign-consistent.

\revision{W}e first compute the true global treatment effect by separately simulating the control and treatment settings with respective demand arrival rates $\bm\lambda$ and $\bm\lambda+\bm\beta$. Then we simulate an experiment with treatment fraction $\rho=0.5$ and compute the RCT{, SP, and Two-LP} estimators. {Results are shown} in Figure~\ref{fig:supply-chain}, where each vertical bar represents the average over 1000 simulations, and each group of vertical bars corresponds to a particular treatment effect on demand $\bm\beta$.

In Figure~\ref{fig:supplychain-oversupply}, we observe the result in the oversupplied regime. As predicted by the theory, the RCT estimator always overestimates the magnitude of the global treatment effect, while the SP estimator {and the Two-LP estimator} sometimes overestimate and sometimes underestimate the effect. Nevertheless, the differences between {all} the estimators are small since there is minimal contention for supply.
{In contrast}, the results in Figure~\ref{fig:supplychain-undersupply} {reveal a} pronounced {difference between the estimators, with the RCT estimator systematically overestimating the magnitude of the treatment effect by more than a factor of two. The SP estimator consistently reduces bias, even when the treatment is not sign-consistent.}

\begin{figure}[ht]
    \begin{subfigure}{0.49\columnwidth}
        \includegraphics[width=\columnwidth]{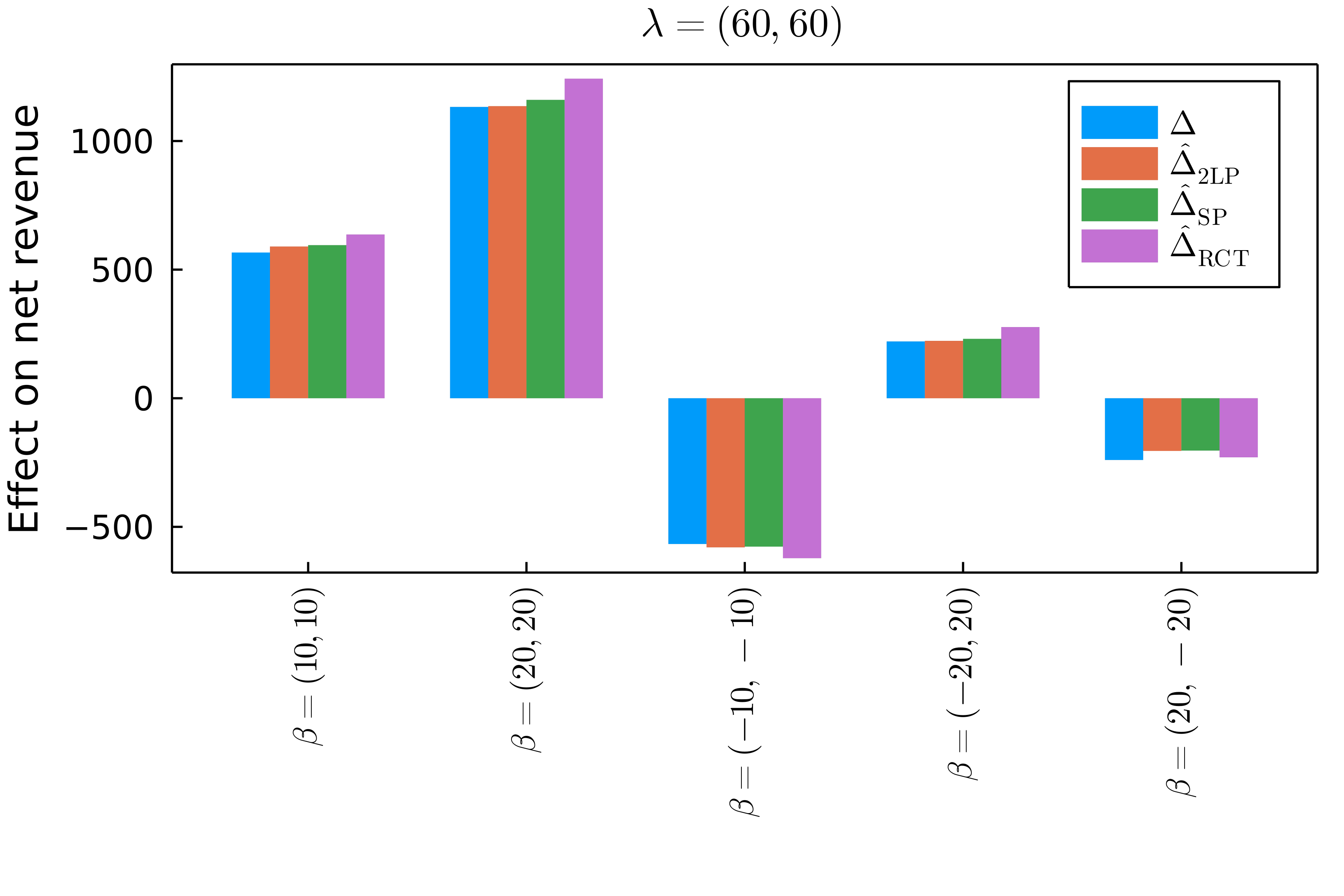}
        \caption{Oversupply}
        \label{fig:supplychain-oversupply}
    \end{subfigure}\hfill
    \begin{subfigure}{0.49\columnwidth}
        \includegraphics[width=\columnwidth]{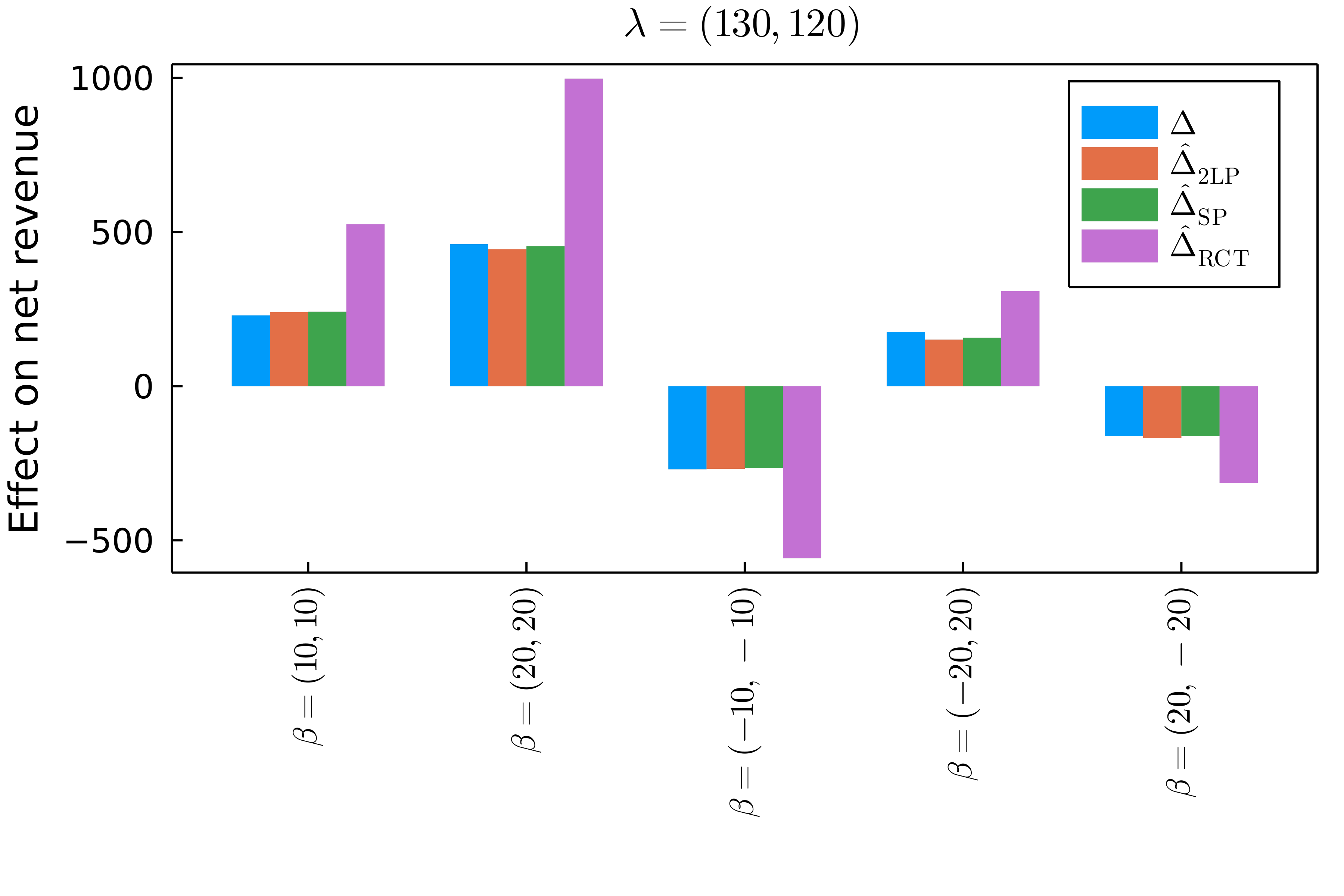}
        \caption{Undersupply}
        \label{fig:supplychain-undersupply}
    \end{subfigure}\hfill
    \caption{Supply chain simulation results comparing the true global treatment effect $\Delta$ to the RCT and SP estimators. Results are averaged over 1000 simulations.}
    \label{fig:supplychain}
\end{figure}

\subsection{{Asymmetric experiments}}

{Finally, we study the performance of our estimators in asymmetric experiments, i.e., experiments where the fraction of units assigned to the treatment group is (much) smaller than 50\%. We consider the same undersupply setting presented in Figure~\ref{fig:supplychain-undersupply} ($\bm\lambda=(130,120)$), and repeat the same analysis, but this time with $\rho=0.1$ and $\rho=0.01$. Results are shown in Figure~\ref{fig:supplychain-asymmetric}.

\begin{figure}[ht]
    \begin{subfigure}[t]{0.57\columnwidth}
        \includegraphics[width=\columnwidth]{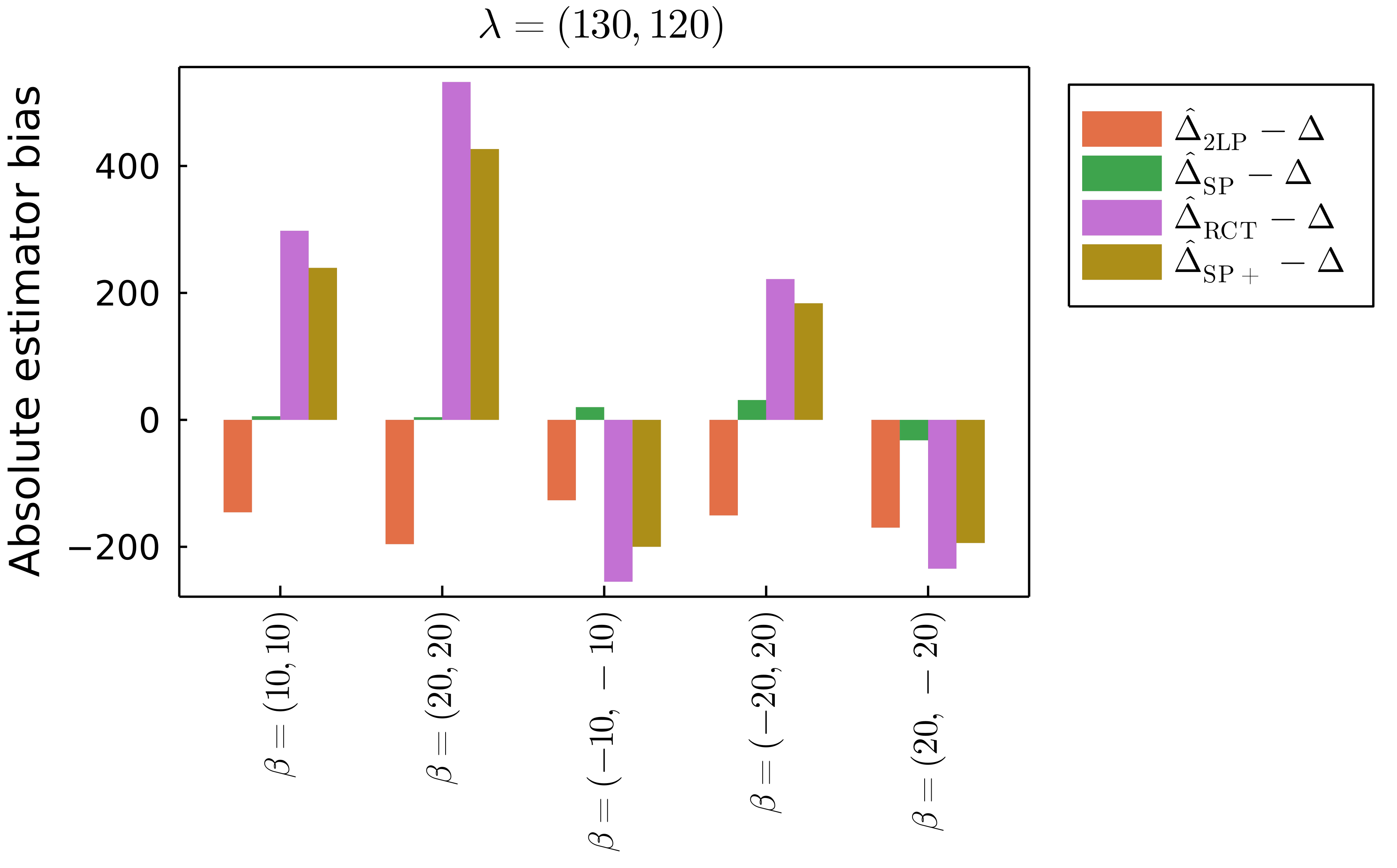}
        \caption{$\rho=0.1$}
        \label{fig:supplychain-rho10}
    \end{subfigure}\hfill
    \begin{subfigure}[t]{0.423\columnwidth}
        \includegraphics[width=\columnwidth]{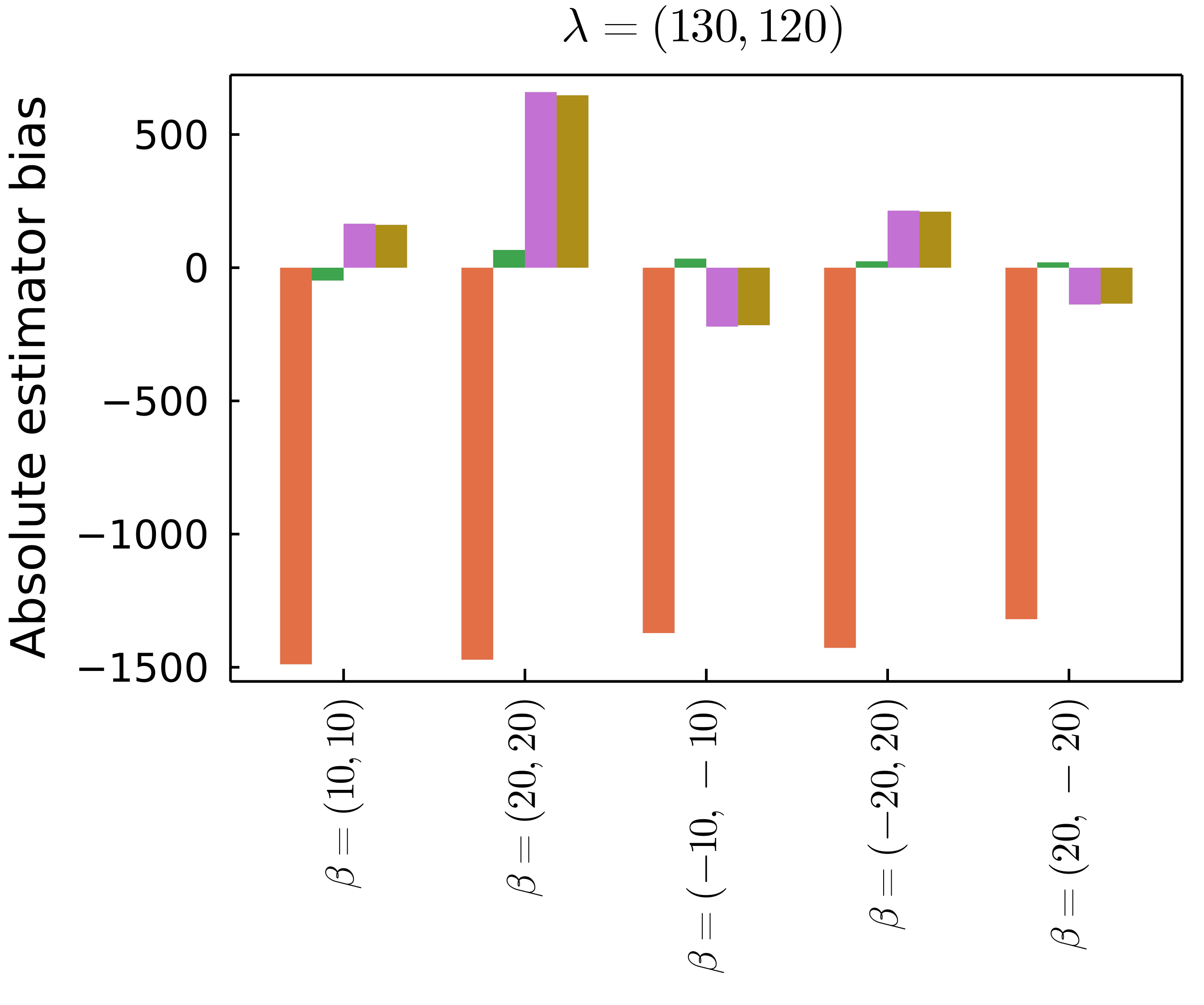}
        \caption{$\rho=0.01$}
        \label{fig:supplychain-rho1}
    \end{subfigure}\hfill
    \caption{{Supply chain simulation results showing the bias of various estimators for asymmetric experiments. Results are averaged over 1000 simulations.}}
    \label{fig:supplychain-asymmetric}
\end{figure}

First, we verify that the SP+ estimator, which we introduced to achieve provable bias reduction over the RCT estimator for any $\rho$, does accomplish this task. However, it remains significantly outperformed by the simpler SP estimator, indicating that the guarantee of bias reduction afforded by the SP+ estimator comes at a steep cost. Most strikingly, the performance of the Two-LP estimator deteriorates sharply as the treatment fraction shrinks. It appears that when the treatment group is small, we observe a similar small-sample issue as the one we analyzed in Proposition~\ref{prop:twolp-worst-case}. Figure~\ref{fig:supplychain-asymmetric} displays a setting where the shadow price estimator performs far better in practice than the theory can guarantee.

}
\section{Takeaways}\label{sec:takeaways}

Most of the literature on interference in marketplaces is built on the implicit assumption the standard mechanism to match supply and demand is consumer choice. Under this assumption, our ability to directly model the interference structure between users is limited by the fundamental challenge of accurately modeling choice dynamics. Not all marketplaces are choice-based, however; matching-based platforms that assign supply units to demand units via algorithmic matching are also common. In such marketplaces, the structure of interference is actually mostly accessible to the platform: the shadow prices associated with the supply and demand constraints of the matching LP directly measure the incremental value of additional demand and supply units. \revision{We use this idea} with the specific aim of correcting interference bias in marketplace experiments.
The matching LP shadow prices can likely prove to be a broadly applicable metric for matching-based platforms in any use case where the signal of interest is the marginal value of a unit of supply or demand.

\singlespacing
\small
\bibliography{references}
\newpage
\onehalfspacing
\appendix
\section{Reducing Bias in Secondary Metrics} \label{sec:secondary}

\subsection{Secondary Metrics}

\revision{In the main body of the paper,} we assume that the platform's estimand of interest is the global treatment effect on system value, i.e., the difference between the total system value under global treatment $\Phi_\tau(\bm{D}^{\tau,\bm\lambda+\bm\beta},\bm{S}^\tau)$ and the total system value under global control $\Phi_\tau(\bm{D}^{\tau,\bm\lambda},\bm{S}^\tau)$. In practice, marketplace companies track more than one business metric to assess their performance. For instance, a ride-hailing company might match riders to drivers primarily based on pickup distance, but may also care about fulfilling more rides using electric vehicles. An example of such a setting is presented in Figure~\ref{fig:secondary-example}.

\begin{figure}[h]
    \centering
    \includegraphics[width=0.8\columnwidth]{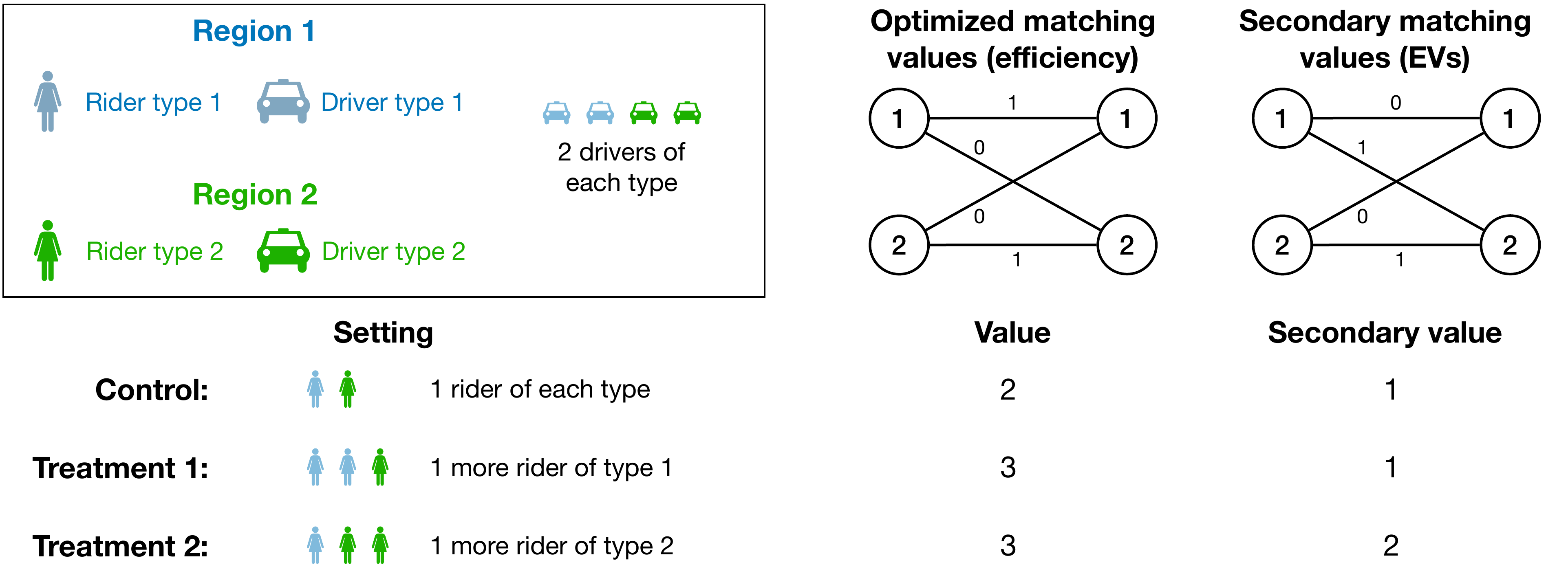}
    \caption{{Example of a secondary metric. In a ride-hailing setting, riders and drivers originate from two regions, far enough apart that it is only efficient to match riders and drivers from the same region. Drivers in Region 2 drive electric vehicles (EVs). Therefore, though both treatment 1 and treatment 2 have the same global treatment effect, treatment 2 also improves the secondary metric which counts the number of rides satisfied with electric vehicles.}}
    \label{fig:secondary-example}
\end{figure}

Even if {the total system value is computed as a weighted combination of several metrics of interest, it remains of interest to understand how each component is affected by the treatment.} For instance, in ride-hailing, a treatment that increases total welfare by {both reducing pickup distance and increasing the number of rides served by electric vehicles} may be preferable to a treatment that increases total welfare {by the same amount}, but by disproportionately {increasing the number of rides served by electric vehicles at the expense of pickup distance}.

{Though it is not the case in the simple example in Figure~\ref{fig:secondary-example}, secondary metrics can also} suffer from marketplace interference bias. {In our ride-hailing example}, a price incentive increasing demand arrival rate could lead to more riders in the treatment group matching with electric vehicles, seemingly suggesting that the price incentive improves sustainability metrics even though the total number of electric vehicle {matches} remains fixed. Therefore, it is of equal interest to reduce bias in secondary metrics as in the optimized value function.

However, solving the matching LP only provides shadow prices with respect to the original value function, as defined by the edge weights {$v_{u,w}$}. {These} shadow prices do not provide information regarding a secondary metric defined by new edge weights {$\tilde{v}_{u,w}$}. {The goal of this section is to develop a method that extends the bias-reducing properties of shadow prices to secondary metrics of interest.}

Formally, consider the matching problem defined in~\eqref{eq:matching-formulation} with edge weights $v_{u,w}$, and consider a secondary metric defined by edge weights $\tilde{v}_{u,w}$. Let $x^*_{u,w}$ denote the (assumed unique) optimal primal solution of $\Phi_\tau(\bm{d},\bm{s})$. We can define the value of the secondary metric in this setting as 
\begin{equation}
\Phi_\tau^{\bm{\tilde{v}}}(\bm d, \bm s)=\sum_{(u,w)\in\mathcal{A}}\tilde{v}_{u,w}x^*_{u,w}
\end{equation}

In our experimental setting, the estimand of interest is the \emph{secondary} treatment effect $\Delta^{\tau,\bm w}$, defined as
\begin{equation}
    \Delta^{\tau, \bm{\tilde{v}}}=\frac{1}{\tau}\mathbb{E}\left[\Phi_\tau^{\bm{\tilde{v}}}(\bm{D}^{\tau,\bm\lambda+\bm\beta}, \bm{S}^\tau)-\Phi_\tau^{\bm{\tilde{v}}}(\bm{D}^{\tau,\bm\lambda}, \bm{S}^\tau)\right].
\end{equation}

\subsection{Secondary Estimators}

{The standard RCT estimator for the secondary metric $\bm{\tilde{v}}$ is given by:}
\begin{equation}
    \hat{\Delta}^{\tau,\bm{\tilde{v}}}_{\text{{RCT}}}=\frac{1}{\tau}\left(\frac{1}{\rho}\sum_{(u,w)\in\mathcal{A}}\tilde{v}_{u,w}X_{u,w}^{\tau,\treatment}-\frac{1}{1-\rho}\sum_{(u,w)\in \mathcal{A}}\tilde{v}_{u,w}X_{u,w}^{\tau,\control}\right),
\end{equation}
{where the optimal matching values $X_{u,w}^{\tau,\treatment}$ and $X_{u,w}^{\tau,\control}$ are computed with respect to the original matching values $v_{u,w}$.}

{In Section~\ref{sec:rct-failure}, we established that the RCT estimator for the primary metric always suffers from interference bias. The same is not true for the RCT estimator for secondary metrics. Because we assume no particular structure on the {values} $\bm{\tilde{v}}$, it is possible to reverse-engineer a secondary metric that does not suffer from interference bias in any marketplace configuration.
In practice, metrics are typically not constructed in such an adverse manner, and may indeed suffer from interference bias, justifying our desire to define an analog of the shadow price estimator for secondary metrics. We first propose a way to extend the notion of a shadow price to a secondary metric.
}

\begin{definition}
Let $\Phi_\tau^{\bm{\tilde{v}}}(\bm d,\bm s)$ designate the value function for a secondary metric with edge weights $\tilde{v}_{i,j}$ for some demand vector $\bm d$ and supply vector $\bm s$. Then the secondary shadow price of demand type $i$ is given by
\[
a_i^{\bm{\tilde{v}}}=\Phi_\tau^{\bm{\tilde{v}}}\left(\mathbf{d},\mathbf{s}\right) - \Phi_\tau^{\bm{\tilde v}}\left(\mathbf{d}-\mathbf{e}_{i},\mathbf{s}\right).
\]
\end{definition}

The secondary shadow prices above are defined analogously to the true shadow prices {(see~\eqref{eq:downward-dual})}. However, unlike the true shadow prices, {it is not obvious how to obtain them directly from} the matching LP. Applying the definition directly would require solving $n_d+1$ optimization problems, which creates a much higher computational burden. However, it turns out that we can {indeed} compute these secondary duals much more efficiently by using {the classic} complementary slackness {conditions}.

{Assuming the optimal solution of the matching problem is unique and nondegenerate, the complementary slackness conditions form a system of linear equations, which uniquely specifies the optimal shadow prices given the optimal primal solution. It turns out that, by replacing each value term $v_{u,w}$ in this linear system with the corresponding \emph{secondary} value term $\tilde{v}_{u,w}$, we obtain a new linear system which uniquely specifies the \emph{secondary} shadow prices. Therefore, rather than solving $n_d$ additional optimization problems, we can simply solve one linear system. We formalize this result in the following proposition.}

\begin{proposition}
\label{prop:secondary-cs}
Assume that the optimal solution $\bm{x^*}$ to $\Phi_\tau(\bm{d},\bm{s})$ is unique and nondegenerate. Then there exists an invertible matrix $\bm{M}$ and a vector $\tilde{\bm{v}}'$ such that
\[
a_i^{\bm{\tilde{v}}} = \left(\bm{M}^{-1}\tilde{\bm{v}}'\right)_i.
\]
\end{proposition}

{We provide an example of how to construct the linear system described in Proposition~\ref{prop:secondary-cs} in Figure~\ref{fig:secondary-complementary-slackness}. We first construct the linear system which uniquely determines the optimal dual solution from the optimal primal solution --- then replace every occurrence of $v_{u,w}$ with $\tilde{v}_{u,w}$. The proof of Proposition~\ref{prop:secondary-cs} in Section~\ref{sec:secondary-proofs} establishes why this construction is correct.}

\begin{figure}[h]
    \centering
    \includegraphics[width=0.9\columnwidth]{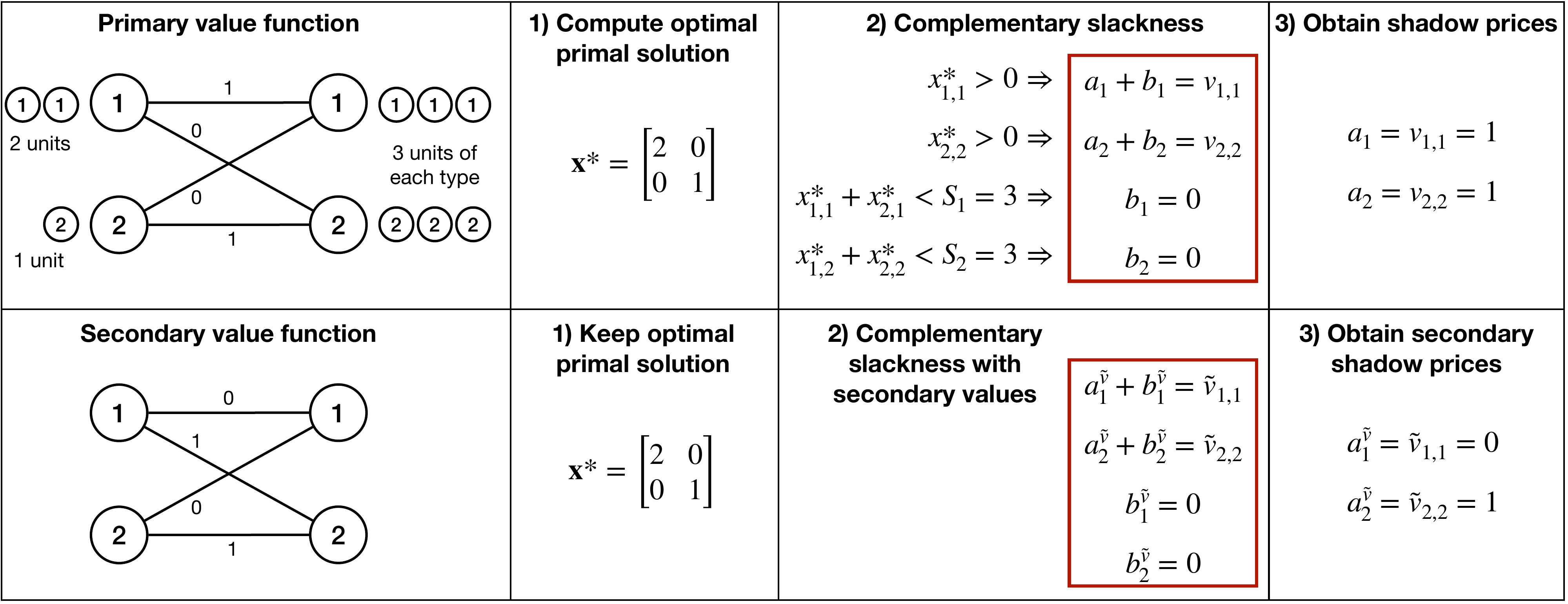}
    \caption{Example of secondary shadow price computation, using the bipartite matching graph from Figure~\ref{fig:secondary-example}. Given a particular supply and demand scenario, we can compute the optimal primal solution, then construct a linear system whose unique solution is the optimal dual solution. Replacing every occurrence of $v_{i,j}$ with $\tilde{v}_{i,j}$ yields a new linear system whose unique solution is the secondary dual solution.}
    \label{fig:secondary-complementary-slackness}
\end{figure}

{By} solving the matching problem in the experiment state with total demand $\bm{D}^{\tau,\experiment}$ and applying Proposition~\ref{prop:secondary-cs}, we can obtain the secondary shadow prices $\bm{A}^{\tau,\bm{\tilde{v}},\experiment}$. Using these marginal values, we can define a secondary version of the shadow price estimator:
\begin{equation}
    \hat{\Delta}^{\tau,\bm\tilde{v}}_{\text{SP}}=\frac{1}{\tau}\bm{A}^{\tau,\bm{\tilde{v}},\experiment}\cdot\left(\frac{1}{\rho}\bm{D}^{\tau,\treatment}-\frac{1}{1-\rho}\bm{D}^{\tau,\control}\right).
\end{equation}

\section{Analysis of bias in imbalanced markets}

So far, we have analyzed the performance of the SP estimator in a general case. In practice, marketplaces may choose to modulate their behavior depending on macroscopic characteristics such as the balance of supply and demand. It is of interest to understand estimator performance in a range of marketplace balance settings. Intuitively, we expect both estimators to perform well in the limit where supply greatly outnumbers demand; however, the naive RCT estimator is unlikely to perform well when demand outperforms supply since it does not capture the notion of contention for scarce supply units.

{For simplicity, we restrict our analysis of the undersupply and oversupply limits to the bipartite matching (uncapacitated) setting. We ignore capacities because we are trying to capture the effect of low or high supply relative to demand, and capacity introduces a third dimension that adds complexity but not insight.}

{W}e introduce scalar parameters $\bar{\lambda}$ (respectively $\bar{\pi}$) such that $\tilde{\bm{\lambda}}=\bar{\lambda}\bm{\alpha}$ with $\lVert\bm{\alpha}\rVert_1\le 1$ (respectively $\bm{\pi}=\bar{\pi}\bm{\gamma}$ with $\lVert\bm{\gamma}\rVert_1\le 1$). In other words, the total demand scales with the single parameter $\bar{\lambda}$ across all types, while the relative proportions of each type remain constant, parametrized by the vector $\bm{\alpha}$ (and similarly for supply). Recall that $\tilde{\bm{\lambda}}$ denotes the baseline arrival rate of demand to the platform; when factoring the request probability $p_i$ of demand intent of type $i$, \revision{$\bm{\lambda}=\bm{\tilde{\lambda}}\odot \bm{p}=\bar{\lambda}\bm{\alpha}\odot\bm{p}$} denotes the demand arrival rate observed by the platform under global control \revision{(where $\odot$ refers to the element-wise vector product)}; and \revision{$\bm{\beta}=\bm{\tilde{\lambda}}\odot \bm{p}=\bar{\lambda}\bm{\alpha}\odot\bm{q}$} denotes the ``extra'' demand arrival rate observed by the platform under global treatment (which could be either negative or positive). As a result, the parameter $\bar{\lambda}$ proportionally scales both the baseline demand arrival rate and the treatment effect.

Using this framework, we can study the limits $\bar{\lambda}/\bar{\pi}\to 0$ (oversupply) and $\bar{\lambda}/\bar{\pi}\to \infty$ (undersupply). We formalize our results in the following theorem. 
\begin{theorem}
    \label{thm:undersupply-oversupply}
    {Consider the bipartite matching setting where $\Phi(\cdot)=\Phi^b(\cdot)$, and a}ssume $\bm{\beta}\ge 0$.
    \begin{enumerate}[(a)]
        \item In the oversupply limit, both the RCT and SP estimators are unbiased:
        \begin{equation}
            \lim_{\bar{\lambda}/\bar{\pi}\to 0}\abs{\hat{\Delta}_{\text{\emph{SP}}}-\Delta}=0=\lim_{\bar{\lambda}/\bar{\pi}\to 0}\abs{\hat{\Delta}_{\text{\emph{RCT}}}-\Delta}.
        \end{equation}
        \item In the undersupply limit, only the SP estimator is unbiased:
        \begin{equation}
            \lim_{\bar{\lambda}/\bar{\pi}\to \infty}\abs{\hat{\Delta}_{\text{\emph{SP}}}-\Delta}=0\neq\lim_{\bar{\lambda}/\bar{\pi}\to
            \infty}\abs{\hat{\Delta}_{\text{\emph{RCT}}}-\Delta}.
        \end{equation}
    \end{enumerate}
\end{theorem}

Theorem~\ref{thm:undersupply-oversupply} is a direct consequence of the particular form of the partial-treatment value function. In the oversupply limit ($\bar{\lambda}/\bar{\pi}\to 0$), $\Psi(\cdot)$ is linear, and coincides exactly with $\hat{\Psi}_{\text{RCT}}(\eta)=\bm{\bar{v}^*}\cdot(\bm\lambda+\eta\bm\beta)$, so the RCT estimator is unbiased. However, in the undersupply limit, $\Psi(\cdot)$ is locally constant, and is not well approximated by $\hat{\Psi}_{\text{RCT}}(\cdot)$, which passes through the origin. In contrast, the SP estimator is unbiased in both limits because the partial-treatment value function is linear and the SP estimator implicitly constructs the best linear approximation in a Taylor series sense.

\section{{Analysis of estimators: proofs}}

\subsection{{Convergence to the fluid limit}}

Denoting $\nu_{\tau}$ as the density-averaged total value with density $\tau$, we seek to show that the average total value $\nu_\tau$ converges to the fluid limit $\Phi(\bm{\lambda}, \bm{\pi})$ as $\tau$ tends to infinity. {We begin with a simple scaling lemma.}

\begin{lemma}
\label{lem:scaled-problem}For any demand vector $\bm{d}\in \mathbb{Z}_+^{n_d}$, supply vector $\bm{s}\in \mathbb{Z}_+^{n_s}$, and density $\tau$, we can rewrite the total value as
\begin{equation}
\label{eq:lemma-scaled-problem}
\frac{1}{\tau}\Phi_\tau\left(\boldsymbol{d},\boldsymbol{s}\right)=\Phi\left(\frac{1}{\tau}\boldsymbol{d},\frac{1}{\tau}\boldsymbol{s}\right).
\end{equation}
Moreover, if $x^*_{u,w}$ and $z^*_{u,w}$ are the respective optimal primal solutions of
the right-hand and left-hand sides, then $z^*_{u,w}=\tau x^*_{u,w}$.
\end{lemma}

\begin{proof}[Proof of Lemma \ref{lem:scaled-problem}.]
Following Eq.~\eqref{eq:general-matching-formulation}, we can write
{
\begin{align*}
            \Phi\left(\frac{1}{\tau}\bm{d},\frac{1}{\tau}\bm{s}\right)=&\;\max\sum_{(u,w)\in\mathcal{A}}v_{u,w}x_{u,w},\\
&\;\text{s.t. }\sum_{w:(w,u)\in\mathcal{A}}x_{w,u} \le \frac{1}{\tau}d_i  & \forall u=u_i^d\in\mathcal{U}_d,\\
&\;\phantom{\text{s.t. }}\sum_{w:(u,w)\in\mathcal{A}}x_{u,w} \le \frac{1}{\tau}s_j  & \forall u=u_j^s\in\mathcal{U}_s,\\
&\;\phantom{\text{s.t. }}\sum_{w:(u,w)\in\mathcal{A}}x_{u,w} - \sum_{w:(w,u)\in\mathcal{A}}x_{w,u} = 0  & \forall u\in\mathcal{U}_*,\\
&\;\phantom{\text{s.t. }}0\le x_{u,w} \le k_{u,w} &\forall (u,w)\in\mathcal{A}.
\end{align*}

Applying a change of variables $z_{u,w}=\tau x_{u,w}$ yields 
\begin{align*}
            \Phi\left(\frac{1}{\tau}\bm{d},\frac{1}{\tau}\bm{s}\right)=&\;\frac{1}{\tau}\max\sum_{(u,w)\in\mathcal{A}}v_{u,w}z_{u,w},\\
&\;\text{s.t. }\sum_{w:(w,u)\in\mathcal{A}}z_{w,u} \le d_i  & \forall u=u_i^d\in\mathcal{U}_d,\\
&\;\phantom{\text{s.t. }}\sum_{w:(u,w)\in\mathcal{A}}z_{u,w} \le s_j  & \forall u=u_j^s\in\mathcal{U}_s,\\
&\;\phantom{\text{s.t. }}\sum_{w:(u,w)\in\mathcal{A}}z_{u,w} - \sum_{w:(w,u)\in\mathcal{A}}z_{w,u} = 0  & \forall u\in\mathcal{U}_*,\\
&\;\phantom{\text{s.t. }}0\le z_{u,w}\le \tau k_{u,w} &\forall (u,w)\in\mathcal{A},\\
&\;=\frac{1}{\tau}\Phi_\tau\left(\bm{d},\bm{s}\right).
\end{align*}
}
\end{proof}

Lemma~\ref{lem:scaled-problem} provides an equivalent definition of the total value in the fluid limit as
\[
\nu=\lim_{\tau\rightarrow\infty}\mathbb{E}\left[\frac{1}{\tau}\Phi_\tau\left(\mathbf{D}^{\tau},\mathbf{S}^{\tau}\right)\right]=\lim_{\tau\rightarrow\infty}\mathbb{E}\left[\Phi\left(\frac{1}{\tau}\mathbf{D}^{\tau},\frac{1}{\tau}\mathbf{S}^{\tau}\right)\right].
\]

{Before proving Theorem~\ref{thm:fluid-limit}, we introduce a second useful lemma.}
\begin{lemma}
\label{lem:max-scaled-pois}For $n\in\mathbb{N}$, let $X_{n}\sim P\left(\lambda n\right)$,
and $Y_{n}=\frac{1}{n}X_{n}$ then $Z=\max_{n=1}^{\infty}Y_{n}$ has a finite mean.
\end{lemma}

\begin{proof}[Proof of Lemma~\ref{lem:max-scaled-pois}.]
For any Poisson variable $V\sim P\left(\lambda_{0}\right)$ we
can bound the tail distribution as follows \citep[Corollary 2.4]{janson2000random}:
\[
P\left(V\ge t\right)\le\exp\left(-t\right)\;\;\;\forall t\ge7\lambda_{0}.
\]

Thus for all $t>t_{0}=\max\left(7\lambda,\log2\right)$
\begin{align*}
P\left(Z\ge t\right) & =P\left(\exists n\text{ s.t. }Y_{n}\ge t\right)\\
 & \le\sum_{n=1}^{\infty}P\left(Y_{n}\ge t\right)\\
 & =\sum_{n=1}^{\infty}P\left(X_{n}\ge nt\right)\\
 & \le\sum_{n=1}^{\infty}\exp\left(-nt\right)\\
 & =\exp\left(-t\right)/\left(1-\exp\left(-t\right)\right)\\
 & \le2\exp\left(-t\right)
\end{align*}
Applying this to $\mathbb{E}\left[z\right]=\int_{0}^{\infty}P\left(z\ge t\right)dt$
we see that 
\begin{align*}
\mathbb{E}\left[z\right] & =\int_{0}^{t_{0}}P\left(z>t\right)dt+\int_{t_{0}}^{\infty}P\left(z>t\right)dt\\
 & \le t_{0}+2\int_{t_{0}}^{\infty}\exp\left(-t\right)dt<\infty
\end{align*}
\end{proof}

We are now ready to prove Theorem~\ref{thm:fluid-limit}.

\begin{proof}[Proof of Theorem \ref{thm:fluid-limit}.] 

For simplicity, in this proof we denote $\bm{D}^{\tau,\bm\lambda}$ by $\bm{D}^\tau$ since we only consider a single arbitrary $\bm{\lambda}$.

\textbf{Part 1: convergence of expected total value.} We first use the strong law of large numbers, which implies that $\frac{1}{\tau}\mathbf{D}^{\tau}\rightarrow\boldsymbol{\lambda}$
and $\frac{1}{\tau}\mathbf{S}^{\tau}\rightarrow\boldsymbol{\pi}$ almost
surely. We then define
\[
Z=\max_{\tau=1}^{\infty}\frac{1}{\tau}Y_{\tau}, \text{ with } Y_{\tau}=\sum_{i=1}^{n_d}D^\tau_i,
\]
and note that $Y_{\tau}$ is a sum of Poisson random variables and is therefore Poisson-distributed with parameter $\tau\sum_{i=1}^{n_d}\lambda_i$. By Lemma~\ref{lem:max-scaled-pois} we know $Z$ has bounded mean, and furthermore we notice that
\[
\Phi\left(\frac{1}{\tau}\mathbf{D}^{\tau},\frac{1}{\tau}\mathbf{S}^{\tau}\right)<Z{\max_{i,j,p}\nu_{i,j,p}}.
\]
We finally apply the dominated convergence theorem to imply convergence of the expectation which is the desired result.

\textbf{Part 2: almost sure convergence of optimal decision variables.} For simplicity, we assume the optimal primal and dual solutions of $\Phi\left(\bm{\lambda},\bm{\pi}\right)$ are unique. If there are multiple primal (or dual) solutions, our proof below implies that every converging sub-sequence will tend to \emph{an} optimal primal (or dual) solution of the fluid limit.

\textit{Step 1: convergence of dual variables.} We first note that for every $\tau$, the optimal dual {variables $(\bm{A}^\tau,\bm{B}^\tau,\bm{M}^\tau,\bm{\Xi}^\tau)$, as well as $(\bm{a}^*,\bm{b}^*,\bm{m}^*,\bm{\xi}^*)$}, must be extreme points of the dual polyhedron~\eqref{eq:general-matching-dual}, of which there are finitely many.

Assume we can find an infinite subsequence $\tau_{1}, \tau_2, \ldots$ such that {$(\bm{A}^{\tau_i},\bm{B}^{\tau_i},\bm{M}^{\tau_i},\bm{\Xi}^{\tau_i})=(\bm{a}',\bm{b}',\bm{m}',\bm{\xi}')\neq (\bm{a}^*,\bm{b}^*,\bm{m}^*,\bm{\xi}^*)$}, a distinct extreme point of the dual polyhedron, implying that
\[
\bm{a}'\cdot \frac{\bm{D}^{\tau_i}}{\tau_i}+\bm{b}'\cdot\frac{\bm{S}^{\tau_i}}{\tau_i}+\bm{\xi}'\cdot \bm{k}<\bm{a}^*\cdot \frac{\bm{D}^{\tau_i}}{\tau_i}+\bm{b}^*\cdot\frac{\bm{S}^{\tau_i}}{\tau_i}+\bm{\xi}^*\cdot\bm{k}
\]
By the strong law of large numbers, the dual objective values of the subsequence converge almost surely to $\bm{a}'\cdot \bm{\lambda}+\bm{b}'\cdot\bm{\pi} +\bm{\xi}'\cdot \bm{k}< \bm{a}^*\cdot \bm{\lambda}+\bm{b}^*\cdot\bm{\pi} + \bm{\xi}^*\cdot \bm{k}$, which contradicts the optimality of {$(\bm{a}^*,\bm{b}^*,\bm{m}^*,\bm{\xi}^*)$ in $\Phi(\bm{\lambda}, \bm{\pi})$. Thus there exists $\tau_0 > 0$ such that $(\bm{a}^*, \bm{b}^*,\bm{m}^*,\bm{\xi}^*)$ is the} optimal dual solution of $\Phi\left(\frac{1}{\tau}\mathbf{D}^{\tau},\frac{1}{\tau}\mathbf{S}^{\tau}\right)$ for every $\tau > \tau_0$.

\textit{Step 2: convergence of primal variables.} For $\tau > \tau_0$, the optimal primal solution is defined by a system of linear equations obtained from the complementary slackness {and primal feasibility conditions:
\begin{align*}
    X_{u_j^s,u_i^d}&=0 & \forall (u_j^s,u_i^d)\in\mathcal{A}_{s\to d},~b_j^* + a_i^* > v_{u_j^s,u_i^d}\\
    X_{u_j^s,w}&=0 & \forall (u_j^s,w)\in\mathcal{A}_{s\to*},~b_j^*-m_w^* > v_{u_j^s,w}\\
    X_{w,u_i^d}&=0 & \forall (w,u_i^d)\in\mathcal{A}_{*\to d},~m_w^*+a_i^* > v_{w,u_i^d}\\
    X_{u,w}&=0 & \forall (u,w)\in\mathcal{A}_{*},~m_u^*-m_w^* > v_{u,w}\\
    X_{u,w}&=k_{u,w} & \forall (u,w)\in\mathcal{A},~\xi^*_{u,w}>0\\
    \sum_{w:(w,u)\in\mathcal{A}}X_{w,u}&=\frac{1}{\tau}D_i^\tau & \forall u=u_i^d\in\mathcal{U}_d,~a_i^* > 0\\
    \sum_{w:(u,w)\in\mathcal{A}}X_{u,w}&=\frac{1}{\tau}S_j^\tau & \forall u=u_j^s\in\mathcal{U}_s,~b_j^* > 0\\
    \sum_{w:(u,w)\in\mathcal{A}}X_{u,w} - \sum_{w:(w,u)\in\mathcal{A}}X_{w,u} &= 0  & \forall u\in\mathcal{U}_*
\end{align*}

The solution of these linear equations is continuous on the inputs
$\frac{1}{\tau}\boldsymbol{D}^{\tau}$ and $\frac{1}{\tau}\boldsymbol{S}^{\tau}$,
both of which converge strongly to $\boldsymbol{\lambda},\boldsymbol{\pi}$.
By continuity $X^{\tau}_{u,w}$ also converges strongly to $x^*_{u,w}$.}
\end{proof}

\begin{proof}[Proof of Proposition~\ref{prop:concave-integrable}.]
This result follows directly from Chapter 5 of \cite{bertsimas1997introduction}. Theorem 5.1 proves concavity, and piece-wise linearity with finitely many pieces is established in the paragraphs following Theorem 5.1.
\end{proof}

\subsection{{Analysis of the RCT estimator}}

\begin{proof}[Proof of Proposition~\ref{prop:infinite-horizon-limit-rct}.]
We assume without loss of generality that $D_i^{\tau}>0$ for each demand type $i$, otherwise we can simply remove this demand type. In the proof, we first assume for clarity that $D_i^{\tau,\treatment}>0$ and $D_i^{\tau,\control}>0$.

We can rewrite each term in Definition~\ref{def:rct-estimator} as follows:
{
\begin{align}
    \sum_{i=1}^{n_d}\sum_{j=1}^{n_s}\sum_{p=1}^{n_{i,j}}\nu_{i,j,p}Y_{i,j,p}^{\tau,\treatment}&=\sum_{i=1}^{n_d}\frac{\sum_{j=1}^{n_s}\sum_{p=1}^{n_{i,j}}Y_{i,j,p}^{\tau,\treatment}\nu_{i,j,p}}{D_i^{\tau,\treatment}}D_i^{\tau,\treatment}\nonumber\\
    &=\bar{\bm{V}}^{\tau,\treatment}\cdot \bm{D}^{\tau,\treatment},\label{eq:average-value-treatment}
\end{align}
}
and similarly
\begin{equation}
\label{eq:average-value-control}
\sum_{i=1}^{n_d}\sum_{j=1}^{n_s}{\sum_{p=1}^{n_{i,j}}\nu_{i,j,p}Y_{i,j,p}^{\tau,\control}}=\bar{\bm{V}}^{\tau,\control}\cdot \bm{D}^{\tau,\control},
\end{equation}
where $\bar{V}^{\tau,\treatment}_i$ (resp. $\bar{V}_i^{\tau,\control}$) designates the average value obtained from demand type $i$ from the treatment (resp. control) group.

Taking the expectation over treatment-control assignments preserves the total number of units in the treatment and control groups. Therefore,
\begin{align*}
    \mathbb{E}[\bar{V}_i^{\tau,\control}|\bm{D}^{\tau,\control}, \bm{D}^{\tau,\treatment},\bm{S}^\tau]&=\mathbb{E}\left[\frac{{\sum\limits_{j=1}^{n_s}\sum\limits_{p=1}^{n_{i,j}}\nu_{i,j,p}Y_{i,j,p}^{\tau,\control}}}{D_i^{\tau,\control}}|\bm{D}^{\tau,\control}, \bm{D}^{\tau,\treatment},\bm{S}^\tau\right]\\
    &=\sum_{j=1}^{n_s}{\sum_{p=1}^{n_{i,j}}}\frac{{\nu_{i,j,p}}}{D_i^{\tau,\control}}\mathbb{E}\left[{Y_{i,j,p}^{\tau,\control}}|\bm{D}^{\tau,\control}, \bm{D}^{\tau,\treatment},\bm{S}^\tau\right].
\end{align*}
By Assumption~\ref{ass:blind}, among the {$Y_{i,j,p}^{\tau}$} demand units of type $i$ matched with supply type $j$ {along path $p$}, each one is assigned to the control group with probability $D_i^{\tau,\control}/D_i^{\tau}$. Therefore we conclude
\begin{align*}
    \mathbb{E}[\bar{V}_i^{\tau,\control}|\bm{D}^{\tau,\control}, \bm{D}^{\tau,\treatment},\bm{S}^\tau]&=\sum_{j=1}^{n_s}{\sum_{p=1}^{n_{i,j}}}\frac{{\nu_{i,j,p}}}{D_i^{\tau,\control}}{Y_{i,j,p}^{\tau}}\frac{D_i^{\tau,\control}}{D_i^\tau}\\
    &={\sum_{j=1}^{n_s}\sum_{p=1}^{n_{i,j}}\frac{Y_{i,j,p}^{\tau}\nu_{i,j,p}}{D_i^{\tau}}}:=\bar{V}_i^\tau,
\end{align*}
with the same proof applying to the treatment group. So far, we have assumed that both $\bm{D}^{\tau,\treatment} > 0$ and $\bm{D}^{\tau,\control} > 0$. If we assume that one of them is $0$, then the right-hand-side in either~\eqref{eq:average-value-treatment} or \eqref{eq:average-value-control} is 0, and the rest of the proof still holds. Therefore, we can write
\[
\mathbb{E}\left[\hat{\Delta}^\tau_{\text{RCT}}~\big|~\bm{D}^{\tau,\control},\bm{D}^{\tau,\treatment},\bm{S}^\tau\right]=\bar{\bm{V}}^\tau\cdot\left(\frac{1}{\rho}\frac{1}{\tau} \bm{D}^{\tau,\treatment}-\frac{1}{1-\rho}\frac{1}{\tau}\bm{D}^{\tau,\control}\right).
\]
As $\tau\to\infty$, we already know from the strong law of large numbers that $\bm{D}^{\tau,\treatment}/\tau$ converges to $\bm{\lambda}^{\treatment}$ (and $\bm{D}^{\tau,\control}/\tau$ to $\bm{\lambda}^{\control}$). Furthermore, Theorem~\ref{thm:fluid-limit} implies that $\bar{\bm{V}}^\tau$ converges to $\bar{\bm{v}}^*$.  Once again applying Lemma~\ref{lem:max-scaled-pois} and the dominated convergence theorem, we obtain
\[
\lim_{\tau\to\infty}\mathbb{E}\left[\hat{\Delta}^\tau_{\text{RCT}}~\big|~\bm{D}^{\tau,\control},\bm{D}^{\tau,\treatment},\bm{S}^\tau\right]=\bar{\bm{v}}^*\cdot\left(\frac{1}{\rho}\rho(\bm{\lambda}+\bm{\beta})-\frac{1}{1-\rho}(1-\rho)\bm{\lambda}\right)=\bar{\bm{v}}^*\cdot\bm{\beta}.
\]
\end{proof}

Notice that the proof of Proposition~\ref{prop:infinite-horizon-limit-rct} analyzes the estimator $\mathbb{E}\left[\hat{\Delta}^\tau_{\text{RCT}}~\big|~\bm{D}^{\tau,\control},\bm{D}^{\tau,\treatment},\bm{S}^\tau\right]$ in lieu of $\hat{\Delta}^\tau_{\text{RCT}}$. Under Assumption 1, we can always compute the former estimator as a lower-variance version of the latter estimator. In the main text, we discuss the latter estimator since it is more intuitive. However, in our results, we analyze the version with variance reduction so that we are always comparing our estimator to to the best possible standard estimator (in terms of variance).

\begin{proof}[Proof of Theorem~\ref{thm:rct-overestimate}.]
For some demand vector $\bm{d}$ and supply vector $\bm s$, let {$\mathcal{Y}(\bm{d}, \bm{s})$} denote the set of feasible solutions in path-aggregated form to Eq.~\eqref{eq:general-matching-formulation} with $\tau=1$ (fluid limit).

{Assume without loss of generality that $\bm{\beta}\ge \bm{0}$.} We first note that the optimal matching value for the experiment state verifies $\Psi(\rho)=\Phi(\bm{\lambda}+\rho\bm{\beta}, \bm{\pi})=\bar{\bm{v}}^*\cdot(\bm{\lambda}+\rho\bm{\beta})$, which follows from the definition of $\bar{\bm{v}}^*$.

\textbf{Step 1:} Starting from the optimal solution to the matching problem in the experiment state $x_{i,j}^*$, we build a ``scaled'' solution $y_{i,j}$ to the matching problem under global control, such that
{
\[
z_{i,j,p}=\frac{\lambda_{i}}{\ensuremath{\lambda_{i}+\rho\beta_{i}}}y_{i,j,p}^{*}.
\]
}

We can easily verify that {$z_{i,j,p}\in \mathcal{Y}(\bm{\lambda}, \bm{\pi})$} (feasibility), since:
\[
    \sum_{j=1}^{n_s}{\sum_{p=1}^{n_{i,j}}z_{i,j,p}}=\frac{\lambda_{i}}{\ensuremath{\lambda_{i}+\rho\beta_{i}}}\sum_{j=1}^{n_s}{\sum_{p=1}^{n_{i,j}}y^*_{i,j,p}}\le\frac{\lambda_{i}}{\ensuremath{\lambda_{i}+\rho\beta_{i}}}\left(\lambda_{i}+\rho\beta_{i}\right)=\lambda_{i},
\]
and
\[
    \sum_{i=1}^{n_d}{\sum_{p=1}^{n_{i,j}}z_{i,j,p}}=\sum_{i=1}^{n_d}{\sum_{p=1}^{n_{i,j}}}\frac{\lambda_{i}}{\ensuremath{\lambda_{i}+\rho\beta_{i}}}{y_{i,j,p}^{*}}\le\sum_{i=1}^{n_d}{\sum_{p=1}^{n_{i,j}}y_{i,j,p}^*}\le\pi_{j}.
\]

{The flow constraints are satisfied by construction since scaling applies to all variables equally. The capacity constraints are also satisfied because we are only decreasing flow values.}

The scaled solution {$z_{i,j,p}$} has objective
\[
    \sum_{i=1}^{n_d}\sum_{j=1}^{n_s}{\sum_{p=1}^{n_{i,j}}\nu_{i,j,p}z_{i,j,p}}=\sum_{i=1}^{n_d}\sum_{j=1}^{n_s}{\sum_{p=1}^{n_{i,j}}\nu_{i,j,p}y^*_{i,j,p}}\frac{\lambda_i}{\lambda_i+\rho\beta_i}=\sum_{i=1}^{n_d}\lambda_i\sum_{j=1}^{n_s}{\sum_{p=1}^{n_{i,j}}}\frac{{\nu_{i,j,p}y^*_{i,j,p}}}{\lambda_i+\rho\beta_i}=\bar{\bm{v}}^*\cdot\bm{\lambda},
\]

and it is feasible in $\mathcal{X}(\bm{\lambda}, \bm{\pi})$, therefore $\bar{\bm{v}}^*\cdot\bm{\lambda}\le \Phi(\bm{\lambda},\bm{\pi})=\Psi(0)$. In other words, the linear approximation of the partial-treatment value function constructed implicitly by the RCT estimator underestimates the value of global control.

\textbf{Step 2:} From the concavity of $\Psi(\cdot)$, we can write:
\begin{align*}
    \Psi(\rho)&\ge \rho\Psi(1)+(1-\rho)\Psi(0)\\
    \bar{\bm{v}}^*\cdot(\bm{\lambda}+\rho\bm{\beta}) & \ge \rho\Psi(1) + (1-\rho)\bar{\bm{v}}^*\cdot\bm{\lambda}\\
    \bar{\bm{v}}^*\cdot(\bm{\lambda}+\bm{\beta}) &\ge \Psi(1).
\end{align*}
In other words, the RCT estimator's implicit linear approximation also overestimates the value of global treatment.

Putting the two halves of the proof together, we obtain:
\[
    \hat{\Delta}_{\text{\emph{RCT}}}=\bar{\bm{v}}^*\cdot\bm{\beta}= \bar{\bm{v}}^*\cdot(\bm{\lambda}+\bm{\beta}) - \bar{\bm{v}}^*\cdot\bm{\lambda}\ge \Psi(1)-\Psi(0) = \Delta.
\]
\end{proof}

\subsection{{Analysis of the two-LP estimator}}

{
\begin{proof}[Proof of Proposition~\ref{prop:twolp-simplified}] The two-LP estimator is defined as:
\begin{multline*}
    \hat{\Delta}^\tau_{\text{2LP}} =\Phi\left(\hat{\bm{\lambda}}(\bm{D}^{\tau,\control},\bm{D}^{\tau,\treatment} )+\hat{\bm{\beta}}(\bm{D}^{\tau,\control},\bm{D}^{\tau,\treatment} ),\hat{\bm{\pi}}(\bm{S}^{\tau,\bm{\pi}})\right)-\\\Phi\left(\hat{\bm{\lambda}}(\bm{D}^{\tau,\control},\bm{D}^{\tau,\treatment} ),\hat{\bm{\pi}}(\bm{S}^{\tau,\bm{\pi}})\right),
\end{multline*}
where $\hat{\bm{\lambda}}(\cdot)$, $\hat{\bm{\beta}}(\cdot)$ and $\hat{\bm{\pi}}(\cdot)$ designate maximum-likelihood estimators for the arrival rates $\bm{\lambda}$, $\bm{\beta}$ and $\bm{\pi}$ obtained from the observed demand and supply counts. In particular,
\begin{align*}
    D_i^{\tau,\control}\sim \text{Poisson}((1-\rho)\lambda_i)&\Rightarrow \hat{\lambda}_i=\frac{D_i^{\tau,\control}}{(1-\rho)\tau}\\
    D_i^{\tau,\treatment}\sim \text{Poisson}(\rho(\lambda_i+\beta_i))&\Rightarrow \hat{\lambda}_i+\hat{\beta}_i=\frac{D_i^{\tau,\treatment}}{\rho\tau}\\
    S_j^{\tau,\pi_j}\sim \text{Poisson}(\pi_j)&\Rightarrow \hat{\pi}_j=\frac{S_j^{\tau,\pi_j}}{\tau}.
\end{align*} 
Using the equations above, we can write:
\[
\hat{\Delta}^\tau_{\text{2LP}} =\Phi\left(\frac{1}{\rho\tau}\bm{D}^{\tau,\treatment}, \frac{1}{\tau}\bm{S}^{\tau,\bm{\pi}}\right)-\Phi\left(\frac{1}{(1-\rho)\tau}\bm{D}^{\tau,\control},\frac{1}{\tau}\bm{S}^{\tau,\bm{\pi}}\right),
\]
then apply Lemma~\ref{lem:scaled-problem} to complete the proof.
\end{proof}

\begin{proof}[Proof of Theorem~\ref{thm:twolp-unbiased}]
    The result follows directly from Proposition~\ref{prop:twolp-simplified} and Theorem~\ref{thm:fluid-limit}.
\end{proof}

\begin{proof}[Proof of Proposition~\ref{prop:twolp-worst-case}]
    Let $n_d=n_s=n \gg 1$. We assume the total demand arrival rate is $1$, evenly divided among the demand types, i.e. $\lambda_i=\frac{1}{n}$ for all $i$. We also assume the treatment effect on demand is $\beta_i=\lambda_i=\frac{1}{n}$. The arrival rate of each supply type $i$ is $\pi_i=m\gg 1$. In particular, we choose $n\ge n_0$ and $m\ge m_{n}$ such that we can use the following approximations:
    \begin{align*}
        1-e^{-m} &= 1 - o\left(\frac{1}{n^2}\right),\\
        1-e^{-\frac{1}{n}} &= \frac{1}{n} - \frac{1}{2n^2} \pm o\left(\frac{1}{n^2}\right),\\
        1-e^{-\frac{2}{n}} &= \frac{2}{n} - \frac{2}{n^2} \pm o\left(\frac{1}{n^2}\right),\\
        \revision{
        1-e^{-\frac{1-\rho}{n}}} &= \revision{\frac{1-\rho}{n} - \frac{(1-\rho)^2}{2n^2} \pm o\left(\frac{1}{n^2}\right),}\\
        \revision{
        1-e^{-\frac{2\rho}{n}}} &= \revision{\frac{2\rho}{n} - \frac{2\rho^2}{n^2} \pm o\left(\frac{1}{n^2}\right),}
    \end{align*}
    where the last \revision{four} equations come from the power series expansion of the exponential function.
    
    We assume that the matching network is bipartite (no intermediate nodes), and all edges have capacity $1$. Let $\tau=1$ to simplify notation. The value function is such that:
    \[
        v_{i,j}=\begin{cases}
            1 & \text{if } i=j,\\
            0 & \text{if } i\neq j.
        \end{cases} 
    \]
    This value function can be decomposed into the sum of the value obtained from each type. Under global control, the expected value obtained from type $i$ is equal to 1 if at least one demand unit and at least one supply unit of type $i$ arrive, and 0 otherwise. In other words,
    \begin{align*}
        \mathbb{E}\left[\Phi\left(\bm{D}^{\tau,\bm{\lambda}},\bm{S}^{\tau,\bm{\pi}}\right)\right]&=\sum_{i=1}^n \mathbb{P}\left[D_i^{\tau,\lambda_i}\ge 1\right]\mathbb{P}\left[S_i^{\tau,\pi_i}\ge 1\right]\\
        &=\sum_{i=1}^n (1-\mathbb{P}\left[D_i^{\tau,\lambda_i}= 0\right])(1-\mathbb{P}\left[S_i^{\tau,\pi_i}=0\right])\\
        &=\sum_{i=1}^n (1-e^{-\frac{1}{n}})(1-e^{-m})\\
        &=\left(1-o\left(\frac{1}{n^2}\right)\right)\sum_{i=1}^n \left(\frac{1}{n} -\frac{1}{2n^2}\pm o\left(\frac{1}{n^2}\right)\right)\\
        &=1 -\frac{1}{2n}\pm o\left(\frac{1}{n}\right),
    \end{align*}
    Similarly, for global treatment,
    \begin{align*}
        \mathbb{E}\left[\Phi\left(\bm{D}^{\tau,\bm{\lambda}+\bm{\beta}},\bm{S}^{\tau,\bm{\pi}}\right)\right]&=\sum_{i=1}^n \mathbb{P}\left[D_i^{\tau,\lambda_i+\beta_i}\ge 1\right]\mathbb{P}\left[S_i^{\tau,\pi_i}\ge 1\right]\\
        &=\sum_{i=1}^n (1-e^{-\frac{2}{n}})(1-e^{-m})\\
        &=\left(1-o\left(\frac{1}{n^2}\right)\right)\sum_{i=1}^n \left(\frac{2}{n} -\frac{2}{n^2}\pm o\left(\frac{1}{n^2}\right)\right)\\
        &=2 -\frac{2}{n}\pm o\left(\frac{1}{n}\right).
    \end{align*}
    Therefore the true global treatment effect is given by $$\Delta^\tau=1 -\frac{3}{2n}\pm o\left(\frac{1}{n}\right).$$

    Now consider an experiment with \revision{treatment fraction $\rho$}. For each demand type, we observe $D_i^{\tau,\control}\sim\text{Poisson}\left(\frac{\revision{1-\rho}}{n}\right)$, and $D_i^{\tau,\treatment}\sim\text{Poisson}\left(\frac{\revision{2\rho}}{n}\right)$. On average, the two-LP estimate for global control is therefore given by:
    \begin{align*}
        \mathbb{E}\left[\Phi\left(\frac{1}{1-\rho}\bm{D}^{\tau,\control}, \bm{S}^{\tau,\bm{\pi}}\right)\right]
        &=\sum_{i=1}^n \mathbb{P}\left[\revision{\frac{1}{1-\rho}}D_i^{\tau,\control}\ge 1\right]\mathbb{P}\left[S_i^{\tau,\pi_i}\ge 1\right]\\
        &=\sum_{i=1}^n \left(1-\mathbb{P}\left[D_i^{\tau,\control}= 0\right]\right)\left(1-\mathbb{P}\left[S_i^{\tau,\pi_i}=0\right]\right)\\
        &=\sum_{i=1}^n (1-e^{-\frac{\revision{1-\rho}}{n}})(1-e^{-m})\\
        &=\left(1-o\left(\frac{1}{n^2}\right)\right)\sum_{i=1}^n \revision{\left(\frac{1-\rho}{n} - \frac{(1-\rho)^2}{2n^2} \pm o\left(\frac{1}{n^2}\right)\right)}\\
        &=\revision{1-\rho -\frac{(1-\rho)^2}{2n}}\pm o\left(\frac{1}{n}\right).
    \end{align*}

    Similarly, the two-LP estimate for global treatment is given by:
    \begin{align*}
        \mathbb{E}\left[\Phi\left(\frac{1}{\rho}\bm{D}^{\tau,\treatment}, \bm{S}^{\tau,\bm{\pi}}\right)\right]
        &=\sum_{i=1}^n \mathbb{P}\left[\revision{\frac{1}{\rho}}D_i^{\tau,\treatment}\ge 1\right]\mathbb{P}\left[S_i^{\tau,\pi_i}\ge 1\right]\\
        &=\sum_{i=1}^n \left(1-\mathbb{P}\left[D_i^{\tau,\treatment}= 0\right]\right)\left(1-\mathbb{P}\left[S_i^{\tau,\pi_i}=0\right]\right)\\
        &=\sum_{i=1}^n (1-e^{-\frac{\revision{2\rho}}{n}})(1-e^{-m})\\
        &=\left(1-o\left(\frac{1}{n^2}\right)\right)\sum_{i=1}^n \left(\revision{\frac{2\rho}{n} - \frac{2\rho^2}{n^2} \pm o\left(\frac{1}{n^2}\right)}\right)\\
        &=\revision{2\rho -\frac{2\rho^2}{n}}\pm o\left(\frac{1}{n}\right).
    \end{align*}
    Putting together the above two results yields:
    \revision{\[
    \mathbb{E}\left[\hat{\Delta}_{\text{2LP}}\right] = 2\rho - \frac{4\rho^2}{2n} - 1 + \rho + \frac{(1-\rho)^2}{2n} \pm o\left(\frac{1}{n}\right)=(3\rho-1)\left(1-\frac{1+\rho}{2n}\right)\pm o\left(\frac{1}{n}\right).
    \]}
    \revision{
    Henceforth ignoring the $o(1/n)$ terms, and noting that
    \[
        \Delta^\tau = 1 - \frac{3}{2n} \le 1 - \frac{1+\rho}{2n} \le 1 - \frac{1}{2n} = \Delta^\tau + \frac{1}{n},
    \]
    we can bound the expected value of the two-LP estimator as follows:
    \begin{align*}
        \rho \ge \frac{1}{3} & \Rightarrow (3\rho - 1)\Delta^\tau \le \mathbb{E}\left[\hat{\Delta}_{\text{2LP}}\right]\le (3\rho - 1)\left(\Delta^\tau + \frac{1}{n}\right),\\
        \rho < \frac{1}{3} &\Rightarrow (3\rho - 1)\left(\Delta^\tau + \frac{1}{n}\right) \le \mathbb{E}\left[\hat{\Delta}_{\text{2LP}}\right]\le(3\rho - 1)\Delta^\tau.
    \end{align*}
    }
\end{proof}
}

\subsection{{Analysis of the SP estimator}}

\subsubsection{{Fluid limit}}

\begin{proof}[Proof of Proposition~\ref{prop:mdv-estimator-limit}.]
Analogously to Proposition~\ref{prop:infinite-horizon-limit-rct}, the result follows from the strong law of large numbers, Theorem~\ref{thm:fluid-limit}, Lemma~\ref{lem:max-scaled-pois}, and the dominated convergence theorem.
\end{proof}

\begin{proof}[Proof of Theorem~\ref{thm:main}.]
We first assume $\bm{\beta}\ge 0$. Recall from Proposition~\ref{prop:concave-integrable} that $\Psi(\cdot)$ is concave, implying that $\Psi'(\eta)=\bm{a}^\eta\cdot\bm{\beta}$ is decreasing in $\eta$. We can therefore separate the SP error into a positive and negative term:

\begin{align*}
\hat{\Delta}_{\text{SP}}-\Delta & =\int_{0}^{1}\left(\mathbf{a}^{\rho}-\mathbf{a}^{\eta}\right)\cdot\bm{\beta}d\eta =\underbrace{\int_{0}^{\rho}\left(\mathbf{a}^{\rho}-\mathbf{a}^{\eta}\right)\cdot\bm{\beta}d\eta}_{-\delta_{-}\le 0} +\underbrace{\int_{\rho}^{1}\left(\mathbf{a}^{\rho}-\mathbf{a}^{\eta}\right)\cdot\bm{\beta}d\eta}_{\delta_+\ge 0}
\end{align*}
We call the positive and negative parts above $\delta_{-}\ge 0$ and
$\delta_{+}\ge 0$, such that $\hat{\Delta}_{\text{SP}}-\Delta=\delta_{+}-\delta_{-}$.

Recalling step 1 of the proof of Theorem~\ref{thm:rct-overestimate}, we know that 
\[
\int_{0}^{\rho}\mathbf{a}^{\eta}\cdot\bm{\beta}d\eta=\Psi(\rho)-\Psi(0)\le \rho\hat{\Delta}_{\text{RCT}},
\]
therefore
\[
\rho\hat{\Delta}_{\text{RCT}}\ge\int_{0}^{\rho}\mathbf{a}^{\eta}\cdot\bm{\beta}d\eta=\int_{0}^{\rho}\mathbf{a}^{\rho}\cdot\bm{\beta}+\delta_{-}=\rho\mathbf{a}^{\rho}\cdot\bm{\beta}+\delta_{-},
\]
which we can re-write as
\[
\hat{\Delta}_{\text{RCT}}\ge \hat{\Delta}_{\text{SP}} + \frac{1}{\rho}\delta_-.
\]
The above result tells us that not only is the RCT estimator always larger than the SP estimator, but in fact the difference exceeds the (scaled) possible negative bias of the SP estimator.

We can then derive the following bound
\begin{align*}
\left|\hat{\Delta}_{\text{\emph{RCT}}}-\Delta\right|=\hat{\Delta}_{\text{RCT}}-\Delta & =\int_{0}^{1}\left(\hat{\Delta}_{\text{RCT}}-\mathbf{a}^{\eta}\cdot\bm{\beta}\right)d\eta\\
 & \ge0+\int_{\rho}^{1}\left(\mathbf{a}^{\rho}\cdot\bm{\beta}+\frac{1}{\rho}\delta_{-}-\mathbf{a}^{\eta}\cdot\bm{\beta}\right)d\eta\\
 & =\delta_{+}+\frac{\left(1-\rho\right)}{\rho}\delta_{-}
\end{align*}

When $\rho\le 0.5$ we obtain
\[
\left|\hat{\Delta}_{\text{RCT}}-\Delta\right|\ge\delta_{+}+\delta_{-}\ge\left|\delta_{+}-\delta_{-}\right|=\left|\hat{\Delta}_{\text{SP}}-\Delta\right|.
\]

We can repeat the proof with $\bm{\beta}\le 0$ and obtain the same result with $\rho\ge0.5$. A symmetric experiment design ($\rho=0.5$) will therefore exhibit less bias with the SP estimator than the RCT estimator as long as the treatment effect is consistent across types.
\end{proof}

\subsubsection{Variance results}

\begin{proof}[Proof of Theorem~\ref{thm:mdv-variance}.]
By the law of total variance, we can decompose the RCT estimator
variance into two components, one coming from the demand arrival rates,
and the other coming from the randomized treatment assignments.

\begin{align}
    \Var{\sqrt{\tau}\hat{\Delta}_{\text{RCT}}^{\tau}}=&\;\mathbb{E}\left[\Var{\sqrt{\tau}\hat{\Delta}_{\text{RCT}}^{\tau}|\bm{D}^{\tau,\control},\bm{D}^{\tau,\treatment}}\right]\nonumber\\
    &+\Var{\mathbb{E}\left[\sqrt{\tau}\hat{\Delta}_{\text{RCT}}^{\tau}|\bm{D}^{\tau,\control},\bm{D}^{\tau,\treatment}\right]}\nonumber\\
    \ge&\; \Var{\sqrt{\tau}\mathbb{E}\left[\hat{\Delta}_{\text{RCT}}^{\tau}|\bm{D}^{\tau,\control},\bm{D}^{\tau,\treatment}\right]}.\label{eq:total-variance}
\end{align}

From the proof of Proposition~\ref{prop:infinite-horizon-limit-rct}, we know that we can express the expectation on the right-hand side in Eq.~\eqref{eq:total-variance} as
\[
\mathbb{E}\left[\hat{\Delta}_{\text{RCT}}^{\tau}|\bm{D}^{\tau,\control},\bm{D}^{\tau,\treatment}\right]=\frac{1}{\tau}\bar{\mathbf{V}}^{\tau}\cdot\left(\frac{1}{\rho}\bm{D}^{\tau,\control}-\frac{1}{1-\rho}\bm{D}^{\tau,\treatment}\right),
\]
where $\bar{V}_i^{\tau}$ denotes the average value obtained from demand type $i$. Therefore, we can write
\[
    \lim_{\tau\to\infty}\Var{\sqrt{\tau}\hat{\Delta}_{\text{RCT}}^{\tau}}\ge
    \lim_{\tau\to\infty}\Var{\bar{\mathbf{V}}^{\tau}\cdot\frac{1}{\sqrt{\tau}}\left(\frac{1}{\rho}\bm{D}^{\tau,\control}-\frac{1}{1-\rho}\bm{D}^{\tau,\treatment}\right)}
\]
As in the proof of Proposition~\ref{prop:infinite-horizon-limit-rct}, we can show that $\bar{\bm{V}}^\tau$ converges almost surely to $\bar{\bm{v}}^*$, the vector of average values by demand type in the fluid limit. Therefore, the above is equivalent to
\begin{align*}
    \lim_{\tau\to\infty}\Var{\sqrt{\tau}\hat{\Delta}_{\text{RCT}}^{\tau}}&\ge
    \lim_{\tau\to\infty}\Var{\bar{\mathbf{v}}^{*}\cdot\frac{1}{\sqrt{\tau}}\left(\frac{1}{\rho}\bm{D}^{\tau,\control}-\frac{1}{1-\rho}\bm{D}^{\tau,\treatment}\right)}\\
    &=\lim_{\tau\to\infty}\sum_{i=1}^{n_d}\Var{\bar{v}_i^{*}\cdot\frac{1}{\sqrt{\tau}}\left(\frac{1}{\rho}{D}_i^{\tau,\control}-\frac{1}{1-\rho}{D}_i^{\tau,\treatment}\right)}\\
    &=\sum_{i=1}^{n_d}\left(\bar{v}_i^*\right)^2\lim_{\tau\to\infty}\Var{\frac{1}{\sqrt{\tau}}\left(\frac{1}{\rho}{D}_i^{\tau,\control}-\frac{1}{1-\rho}{D}_i^{\tau,\treatment}\right)}\\
    &=\sum_{i=1}^{n_d}(\bar{v}_i^*)^2 \zeta_i,
\end{align*}
where 
\[
\zeta_i:=\lim_{\tau\to\infty}\Var{\frac{1}{\sqrt{\tau}}\left(\frac{1}{\rho}{D}_i^{\tau,\control}-\frac{1}{1-\rho}{D}_i^{\tau,\treatment}\right)}\ge 0
\]
is bounded as a result of asymptotic properties of Poisson distributions.

Similarly, we can decompose the asymptotic variance of the SP estimator:
\begin{align*}
    \Var{\sqrt{\tau}\hat{\Delta}_{\text{SP}}^{\tau}}&=\mathbb{E}\left[\Var{\sqrt{\tau}\hat{\Delta}_{\text{SP}}^{\tau}|\bm{D}^{\tau,\control},\bm{D}^{\tau,\treatment}}\right]+\Var{\mathbb{E}\left[\sqrt{\tau}\hat{\Delta}_{\text{SP}}^{\tau}|\bm{D}^{\tau,\control},\bm{D}^{\tau,\treatment}\right]}\\
    &= \Var{\mathbb{E}\left[\sqrt{\tau}\hat{\Delta}_{\text{SP}}^{\tau}|\bm{D}^{\tau,\control},\bm{D}^{\tau,\treatment}\right]}.
\end{align*}
We also know that the shadow prices $\bm{A}^\tau$ converge almost surely to the shadow prices in the fluid limit $\bm{a}^*$, so we can write:
\begin{align*}
    \lim_{\tau\to\infty}\Var{\sqrt{\tau}\hat{\Delta}_{\text{SP}}^{\tau}}&=
    \lim_{\tau\to\infty}\Var{\mathbf{a}^{*}\cdot\frac{1}{\sqrt{\tau}}\left(\frac{1}{\rho}\bm{D}^{\tau,\control}-\frac{1}{1-\rho}\bm{D}^{\tau,\treatment}\right)}\\
    &=\sum_{i=1}^{n_d}\left(a_i^*\right)^2\lim_{\tau\to\infty}\Var{\frac{1}{\sqrt{\tau}}\left(\frac{1}{\rho}{D}_i^{\tau,\control}-\frac{1}{1-\rho}{D}_i^{\tau,\treatment}\right)}\\
    &=\sum_{i=1}^{n_d}(a_i^*)^2 \zeta_i.
\end{align*}

From complementary slackness, we know that the average value of demand type $i$ must be no smaller than its marginal value, i.e., $a_i^*\le \bar{v}_i^*$. Therefore we conclude $(a_i^*)^2\le(\bar{v}_i^*)^2$ for all $i$ and therefore
\[
\lim_{\tau\to\infty}\Var{\sqrt{\tau}\hat{\Delta}_{\text{RCT}}^{\tau}}\ge \lim_{\tau\to\infty}\Var{\sqrt{\tau}\hat{\Delta}_{\text{SP}}^{\tau}}.
\]

\end{proof}

\subsubsection{{Imbalanced bipartite matching markets}}

\begin{proof}[Proof of Theorem~\ref{thm:undersupply-oversupply}.] Consider a sequence $\{\bar{\lambda}^k/\bar{\pi}^k\}_{k=1}^{k=\infty}$. Let $\bm{\lambda}^k$ and $\bm{\pi}^k$ designate the demand and supply arrival rate vectors corresponding to scale parameters $\bar{\lambda}^k$ and $\bar{\pi}^k$. Similarly, let $\bm{\beta}^k$ designate the treatment effect on demand corresponding to demand scaling $\bar{\lambda}$.

\textbf{Oversupply limit:} Assume $\{\bar{\lambda}^k/\bar{\pi}^k\}_k$ converges to 0. We can find $k_0$ such that for all $k\ge k_0$,
\[
\min_{j\in[n_s]}\pi_j^k > \sum_{i=1}^{n_d}\lambda_i^k+\beta_i^k.
\]
For any $\eta\in[0,1]$, this result implies that the supply constraint~\eqref{eq:matching-supply} for each supply type $j$ in the matching problem with demand arrival rate $\bm{\lambda}+\eta\bm{\beta}$ verifies
\[
\sum_{i=1}^{n_d} x_{i,j}\le \sum_{i=1}^{n_d}\lambda_i^k+\eta\beta_i^k\le\sum_{i=1}^{n_d}\lambda_i^k+\beta_i^k<\min_{j'\in[n_s]}\pi_{j'}^k\le\pi_{j}^k.
\]
From complementary slackness, we obtain $b^\eta_j=0$ for each supply type $j$. We then observe that the optimal demand dual variable is given by $a^\eta_i=\max_{j\in[n_s]}v_{i,j}$ for each demand type $i$, for all $\eta\in[0,1]$. Furthermore, for the experiment state ($\eta=\rho$), complementary slackness means the optimal primal variables are given by
\[
x^*_{i,j}=\begin{cases}
\lambda_i^k+\rho \beta_i^k, & \text{if } j = \arg\max_{j'\in[n_s]}
v_{i,j'},\\
0, & \text{otherwise}.
\end{cases}
\]

By proposition~\ref{prop:fundamental-theorem-calculus}, we can write the true treatment effect as
\[
\Delta = \int_0^1\bm{a}^{\eta}\cdot\bm{\beta}d\eta=\int_0^1d\eta\sum_{i=1}^{n_d}\max_{j\in[n_s]}v_{i,j}\beta_i=\sum_{i=1}^{n_d}\max_{j\in[n_s]}v_{i,j}\beta_i.
\]
By proposition~\ref{prop:mdv-estimator-limit}, the SP estimator is given by
\[
\hat{\Delta}_{\text{SP}}=\bm{a}^{\rho}\cdot\bm{\beta}=\sum_{i=1}^{n_d}\max_{j\in[n_s]}v_{i,j}\beta_i=\Delta,
\]
and by proposition~\ref{prop:infinite-horizon-limit-rct}, the RCT estimator is given by
\[
\hat{\Delta}_{\text{RCT}}=\bar{\bm{v}}^*\cdot\bm{\beta}=\sum_{i=1}^{n_d}\frac{\sum_{j=1}^{n_s}x_{i,j}^*v_{i,j}}{\lambda_i^k+\rho\beta_i^k}=
\sum_{i=1}^{n_d}\max_{j\in[n_s]}v_{i,j}\beta_i=\Delta.
\]

\textbf{Undersupply limit:} Now assume $\{\bar{\lambda}^k/\bar{\pi}^k\}_k\to\infty$. We can find $k_0$ such that for all $k\ge k_0$,
\[
\min_{i\in[n_d]}\lambda_i^k > \sum_{j=1}^{n_s}\pi_j^k.
\]
Using a similar complementary slackness argument as in the first part of the proof, we can show that for any $\eta\in[0,1]$, $a_i^\eta=0$ for each demand type $i$, and additionally the optimal primal solution for the experiment state ($\eta=\rho$) can be written as
\[
\bar{v}^*_i=\frac{\sum_{j=1}^{n_s}\delta_{i,j}\pi^k_jv_{i,j}}{\lambda_i+\rho\beta_i},~\text{with}~\delta_{i,j}=\begin{cases}
1, & \text{if } i = \arg\max_{i'\in[n_d]}v_{i',j},\\
0, & \text{otherwise}.
\end{cases}
\]

Once again applying Propositions~\ref{prop:infinite-horizon-limit-rct}, \ref{prop:mdv-estimator-limit}, and \ref{prop:fundamental-theorem-calculus}, we obtain
\[
\Delta=\int_0^1\bm{a}^{\eta}\cdot\bm{\beta}d\eta=0=\bm{a}^\rho\cdot\bm{\beta}=\hat{\Delta}_{\text{SP}},
\]
and 
\begin{align*}
\hat{\Delta}_{\text{RCT}}&=\sum_{i=1}^{n_d}\bar{v}^*_i\beta_i=\sum_{i=1}^{n_d}\frac{\sum_{j=1}^{n_s}\delta_{i,j}\pi^k_jv_{i,j}}{\lambda_i+\rho\beta_i}\beta_i\\
&=\sum_{i=1}^{n_d}\frac{\sum_{j=1}^{n_s}\delta_{i,j}\gamma_j\bar{\pi}^kv_{i,j}}{\bar{\lambda}^k\alpha_ip_i+\rho\bar{\lambda}^k\alpha_iq_i}\bar{\lambda}^k\alpha_iq_i\\
&=\sum_{i=1}^{n_d}\sum_{j=1}^{n_s}\delta_{i,j}v_{i,j}\gamma_j\bar{\pi}^k\frac{q_i}{p_i+\rho q_i}.
\end{align*}
The right-hand side above does not tend to 0 in general as $\bar{\lambda}^k/\bar{\pi}^k\to\infty$. In particular, if $\bar{\pi}^k$ converges to a positive constant (and $\bar{\lambda}^k\to\infty$), then $\hat{\Delta}_{\text{RCT}}>0$ in the limit.
\end{proof}

\subsubsection{{Finite-sample analysis}}

{
The goal of this section is to build up to the proof of Theorem~\ref{thm:shadow-price-derivative}. We first prove Theorem~\ref{thm:directional-derivative}, which establishes the existence of a positive-demand and negative-demand dual or subgradient of the matching value function $\Phi_\tau(\cdot,\cdot)$.

\begin{proof}[Proof of Theorem~\ref{thm:directional-derivative}]
Let us first assume that $\bm{\epsilon}$ is a rational, and that
$\epsilon_{i}=\frac{f_{i}}{N}$ for some $N\in Z_{+}$ and $f_{i}\in[N]$.
Using Lemma \ref{lem:scaled-problem}  we rewrite
{
\[
\Phi_\tau\left(\mathbf{d},\mathbf{s}\right)=\frac{1}{N}\Phi_{N\tau}\left(N\mathbf{d},N\mathbf{s}\right)
\]
 and 
\[
\Phi_\tau\left(\mathbf{d}+\bm{\epsilon},\mathbf{s}\right)=\frac{1}{N}\Phi_{N\tau}\left(N\mathbf{d}+\mathbf{f},N\mathbf{s}\right).
\]
}

In order to get from an optimal solution {$x_{u,w}^{*}$ of $\Phi_\tau\left(\mathbf{d},\mathbf{s}\right)$
to an optimal solution of $\Phi_\tau\left(\mathbf{d}+\mathbf{e}_{k},\mathbf{s}\right)$,
for some $k\in\left[n_{d}\right]$, we know that we need to find an
augmenting path, denoted by $z_{u,w}^{k}$, such that $x_{u,w}^{*}+z^k_{u,w}$
is an optimal solution of $\Phi_\tau\left(\mathbf{d}+\mathbf{e}_{k},\mathbf{s}\right)$.} 

From diminishing returns we know that if $f_{i}>0$ then
\[
\Phi_{{N\tau}}\left(N\mathbf{d}+\mathbf{f},N\mathbf{s}\right)-\Phi_{{N\tau}}\left(N\mathbf{d}+\mathbf{f}-\mathbf{e}_{k},N\mathbf{s}\right)\le\Phi_{{N\tau}}\left(N\mathbf{d}+\mathbf{e}_{k},N\mathbf{s}\right)-\Phi_{{N\tau}}\left(N\mathbf{d},N\mathbf{s}\right).
\]

To {compute} the right-hand side, we observe that to get from the optimal
solution of $\Phi_{{N\tau}}\left(N\mathbf{d},N\mathbf{s}\right)$ to the optimal
solution of $\Phi_{{N\tau}}\left(N\mathbf{d}+\mathbf{e}_{k},N\mathbf{s}\right)$, we follow the same {augmenting path $z_{u,w}^{k}$}, thus,
\[
\Phi_{{N\tau}}\left(N\mathbf{d}+\mathbf{f},N\mathbf{s}\right)-\Phi_{{N\tau}}\left(N\mathbf{d}+\mathbf{f}-\mathbf{e}_{k},N\mathbf{s}\right)\le\Phi_{{\tau}}\left(\mathbf{d}+\mathbf{e}_{k},\mathbf{s}\right)-\Phi_{{\tau}}\left(\mathbf{d},\mathbf{s}\right).
\]
 and by diminishing returns
\[
\Phi_{{N\tau}}\left(N\mathbf{d}+\mathbf{f},N\mathbf{s}\right)-\Phi_{{N\tau}}\left(N\mathbf{d},N\mathbf{s}\right)\le\sum_{k=1}^{n_{d}}f_{k}\left(\Phi_{{\tau}}\left(\mathbf{d}+\mathbf{e}_{k},\mathbf{s}\right)-\Phi_{{\tau}}\left(\mathbf{d},\mathbf{s}\right)\right).
\]
 Using Lemma \ref{lem:scaled-problem}  again, we see that
\[
\Phi_{{\tau}}\left(\mathbf{d}+\mathbf{\bm{\epsilon}},\mathbf{s}\right)-\Phi_{{\tau}}\left(\mathbf{d},\mathbf{s}\right)\le\sum_{k=1}^{n_{d}}\epsilon_{k}\left(\Phi_{{\tau}}\left(\mathbf{d}+\mathbf{e}_{k},\mathbf{s}\right)-\Phi_{{\tau}}\left(\mathbf{d},\mathbf{s}\right)\right),
\]
{which we can also write as}
\[
\Phi_{{\tau}}\left(\mathbf{d}+\mathbf{\bm{\epsilon}},\mathbf{s}\right)\le\left(1-\sum_{k=1}^{n_{d}}\epsilon_{k}\right)\Phi_{{\tau}}\left(\mathbf{d},\mathbf{s}\right)+\sum_{k=1}^{n_{d}}\epsilon_{k}\Phi_{{\tau}}\left(\mathbf{d}+\mathbf{e}_{k},\mathbf{s}\right).
\]
To show the equality is strict we observe that the concavity of $\Phi_{{\tau}}\left(\mathbf{d},\mathbf{s}\right)$
implies that
\begin{align*}
\Phi_{{\tau}}\left(\mathbf{d}+\mathbf{\bm{\epsilon}},\mathbf{s}\right) & =\Phi_{{\tau}}\left(\mathbf{d}+\sum_{k=1}^{n_{d}}\epsilon_{k}\mathbf{e}_{k},\mathbf{s}\right)\\
 & =\Phi_{{\tau}}\left(\left(1-\sum_{k=1}^{n_{d}}\epsilon_{k}\right)\mathbf{d}+\sum_{k=1}^{n_{d}}\epsilon_{k}\left(\mathbf{d}+\mathbf{e}_{k}\right),\mathbf{s}\right)\\
 & \ge\left(1-\sum_{k=1}^{n_{d}}\epsilon_{k}\right)\Phi_{{\tau}}\left(\mathbf{d},\mathbf{s}\right)+\sum_{k=1}^{n_{d}}\epsilon_{k}\Phi_{{\tau}}\left(\mathbf{d}+\mathbf{e}_{k},\mathbf{s}\right).
\end{align*}

{Continuity of} $\Phi_{{\tau}}\left(\cdot,\cdot \right)$ implies {that the result extends to real-valued $\bm{\epsilon}$.} {A similar proof holds for $\Phi_{\tau}(\mathbf{d},\mathbf{s})-\Phi_{\tau}\left(\mathbf{d}-{\bm{\epsilon}},\mathbf{s}\right)$.}
\end{proof}

Having proven Theorem~\ref{thm:directional-derivative}, we now introduce three lemmas on the form of the derivative of a function of Poisson random variables.

\begin{lemma}
    \label{lem:derivative-poisson-rate}
    Suppose $D\sim \text{Poisson}(\lambda)$, with $0 < \lambda < \infty$, and $\bm{x}\in\mathbb{R}^m$ be a constant vector. Let $g(\bm{x},d)$ be a continuous function satisfying $0\le g(\bm{x},d) \le M \left(\abs{d} + \norm{\bm{x}}_1\right)$. Then $\mathbb{E}\left[g(\bm{x}, D)\right]$ is differentiable with respect to $\lambda$, and in particular
    \[
    \frac{\partial}{\partial\lambda}\mathbb{E}\left[g(\bm{x}, D)\right]=\mathbb{E}\left[g(\bm{x}, D+1) - g(\bm{x}, D)\right].
    \]
\end{lemma}

We note that the expectation above is taken only with respect to $D$, while keeping $\bm{x}$ constant.

\begin{proof}[Proof of Lemma~\ref{lem:derivative-poisson-rate}]
    We use a constructive argument for the proof. In particular, using the exact form of the Poisson probability mass function, we can write
    \[
        \mathbb{E}\left[g(\bm{x}, D)\right]=\sum_{d=0}^{\infty}g(\bm{x}, d)\mathbb{P}(D=d) = \sum_{d=0}^{\infty}g(\bm{x}, d)\frac{e^{-\lambda}\lambda^d}{d!}.
    \]
    We need to establish that the derivative of $\mathbb{E}\left[g(\bm{x}, D)\right]$ with respect to $\lambda$ exists. For any finite $n \ge 0$, we define the sequence of partial sums:
    \[
        f_n(\lambda) = \sum_{d=0}^ng(\bm{x},d)\frac{e^{-\lambda}\lambda^d}{d!},
    \]
    such that $\mathbb{E}\left[g(\bm{x}, D)\right] = \lim_{n\to\infty}f_n(\lambda)$. We then compute the derivative:
    \begin{align*}
        f'_n(\lambda) &= \frac{\partial}{\partial\lambda}\sum_{d=0}^ng(\bm{x},d)\frac{e^{-\lambda}\lambda^d}{d!}\\
        &=\sum_{d=0}^n\frac{g(\bm{x},d)}{d!}\frac{\partial}{\partial\lambda} \left(e^{-\lambda}\lambda^d\right)\\
        &=\sum_{d=0}^n\frac{g(\bm{x},d)}{d!}\left(d\lambda^{d-1}e^{-\lambda} - \lambda^de^{-\lambda}\right)\\
        &=\sum_{d=1}^ng(\bm{x},d)\frac{\lambda^{d-1}e^{-\lambda}}{(d-1)!} - \sum_{d=0}^ng(\bm{x},d)\frac{e^{-\lambda}\lambda^d}{d!}\\
        &=\sum_{d=0}^ng(\bm{x},d+1)\frac{\lambda^{d}e^{-\lambda}}{d!} - \sum_{d=0}^ng(\bm{x},d)\frac{e^{-\lambda}\lambda^d}{d!}\\
        &= g_n(\lambda) - f_n(\lambda),
    \end{align*}
    where we define
    \[
        g_n(\lambda) = \sum_{d=0}^ng(\bm{x},d+1)\frac{\lambda^{d}e^{-\lambda}}{d!}.
    \]

    Because $g(\bm{x}, d)\le M(\norm{x}_1+\abs{d})$, we can use the Weierstrass M-test to establish uniform convergence of $f_n$ to $\mathbb{E}\left[g(\bm{x}, D)\right]$, and of $g_n$ to $\mathbb{E}\left[g(\bm{x}, D + 1)\right]$. Therefore $f'_n$ also converges uniformly to the derivative of $\mathbb{E}\left[g(\bm{x}, D)\right]$, and we have that
    \[
    \frac{\partial}{\partial\lambda}\mathbb{E}\left[g(\bm{x}, D)\right]=\lim_{n\to\infty}(g_n(\lambda) - f_n(\lambda))=\mathbb{E}\left[g(\bm{x}, D+1) - g(\bm{x}, D)\right].
    \]
\end{proof}
}

{
\begin{lemma}
\label{lem:intermediate-derivative}
    Suppose demand $D_i$ for type $i$ is Poisson-distributed with parameter $\lambda_i$ and supply $S_j$ for type $j$ is Poisson distributed with parameter $\pi_j$. If $\lambda_i > 0$, we can write
    \[
    \frac{\partial}{\partial\lambda_i}\mathbb{E}\left[\Phi(\bm{D},\bm{S})\right]=\mathbb{E}\left[A_i^+\right],
    \]
    where $A_i^+$ is the positive-demand dual or subgradient associated with the optimal solution $\Phi(\bm{D},\bm{S})$, as defined in Theorem~\ref{thm:directional-derivative}.
\end{lemma}

\begin{proof}[Proof of Lemma~\ref{lem:intermediate-derivative}]
    Let $i=1$ without loss of generality. By the law of total expectation, we have
    \[
    \mathbb{E}\left[\Phi(\bm{D},\bm{S})\right]=\mathbb{E}_{D_1}\left[\mathbb{E}_{\bm{D}_{2:n_d},\bm{S}}\left[\Phi(\bm{D},\bm{S})|D_1\right]\right].
    \]
    Setting $g(\bm{x},d)=\mathbb{E}_{\bm{D}_{2:n_d},\bm{S}}\left[\Phi(\bm{D},\bm{S})|D_1=d\right]$, we can apply Lemma~\ref{lem:derivative-poisson-rate} to obtain
    \begin{align*}
        \frac{\partial}{\partial\lambda_1}\mathbb{E}\left[\Phi(\bm{D},\bm{S})\right] &= \mathbb{E}_{D}\left[\mathbb{E}_{\bm{D}_{2:n_d},\bm{S}}\left[\Phi(\bm{D},\bm{S})|D_1=D+1\right]-\mathbb{E}_{\bm{D}_{2:n_d},\bm{S}}\left[\Phi(\bm{D},\bm{S})|D_1=D\right]\right]\\
        &=\mathbb{E}_{D_1}\left[\mathbb{E}_{\bm{D}_{2:n_d},\bm{S}}\left[\Phi(\bm{D}+\bm{e}_1,\bm{S})|D_1\right]\right]-\mathbb{E}_{D_1}\left[\mathbb{E}_{\bm{D}_{2:n_d},\bm{S}}\left[\Phi(\bm{D},\bm{S})|D_1\right]\right]\\
        &=\mathbb{E}\left[\Phi(\bm{D}+\bm{e}_1,\bm{S})-\Phi(\bm{D},\bm{S})\right]\\
        &=\mathbb{E}\left[A_1^+\right].
    \end{align*}
\end{proof}

\begin{lemma}
    \label{lem:poisson-multiply-rate}
    Let $D$ be a Poisson-distributed random variable with parameter $\lambda$, and let the function $f(\cdot)$ verify $\mathbb{E}[f(D)]<\infty$. Then we can write
    \[
    \mathbb{E}[Df(D)]=\lambda\mathbb{E}\left[f(D+1)\right].
    \]
\end{lemma}

\begin{proof}[Proof of Lemma~\ref{lem:poisson-multiply-rate}]
    We can write out the right-hand side explicitly using the Poisson mass function:
    \begin{align*}
        \mathbb{E}\left[Df(D)\right] & = \sum_{d=0}^\infty df(d)\frac{e^{-\lambda}\lambda^d}{d!}\\
        &=\sum_{d=1}^\infty df(d)\frac{e^{-\lambda}\lambda^d}{d!}\\
        &=\lambda\sum_{d=1}^\infty f(d)\frac{e^{-\lambda}\lambda^{d-1}}{(d-1)!}.
    \end{align*}
    We can set $d'=d-1$ (equivalently $d=d'+1$) and obtain
    \[
        \mathbb{E}\left[Df(D)\right] = \lambda \sum_{d'=0}^\infty f(d'+1)\frac{e^{-\lambda}\lambda^{d'}}{d'!} = \lambda\mathbb{E}\left[f(D+1)\right].
    \]
\end{proof}

We are now ready to prove Theorem~\ref{thm:shadow-price-derivative}.

\begin{proof}[Proof of Theorem~\ref{thm:shadow-price-derivative}]
    \textbf{Step 1:} Explicit form of the derivative (right-hand side). Using the chain rule, we can write
    \[
        \Psi'_\tau(\eta) = \frac{1}{\tau}\frac{d}{d\eta} (\bm{\lambda}+\eta\bm{\beta})\cdot \nabla_{\bm{\lambda}+\eta\bm{\beta}}\mathbb{E}\left[\Phi\left(\bm{D}^{\tau,\bm{\lambda}+\eta\bm{\beta}},\bm{S}^{\tau,\bm{\pi}}\right)\right],
    \]
    where both terms on the right-hand side are vectors of length $n_d$. We can then apply Lemma~\ref{lem:intermediate-derivative} to obtain
    \[
        \Psi'_\tau(\eta) = \frac{1}{\tau}\bm{\beta}\cdot \mathbb{E}\left[\bm{A}^+\right].
    \]

    \textbf{Step 2:} Explicit form of the shadow price estimator in expectation (left-hand side). Recall that the shadow price estimator is given by:
    \[
    \hat{\Delta}^\tau_{\text{SP}}=\frac{1}{\tau}\bm{A}^{\tau,\experiment,-}\cdot \left(\frac{1}{\rho}\bm{D}^{\tau, \rho(\bm{\lambda}+\bm{\beta})}-\frac{1}{1-\rho}\bm{D}^{\tau,(1-\rho)\bm{\lambda}}\right).
    \]
    To simplify notation, we denote the conditional expectation on supply of any random variable $\bm{Z}$ by $\mathbb{E}_{|\bm{S}}\left[\bm{Z}\right]=\mathbb{E}[\bm{Z}|\bm{S}^{\tau,\bm{\pi}}]$. We further notice that after conditioning on supply, the shadow price only depends on total demand, and therefore simplify notation to $A_i^{\tau,\experiment,-}=A_i^{-}(\bm{D}^{\tau,\bm{\lambda}+\rho\bm{\beta}})$.
    Taking the conditional expectation on supply yields
    \begin{align}
        \mathbb{E}_{|\bm{S}}\left[\hat{\Delta}_{\text{SP}}^\tau\right] &= \frac{1}{\tau}\mathbb{E}_{|\bm{S}}\left[\sum_{i=1}^{n_d}A_i^-(\bm{D}^{\tau,\bm{\lambda}+\rho\bm{\beta}})\left(\frac{1}{\rho}D_i^{\tau,\rho(\lambda_i+\beta_i)}-\frac{1}{1-\rho}D_i^{\tau,(1-\rho)\lambda_i}\right)\right]\nonumber\\
        &=\frac{1}{\tau}\sum_{i=1}^{n_d}\mathbb{E}_{|\bm{S}}\left[A_i^-(\bm{D}^{\tau,\bm{\lambda}+\rho\bm{\beta}})\left(\frac{1}{\rho}D_i^{\tau,\rho(\lambda_i+\beta_i)}-\frac{1}{1-\rho}D_i^{\tau,(1-\rho)\lambda_i}\right)\right]\nonumber\\
        &=\frac{1}{\tau}\sum_{i=1}^{n_d}\mathbb{E}_{|\bm{S}}\left[\mathbb{E}_{|\bm{S}}\left[A_i^-(\bm{D}^{\tau,\bm{\lambda}+\rho\bm{\beta}})\left(\frac{1}{\rho}D_i^{\tau,\rho(\lambda_i+\beta_i)}-\frac{1}{1-\rho}D_i^{\tau,(1-\rho)\lambda_i}\right)\bigg|\bm{D}^{\tau,\bm{\lambda}+\rho\bm{\beta}}\right]\right]\nonumber\\
        &=\frac{1}{\tau}\sum_{i=1}^{n_d}\mathbb{E}_{|\bm{S}}\left[\mathbb{E}_{|\bm{S}}\left[\left(\frac{1}{\rho}D_i^{\tau,\rho(\lambda_i+\beta_i)}-\frac{1}{1-\rho}D_i^{\tau,(1-\rho)\lambda_i}\right)\bigg|\bm{D}^{\tau,\bm{\lambda}+\rho\bm{\beta}}\right]A_i^-(\bm{D}^{\tau,\bm{\lambda}+\rho\bm{\beta}})\right]\label{eq:thm7-intermediate},
    \end{align}
    where the last two steps follow from the law of total expectation. We now recall that the total demand is the sum of the control and treatment demand, i.e., $\bm{D}^{\tau,\bm{\lambda}+\rho\bm{\beta}}=\bm{D}^{\tau,\rho(\bm{\lambda}+\bm{\beta})}+\bm{D}^{\tau,(1-\rho)\bm{\lambda}}$. As a result, we can write
    \begin{align}
            \mathbb{E}_{|\bm{S}}\left[D_i^{\tau,\rho(\lambda_i+\beta_i)}\bigg|\bm{D}^{\tau,\bm{\lambda}+\rho\bm{\beta}}\right]&=D_i^{\tau,\lambda_i+\rho\beta_i}\frac{\rho(\lambda_i+\beta_i)}{\lambda_i+\rho\beta_i}, \text{ and}\nonumber\\
            \mathbb{E}_{|\bm{S}}\left[D_i^{\tau,(1-\rho)\lambda_i}\bigg|\bm{D}^{\tau,\bm{\lambda}+\rho\bm{\beta}}\right]&=D_i^{\tau,\lambda_i+\rho\beta_i}\frac{(1-\rho)\lambda_i}{\lambda_i+\rho\beta_i}.\label{eq:shadow-price-expectation-intermediate}
    \end{align}
    Substituting the two expressions above into~\eqref{eq:thm7-intermediate}, we obtain:
    \begin{align*}
        \mathbb{E}_{|\bm{S}}\left[\hat{\Delta}_{\text{SP}}^\tau\right] &= \frac{1}{\tau}\sum_{i=1}^{n_d}\mathbb{E}_{|\bm{S}}\left[\left(\frac{\lambda_i+\beta_i}{\lambda_i+\rho\beta_i}-\frac{\lambda_i}{\lambda_i+\rho\beta_i}\right)D_i^{\tau,\lambda_i+\rho\beta_i}A_i^-(\bm{D}^{\tau,\bm{\lambda}+\rho\bm{\beta}})\right]\\
        &=\frac{1}{\tau}\sum_{i=1}^{n_d}\frac{\beta_i}{\lambda_i+\rho\beta_i}\mathbb{E}_{|\bm{S}}\left[D_i^{\tau,\lambda_i+\rho\beta_i}A_i^-(\bm{D}^{\tau,\bm{\lambda}+\rho\bm{\beta}})\right].
    \end{align*}
    Fix $i=1$. Using the law of total expectation, we can expand the expectation on the right-hand side to
    \begin{align*}
        \mathbb{E}_{|\bm{S}}\left[D_1^{\tau,\lambda_1+\rho\beta_1}A_1^-(\bm{D}^{\tau,\bm{\lambda}+\rho\bm{\beta}})\right] &= \mathbb{E}_{|\bm{S}}\left[\mathbb{E}_{|\bm{S}}\left[D_1^{\tau,\lambda_1+\rho\beta_1}A_1^-(\bm{D}^{\tau,\bm{\lambda}+\rho\bm{\beta}})\Big| D_1^{\tau,\lambda_1+\rho\beta_1}\right]\right]\\
        &=\mathbb{E}_{|\bm{S}}\left[D_1^{\tau,\lambda_1+\rho\beta_1}\mathbb{E}_{|\bm{S}}\left[A_1^-(\bm{D}^{\tau,\bm{\lambda}+\rho\bm{\beta}})\Big| D_1^{\tau,\lambda_1+\rho\beta_1}\right]\right].
    \end{align*}
    Let $f(d)=\mathbb{E}_{|\bm{S}}\left[A_1^-(\bm{D}^{\tau,\bm{\lambda}+\rho\bm{\beta}})\Big| D_1^{\tau,\lambda_1+\rho\beta_1}=d\right]$. Now we apply Lemma~\ref{lem:poisson-multiply-rate} to obtain
    \begin{align*}
        \mathbb{E}_{|\bm{S}}\left[D_1^{\tau,\lambda_1+\rho\beta_1}A_1^-(\bm{D}^{\tau,\bm{\lambda}+\rho\bm{\beta}})\right] &= \mathbb{E}_{|\bm{S}}\left[D_1^{\tau,\lambda_1+\rho\beta_1}f\left(D_1^{\tau,\lambda_1+\rho\beta_1}\right)\right]\\
        &= (\lambda_1+\rho\beta_1)\tau\,\mathbb{E}_{|\bm{S}}\left[f(D_1^{\tau,\lambda_1+\rho\beta_1}+1)\right]\\
        &= (\lambda_1+\rho\beta_1)\tau\,\mathbb{E}_{|\bm{S}}\left[A_1^-(\bm{D}^{\tau,\bm{\lambda}+\rho\bm{\beta}}+\bm{e}_1)\right]\\
        &=(\lambda_1+\rho\beta_1)\tau\,\mathbb{E}_{|\bm{S}}\left[\Phi_\tau(\bm{D}^{\tau,\bm{\lambda}+\rho\bm{\beta}}+\bm{e}_1,\bm{S}^{\tau,\pi}) - \Phi_\tau(\bm{D}^{\tau,\bm{\lambda}+\rho\bm{\beta}},\bm{S}^{\tau,\pi})\right]\\
        &= (\lambda_1+\rho\beta_1)\tau\,\mathbb{E}_{|\bm{S}}\left[A_1^+(\bm{D}^{\tau,\bm{\lambda}+\rho\bm{\beta}})\right],
    \end{align*}
    where we use the definitions of the positive-demand and negative-demand subgradients in the last step. We can apply the derivation above for each $i\ge 1$, and substitute into~\eqref{eq:shadow-price-expectation-intermediate} to obtain
    \begin{align*}
        \mathbb{E}_{|\bm{S}}\left[\hat{\Delta}_{\text{SP}}^\tau\right] &= \frac{1}{\tau}\sum_{i=1}^{n_d}\frac{\beta_i}{\lambda_i+\rho\beta_i}\mathbb{E}_{|\bm{S}}\left[D_i^{\tau,\lambda_i+\rho\beta_i}A_i^-(\bm{D}^{\tau,\bm{\lambda}+\rho\bm{\beta}})\right]\\
        &= \frac{1}{\tau}\sum_{i=1}^{n_d}\frac{\beta_i}{\lambda_i + \rho\beta_i}(\lambda_i+\rho\beta_i)\tau\,\mathbb{E}_{|\bm{S}}\left[A_i^+(\bm{D}^{\tau,\bm{\lambda}+\rho\bm{\beta}})\right].
    \end{align*}
    Taking the expectation over both demand and supply, we obtain that
    \[
        \mathbb{E}\left[\hat{\Delta}_{\text{SP}}^\tau\right] = \frac{1}{\tau}\bm{\beta}\cdot \mathbb{E}\left[\bm{A}^+\right],
    \]
    which, together with the result from step 1, completes the proof.
\end{proof}

}

\subsubsection{{Asymmetric experiments}}

\begin{proof}[Proof of Theorem~\ref{thm:zero-bias-random}.]

We recall the definition of the shadow price estimator {in the fluid limit}
\[
\hat{\Delta}_{\text{SP}}=\bm{a}^\rho\cdot\bm\beta.
\]
Taking the expectation over $\rho$ uniformly random over $[0,1]$ yields
\begin{align*}
    \mathbb{E}_{\rho}\left[\hat{\Delta}_{\text{SP}}\right] &= \int_0^1\bm{a}^\rho\cdot\bm\beta d\rho\\
    &=\Psi(1)-\Psi(0) & \text{(by Proposition~\ref{prop:fundamental-theorem-calculus})}
    &=\Delta.
\end{align*}
{Similarly, in the finite-sample case, we can write:
\[
    \mathbb{E}_{\bm{D},\bm{S}}\left[\mathbb{E}_{\rho}\left[\hat{\Delta}^\tau_{\text{SP}}\right]\right]=\int_0^1\Psi_\tau'(\rho)d\rho=\Psi_\tau(1)-\Psi_\tau(0)=\Delta^\tau.
\]
Note that we can change the order of integration by Fubini's theorem.}
\end{proof}

{
\begin{proof}[Proof of Theorem~\ref{thm:shadow-price-plus}]
    We fix $\bm{\beta}\le \bm{0}$ and $\rho < 0.5$. We first recall that for any demand type $i$, $a_i\le \bar{v}_i^*$, therefore
    \begin{equation}
    \label{eq:SP-RCT-negative}
    \hat{\Delta}_{\text{SP}} = \sum_{i=1}^{n_d}a_i\beta_i \ge \sum_{i=1}^{n_d}\bar{v}^*_i\beta_i = \hat{\Delta}_{\text{RCT}}.
    \end{equation}
    Using the definition of the SP+ estimator, we can therefore write
    \[
    \hat{\Delta}_{\text{SP+}}=(1-2\rho)\hat{\Delta}_{\text{RCT}}+2\rho\hat{\Delta}_{\text{SP}} \ge (1-2\rho)\hat{\Delta}_{\text{RCT}}+2\rho\hat{\Delta}_{\text{RCT}} = \hat{\Delta}_{\text{RCT}}.
    \]
    This implies that
    \[
        \Delta - \hat{\Delta}_{\text{RCT}} \ge \Delta - \hat{\Delta}_{\text{SP+}}.
    \]
    From Theorem~\ref{thm:rct-overestimate}, we know that $\Delta - \hat{\Delta}_{\text{RCT}}\ge 0$. To complete the proof, we simply need to show that $\Delta - \hat{\Delta}_{\text{RCT}}\ge\hat{\Delta}_{\text{SP+}}-\Delta$. We can write their difference as

    \begin{align*}
        \Delta - \hat{\Delta}_{\text{RCT}}-\hat{\Delta}_{\text{SP+}}+\Delta &= 2\Delta - \hat{\Delta}_{\text{RCT}}-(1-2\rho)\hat{\Delta}_{\text{RCT}} - 2\rho\hat{\Delta}_{\text{SP}}\\
        &= 2\Delta - 2(1-\rho)\hat{\Delta}_{\text{RCT}} - 2\rho\hat{\Delta}_{\text{SP}}\\
        &= 2\int_0^1\bm{a}^\eta\cdot\bm{\beta}d\eta - 2\int_\rho^1\bm{\bar{v}^*}\cdot\bm{\beta}d\eta-2\int_{0}^\rho \bm{a}^\rho\cdot\bm{\beta}d\eta\\
        &= 2\underbrace{\int_0^\rho (\bm{a}^\eta-\bm{a}^\rho)\cdot\bm{\beta}d\eta}_{T_1} + 2\underbrace{\int_\rho^1(\bm{a}^\eta-\bm{\bar{v}^*})\cdot\bm{\beta}d\eta}_{T_2}.
    \end{align*}
    We can separetely prove each term is non-negative.
    \begin{itemize}
        \item First term: by concavity of the partial-treatment value function, we know that $\bm{a}^\eta\cdot\bm{\beta}\ge \bm{a}^\rho\cdot\bm{\beta}$ for all $\eta \le \rho$. Therefore $T_1\ge 0$.
        \item Second term: the proof of Theorem~\ref{thm:rct-overestimate} established the following statements for $\bm{\beta}\le \bm{0}$:
        \begin{align*}
            \bm{\bar{v}^*}\cdot(\bm\lambda + \bm\beta) &\le \Psi(1)\\
            \bm{\bar{v}^*}\cdot(\bm\lambda + \rho\bm\beta) &= \Psi(\rho).
        \end{align*}
        Subtracting the two above equations yields
        \begin{align*}
        (1-\rho)\bm{\bar{v}^*}\cdot\bm\beta \le \Psi(1) - \Psi(\rho)&\Leftrightarrow \int_\rho^1\bm{\bar{v}^*}\cdot\bm\beta d\eta\le \int_\rho^1\bm{a}^\eta\cdot\bm\beta d\eta\\
        &\Leftrightarrow \int_\rho^1(\bm{a}^\eta-\bm{\bar{v}^*})\cdot\bm{\beta}d\eta \ge 0\\
        &\Leftrightarrow T_2 \ge 0.
        \end{align*}
    \end{itemize}
    
\end{proof}
}

\begin{proof}[Proof of Theorem~\ref{thm:conditions-bias-reduction}.]
    Assume without loss of generality that $\bm{\beta}\ge 0$. By Proposition~\ref{prop:concave-integrable}, we recall that for any $0\le \eta \le 1$,
    \[
    \bm{a}^0\cdot\bm\beta \ge \bm{a}^\rho\cdot\bm\beta \ge \bm{a}^1\cdot\bm\beta.
    \]
    By definition,
    \begin{align*}
        \hat{\Delta}_{\text{SP}} - \Delta &= \int_0^1 (\bm{a}^{\rho} - \bm{a}^{\eta})\cdot\bm{\beta}d\eta,
    \end{align*}
    which we can bound as:
    \[
    (\bm{a}^{\rho} - \bm{a}^0)\cdot\bm\beta \le \hat{\Delta}_{\text{SP}} - \Delta \le (\bm{a}^{\rho} - \bm{a}^1)\cdot\bm\beta,
    \]
    and in absolute value
    \[
    \text{Bias}_{\text{SP}} = \abs{\hat{\Delta}_{\text{SP}} - \Delta} \le (\bm{a}^{0} - \bm{a}^1)\cdot\bm\beta.
    \]
    This proves the second result.
    
    Similarly, recalling Theorem~\ref{thm:rct-overestimate}, we can write the RCT bias as
    \[
        \text{Bias}_{\text{RCT}} = \hat{\Delta}_{\text{RCT}} - \Delta = \int_0^1 (\bar{\bm{v}}^* - \bm{a}^\eta)d\eta \ge (\bar{\bm{v}}^* - \bm{a}^0)\cdot\bm\beta.
    \]
    The last two inequalities complete the proof of the first result.
\end{proof}

\section{Secondary metrics: proofs} \label{sec:secondary-proofs}

\begin{proof}[Proof of Proposition~\ref{prop:secondary-cs}.]
\textbf{Step 1:} We start with $\bm{x^*}$, the unique and nondegenerate optimal primal solution to $\Phi_\tau(\bm{d},\bm{s})$, and write down the complementary slackness conditions on the optimal dual variables:
{
\begin{align*}
    b_j^* + a_i^* + \xi_{u_j^s,u_i^d}^* &= v_{u_j^s,u_i^d} && \text{when}~x^*_{u_j^s,u_i^d}>0,~(u_j^s,u_i^d)\in\mathcal{A}_{s\to d},\\
    b_j^*-m_w^* + \xi^*_{u_j^s,w}&= v_{u_j^s,w} &&\text{when}~x^*_{u_j^s,w}>0,~ (u_j^s,w)\in\mathcal{A}_{s\to *},\\
    m_w^*+a_i^* + \xi^*_{w,u_i^d}&= v_{w,u_i^d} && \text{when}~x^*_{w,u_i^d}>0,~(w, u_i^d)\in\mathcal{A}_{*\to d},\\
    m_u^* - m_w^* + \xi^*_{u,w} &= v_{u,w} && \text{when}~x^*_{u,w}>0, (u,w)\in\mathcal{A}_{*},\\
    a^*_{i} & =0 && \text{when}~\sum_{w:(w,u)\in\mathcal{A}}x^*_{w,u} < d_i,~u=u_i^d\in\mathcal{U}_d,\\
    b^*_{j} & =0 && \text{when}~\sum_{w:(u,w)\in\mathcal{A}}x^*_{u,w} < s_j,~ u=u_j^s\in\mathcal{U}_s,\\
    \xi^*_{u,w} & = 0 && \text{when}~x^*_{u,w} < \tau k_{u,w},~(u,w)\in\mathcal{A}.
\end{align*}}
Because the primal solution is nondegenerate, the above system of linear equations consists of {$|\mathcal{U}|+|\mathcal{A}|$} equations which uniquely determine the optimal dual solution. We can write this system in matrix form as 
{
\[
\bm{M}\begin{pmatrix}\bm{a^*}\\\bm{b^*}\\\bm{m^*}\\\bm{\xi^*}\end{pmatrix} = \bm{v}'.
\]
Each row of the matrix $\bm{M}$ is a sparse vector, and all nonzero entries are equal to 1 or -1. The first $|\mathcal{A}|$ rows have exactly three nonzeros, while the last $|\mathcal{U}|$ rows are $(|\mathcal{U}|+|\mathcal{A}|)$-dimensional unit vectors. Meanwhile, the vector $\bm{v'}$ consists of two parts: first, the sequence of $v_{u,w}$ corresponding to $x_{u,w}^*>0$, then a sequence of zeros.
}

Let $\varepsilon > 0$. Consider an alternative matching problem $\Phi_\tau^{\varepsilon}(\bm{d},\bm{s})$ in which the edge weights are given by {$v_{u,w}+\varepsilon \tilde{v}_{u,w}$}. This is a small perturbation of the original weights of the matching problem: if we take $\varepsilon$ small enough, the optimal solution $\bm{x^*}$ in $\Phi_\tau(\bm{d}, \bm{s})$ remains optimal in $\Phi_\tau^\varepsilon(\bm{d}, \bm{s})$. As a result, we can compute the dual variables for $\Phi^\varepsilon(\bm{d}, \bm{s})$ directly from $\bm{x^*}$ using {the same} complementary slackness {system:
\[
\bm{M}\begin{pmatrix}\bm{a}^\varepsilon\\\bm{b}^\varepsilon\\\bm{m}^\varepsilon\\\bm{\xi}^\varepsilon\end{pmatrix} = {\bm{v'}} + \varepsilon{\bm{\tilde{v}'}},
\]}
where the matrix $\bm{M}$ is constructed exactly as before, while the vector $\bm{\tilde{v}'}$ is constructed from the weights $\tilde{v}_{u,w}$ exactly as the vector $\bm{v'}$ was constructed from the weights $v_{u,w}$.

\textbf{Step 2:} By definition of $\Phi_\tau^{\bm{\tilde{v}}}(\cdot,\cdot)$, we can write
{
\[
\Phi_\tau^\varepsilon(\bm{d},\bm{s})=\sum_{(u,w)\in\mathcal{A}}(v_{u,w}+\varepsilon \tilde{v}_{u,w})x_{u,w}^*=\Phi_\tau(\bm{d},\bm{s}) + \varepsilon\Phi^{\bm{\tilde{v}}}_\tau(\bm{d},\bm{s}).
\]}
We obtain a new expression for the marginal values {$a_i^{\bm{\tilde{v}}}$:
\begin{align*}
\varepsilon a_i^{\bm{\tilde{v}}} & = \varepsilon\Phi^{\bm{\tilde{v}}}(\bm{d}+\bm{e}_i,\bm{s})-\varepsilon\Phi^{\bm{\tilde{v}}}(\bm{d},\bm{s})\\
&= \left(\Phi^\varepsilon(\bm{d}+\bm{e}_i,\bm{s})-\Phi(\bm{d}+\bm{e}_i,\bm{s})\right)-\left(\Phi^\varepsilon(\bm{d},\bm{s})-\Phi(\bm{d},\bm{s})\right)\\
&=\left(\Phi^\varepsilon(\bm{d}+\bm{e}_i,\bm{s})-\Phi^\varepsilon(\bm{d},\bm{s})\right)-\left(\Phi(\bm{d}+\bm{e}_i,\bm{s})-\Phi(\bm{d},\bm{s})\right)\\
&=a_i^\varepsilon - a_i^*,
\end{align*}}
where the {penultimate} equality follows from {Eq.~\eqref{eq:downward-dual}.}

\textbf{Step 3:} We know from step 1 that we can obtain $a_i^*$ and $a_i^\varepsilon$ by solving two linear systems. In other words, we can write:
{
\begin{align*}
    a_i^*&=\left(\bm{M}^{-1}{\bm{v'}}\right)_i,\\
    a_i^\varepsilon &= \left(\bm{M}^{-1}({\bm{v'}}+\varepsilon{\bm{\tilde{v}'}})\right)_i.
\end{align*}

Applying these identities to the result from step 2, we obtain
\begin{align*}
     a_i^{\bm{\tilde{v}}} &= \frac{1}{\varepsilon}\left[\left(\bm{M}^{-1}({\bm{v'}}+\varepsilon{\bm{\tilde{v}'}})\right)_i-\left(\bm{M}^{-1}{\bm{v'}}\right)_i\right]\\
     &=\frac{1}{\varepsilon}\left(\bm{M}^{-1}{\bm{v'}}+\varepsilon\bm{M}^{-1}{\bm{\tilde{v}'}}-\bm{M}^{-1}{\bm{v'}}\right)_i\\
     &=\left(\bm{M}^{-1}{\bm{\tilde{v}'}}\right)_i.
\end{align*}
}
\end{proof}
\end{document}